\documentclass[final]{siamltex}
\usepackage{amsfonts,amsmath,amssymb,stmaryrd}
\usepackage{mathrsfs,mathtools}
\usepackage{enumerate}
\usepackage{xspace,MACROS/mydef} 
\usepackage{hyperref}
\usepackage{esint}
\usepackage{graphicx}

\usepackage{booktabs}
\usepackage{psfrag}
\usepackage{bbold}

\DeclareMathAlphabet{\mathpzc}{OT1}{pzc}{m}{it}

\usepackage[notcite,notref]{showkeys} 
\usepackage[usenames,dvipsnames]{color}
\newcommand{\EO}[1]{{\color{black}#1}}
\newcommand{\AJS}[1]{{\color{black}#1}}
\newcommand{\RR}[1]{{\color{black}#1}}
\newcommand{\AAF}[1]{{\color{black}#1}}

\usepackage{pgf,tikz}
\usetikzlibrary{arrows}

\usepackage{caption}
\usepackage{subcaption}

\newtheorem{remark}[theorem]{Remark}
\numberwithin{equation}{section}

\title{Adaptive finite element methods for an optimal control problem involving Dirac measures\thanks{AA is partially supported by USM through project 116.12.1. EO is supported by CONICYT through FONDECYT project 3160201 and Anillo ACT-1106. RR is supported by BASAL PFB03 CMM project, Universidad de Chile. AJS is partially supported by NSF grant DMS-1418784. Part of this work was carried out while EO was visiting the University of Tennessee through a special visitors program for the academic year 2015--2016.}}

\author{Alejandro Allendes\thanks{Departamento de Matem\'atica, Universidad T\'ecnica Federico Santa Mar\'ia, Valpara\'iso, Chile.
\texttt{alejandro.allendes@usm.cl}}
\and
Enrique Ot\'arola\thanks{Departamento de Matem\'atica, Universidad T\'ecnica Federico Santa Mar\'ia, Valpara\'iso, Chile.
\texttt{enrique.otarola@usm.cl}}
\and
Richard Rankin\thanks{Departamento de Matem\'atica, Universidad T\'ecnica Federico Santa Mar\'ia, Valpara\'iso, Chile.
\texttt{richard.rankin@usm.cl}}
\and
Abner J.~Salgado\thanks{Department of Mathematics, University of Tennessee, Knoxville, TN 37996, USA. \texttt{asalgad1@utk.edu}}
}

\date{Draft version of \today.}

\begin{document}

\maketitle

\begin{abstract}
The purpose of this work is the design and analysis of a reliable and efficient a posteriori error estimator for the so-called pointwise tracking optimal control problem. This linear-quadratic optimal control problem entails the minimization of a cost functional that involves point evaluations of the state, thus leading to an adjoint problem with Dirac measures on the right hand side; control constraints are also considered. The proposed error estimator relies on a posteriori error estimates in the maximum norm for the state and in Muckenhoupt weighted Sobolev spaces for the adjoint state. We present an analysis that is valid for two and three-dimensional domains. We conclude by presenting several numerical experiments which reveal the competitive performance of adaptive methods based on the devised error estimator.
\end{abstract}

\begin{keywords}
pointwise tracking optimal control problem, Dirac measures, a posteriori error analysis, adaptive finite elements, maximum norm, Muckenhoupt weights, weighted Sobolev spaces.
\end{keywords}

\begin{AMS}
49J20,    
49M25,    
65K10,    
65N15,    
65N30,    
65N50,    
65Y20.    
\end{AMS}

\section{Introduction}
\label{sec:introduccion}

The design of efficient techniques to approximate the solution of an optimal control problem is of paramount importance in science and engineering. When the optimal control problem is based on the minimization of a quadratic functional subject to a linear partial differential equation (PDE) and control/state constraints, several solution techniques have been proposed and analyzed in the literature. We refer to \cite{HPUU:09,HT:10,IK:08,JCarlos,Tbook} for an overview and an up-to-date discussion. A class of numerical methods that has proven useful for approximating the solution to such problems, and the ones we will use in this work, are so-called adaptive finite element methods (AFEMs).

Over the last three decades, the numerical approximation of the solution to a PDE based on AFEMs has become an important tool in modern scientific and engineering computation: it allows for the resolution of PDEs with relatively modest computational resources. An essential ingredient of AFEMs is an posteriori error estimator, which is a computable quantity that depends on the discrete solution and data, and provides information about the local quality of the approximate solution. Therefore, it can be used for adaptive mesh refinement and coarsening, error control and equidistribution of the computational effort. We refer to \cite{AObook,MR1770058,NSV:09,NV,Verfurth} for an up-to-date discussion of a posteriori error analysis for linear elliptic PDEs and the construction of AFEMs, their convergence and optimal complexity.

As opposed to the well-established theory for linear elliptic PDEs, the a posteriori error analysis for finite element approximations of a constrained optimal control problem is far from complete. The main source of difficulty is its inherent nonlinear feature. In fact, the optimality conditions that characterize the solution to a constrained linear-quadratic optimal control problem consist of a \emph{state equation}, an \emph{adjoint equation} and a \emph{variational inequality} \cite{IK:08,Tbook}. This heuristically implies that an AFEM driven by an a posteriori error indicator based only on the \emph{state equation} cannot be applied confidently, and also \RR{sets} the need for the development and analysis of a posteriori error estimators for optimal control problems; see \cite{LiuYan} for a discussion. 

In the context of a distributed optimal control problem, to the best of our knowledge, the first work that provides an advance concerning a posteriori error analysis is \cite{LiuYan}. In this work, the authors derive a residual-type a posteriori error estimator \cite[(3.10)--(3.11)]{LiuYan} and prove that it yields an upper bound for the error \cite[Theorem 3.1]{LiuYan}. However, the sharpness of such a bound was not analyzed. In the linear elliptic PDE case, it is well known that a residual-type a posteriori lower bound without pollution cannot be true in general: the so-called \emph{oscillation terms} appear in the relationship between error and estimator \cite{NSV:09,NV}. In addition, the oscillation might be dominant in \RR{the} early stages of an AFEM, thus it cannot be ignored to obtain optimality without assuming that the initial mesh is sufficiently fine \cite{MR1770058,NSV:09,NV}. In the context of optimal control problems, reference \cite{HHIK} continues and extends the prior work developed in \cite{LiuYan}. The authors propose a slight modification of the residual-type a posteriori error estimator of \cite{LiuYan} and prove upper and lower error bounds which include oscillation terms \cite[Theorems 5.1 and 6.1]{HHIK}. 

An attempt to unify these ideas has been carried out recently in \cite{KRS}. The authors derive an important error equivalence that simplifies the a posteriori error analysis to, simply put, provide estimators for the \emph{state} and \emph{adjoint} equations which satisfy a set of suitable assumptions \cite[Theorem 3.2]{KRS}. Unfortunately, this analysis relies fundamentally on a particular structure for the problem and the relations among the natural spaces for the state, adjoint state and control. Many problems do not fit into this framework and thus one must either extend the theory or devise new estimators. The problem we consider in this work is an instance of this issue. For different approaches based on \emph{weighted residual} and \emph{goal-oriented} methods and advances in the semilinear and nonlinear case, the reader is referred to \cite{BBMRV,BV:09,HH:08,MRW:15,VW:08}.

In this work we will be interested in the design and analysis of a reliable and efficient a posteriori error estimator for the so-called \emph{pointwise tracking optimal control problem}. To describe this problem, for $n \in \{2,3\}$, we let $\Omega \subset \R^n$ be an open polytopal domain with Lipschitz boundary and let $\calZ \subset \Omega$ with $\# \calZ < \infty$. Given a set of desired states $\{ \ysf_z \}_{z \in \calZ}$, a regularization parameter $\lambda>0$ and the cost functional
\begin{equation}
\label{eq:defofJp}
  J(\ysf,\usf) = \frac12 \sum_{z \in \calZ} | \ysf(z) - \ysf_z |^2 + \frac\lambda2 \| \usf \|_{L^2(\Omega)}^2,
\end{equation}
the problem reads as follows: Find $\min J(\ysf, \usf)$ subject to the linear \emph{state equation}
\begin{equation}
\label{eq:defofPDEp}
  -\LAP \ysf = \fsf + \usf \text{ in } \Omega, \qquad \ysf = 0 \text{ on } \partial\Omega,
\end{equation}
and the \emph{control constraints}
\begin{equation}
\label{eq:cc2}
\asf \leq \usf(x) \leq \bsf \RR{\mbox{ for almost every }} x \in \Omega.
\end{equation}
The bounds $\asf$, $\bsf \in \mathbb{R}$ satisfy the property $\asf < \bsf$, and the forcing term $\fsf$ belongs to $L^\infty(\Omega)$. The cost functional involves point evaluations of the state, which leads to a subtle formulation of the adjoint problem:
\begin{equation}
\label{eq:adjp}
-\LAP \psf = \sum_{z \in \calZ} ( \ysf(z) - \ysf_z) \delta_z \text{ in } \Omega, \qquad \psf = 0 \text{ on } \partial\Omega.
\end{equation}

As already pointed out in \cite{AOS2}, the pointwise tracking optimal control problem \eqref{eq:defofJp}--\eqref{eq:adjp} is relevant in several applications where the state observations are carried out at specific locations. For instance, the calibration problem with American options \cite{MR2137495}, selective cooling of steel \cite{MR1844451}, and many others.   
  
Notice that since, for $n>1$, $\delta_z \notin H^{-1}(\Omega)$, the solution $\psf$ to \eqref{eq:adjp} does not belong to $H^1(\Omega)$. Therefore, the analysis of the finite element method applied to the pointwise tracking optimal control problem is not standard. An a priori error analysis has been recently provided in \RR{\cite{AOS2, BrettElliott,MR3449612}}. In \cite{AOS2}, the authors operate under the framework of Muckenhoupt weighted Sobolev spaces analyzed in \cite{NOS2} and thus circumvent the difficulties associated with the adjoint equation \eqref{eq:adjp}. Indeed, weighted Sobolev spaces allow \RR{one} to work under a Hilbert space-based framework in comparison to the non-Hilbertian setting of \cite{BrettElliott,MR3449612}. In \cite{AOS2}, the authors propose a fully discrete scheme on quasi-uniform meshes that discretizes the control using piecewise constant functions. The state and adjoint are discretized using piecewise linear functions. For $n=2$, the authors obtain a $\calO(h|\log h|)$ rate of convergence for the optimal control in the $L^2$-norm. However, for $n=3$, the derived a priori estimate reads $\calO(h^{1/2}|\log h|^2)$, which is suboptimal in terms of approximation. This estimate motivates the study of a posteriori error estimators and adaptivity. AFEMs are also motivated by the fact that the a priori theory developed in \cite{AOS2} requires that $\Omega$ is convex. If this condition is violated the optimal variables may have geometric singularities which should be efficiently resolved. \AJS{For a convex domain and under certain geometric assumptions on the mesh (see \eqref{eq:rhnorig}), an a posteriori error estimator is provided in \cite{MR3449612} and its reliability is proven. No efficiency estimates are provided.}

\AJS{We comment that a somewhat similar problem is studied in \cite{MR2193509}. The authors study, for $n=2$ and $\Omega$ a convex polygon, a parameter identification problem with point observations. The cost functional reads as ours with $\lambda = 0$ and the parameter to be recovered ranges over an open subset of a finite dimensional space. They obtain, by carefully studying the behavior of associated discrete Green's functions, an a priori error estimate of order $\calO(h^2 |\log h|^2)$ for the approximation of the unknown parameter.}

The derivation and analysis of an a posteriori error estimator for the pointwise tracking optimal control problem is quite challenging. This is due to the fact that this problem involves:
\begin{enumerate}[($i$)]
 \item pointwise evaluations of the optimal state $\bar{\ysf}$ in the cost functional \eqref{eq:defofJp}, 
 \item an elliptic equation with point sources as the adjoint equation \eqref{eq:adjp}, and
 \item an intrinsic nonlinearity introduced by the constraints \eqref{eq:cc2} on the optimal control $\bar{\usf}$.
\end{enumerate}
Therefore, an a posteriori error estimator must incorporate all these features in order to drive an efficient AFEM. Given a mesh $\T$ and corresponding approximations $\bar{\ysf}_{\T}$, $\bar{\psf}_{\T}$ and $\bar{\usf}_{\T}$, our proposed error indicator $\E_{\textrm{ocp}}(\bar{\ysf}_{\T},\bar{\psf}_{\T},\bar{\usf}_{\T}; \T)$ is based on the following three contributions:
\begin{equation*}
 \RR{\E_{\textrm{ocp}}^2}(\bar{\ysf}_{\T},\bar{\psf}_{\T},\bar{\usf}_{\T}; \T) = \RR{\E_{\ysf}^2}(\bar{\ysf}_{\T},\bar{\usf}_{\T}; \T) + \RR{\E_{\psf}^2}(\bar{\psf}_{\T},\bar{\ysf}_{\T}; \T) + \RR{\E_{\usf}^2}(\bar{\usf}_{\T},\bar{\psf}_{\T}; \T),
\end{equation*}
where $\E_{\ysf}(\bar{\ysf}_{\T},\bar{\usf}_{\T}; \T)$ corresponds to the max-norm a posteriori error estimator analyzed in \cite{deldia,rhn} and extended in \cite{MR2249676,demlow2014maximum}, $\E_{\psf}(\bar{\psf}_{\T},\bar{\ysf}_{\T}; \T)$ denotes the residual-type a posteriori error indicator on Muckenhoupt weighted Sobolev spaces proposed and studied in \cite{AGM}, and $\E_{\usf}(\bar{\usf}_{\T},\bar{\psf}_{\T}; \T)$ is defined as the $\ell^2$-sum of the local contributions $\E_{\usf}(\bar{\usf}_{\T},\bar{\psf}_{\T}; T) = \|\bar{\usf}_{\T} - \Pi (-\tfrac{1}{\lambda} \bar{\psf}_{\T})\|_{L^2(T)}$, with $T \in \T$ and $\Pi(v) = \min \{\bsf, \max\{ \asf,v\}\}$. The main contribution of this work is the analysis of the a posteriori error estimator $\E_{\textrm{ocp}}=\E_{\textrm{ocp}}(\bar{\ysf}_{\T},\bar{\psf}_{\T},\bar{\usf}_{\T}; \T)$. \AJS{Assuming only that $\Omega$ is a Lipschitz polytope and $\fsf \in L^\infty(\Omega)$} we prove its global reliability and \EO{global efficiency. We prove local efficiency of the terms $\E_{\ysf}$ and $\E_{\psf}$. However,} we do not obtain a local efficiency bound because the term $\E_{\usf}$ is not locally efficient. This is a recurring feature in a posteriori error estimation for optimal control problems with control constraints and we refer the reader to \cite{KRS} for a thorough discussion on this matter. \AJS{Notice that we do not require convexity of $\Omega$.} We remark that the analysis involves estimates in $L^{\infty}$-norms and weighted Sobolev spaces, combined with having to deal with the variational inequality that characterizes the optimal control. This subtle intertwining of ideas is one of the highlights of this contribution.

We remark that our approach is not restricted to \eqref{eq:defofJp}--\eqref{eq:adjp} and can be applied to a wider class of problems. For instance, the so-called optimal control problem with point sources in the state equation \cite{AOS2} and an optimal control problem with finitely many pointwise state constraints \cite{DLeykekhman_DMeidner_BVexler_2013a}. The study of these will be part of our future work.

We organize our exposition as follows. We set notation in section~\ref{sec:notation}, where we also recall basic facts about weights and weighted spaces. The a priori and a posteriori error analysis of elliptic problems with delta sources is reviewed in section~\ref{sec:elliptic_problems}. Section~\ref{sec:maximum} recalls the maximum norm error estimation of elliptic problems. The core of our work is section~\ref{sec:pointwise_tracking}, where we describe our problem and its a priori error analysis and, combining the results of previous sections, we devise an a posteriori error estimator and show, in \S\ref{subsub:reliable} and \S\ref{subsub:efficient}, its reliability and efficiency, respectively. We conclude, in section~\ref{sec:numex}, with a series of numerical examples that illustrate and go beyond our theory.

\section{Notation and preliminaries}
\label{sec:notation}

Let us set notation and describe the setting we shall operate with.
\subsection{Notation}
\label{sub:notation}
Throughout this work $n \in \{2,3\}$ and $\Omega\subset\mathbb{R}^n$ is an open and bounded polytopal domain with Lipschitz boundary $\partial\Omega$. If $\Xcal$ and $\Ycal$ are normed vector spaces, we write $\Xcal \hookrightarrow \Ycal$ to denote that $\Xcal$ is continuously embedded in $\Ycal$. We denote by $\Xcal'$ and $\|\cdot\|_{\Xcal}$ the dual and the norm of $\Xcal$, respectively. 

The set of locally integrable functions on $\Omega$ is denoted by $L^1_{\textrm{loc}}(\Omega)$. For $E \subset \Omega$ of finite Hausdorff $i$-dimension, $i \in \{1,2,3\}$, we denote its measure by $|E|$. The mean value of a function $f$ over a set $E$ is
\[
 \fint_E f  = \frac{1}{|E|}\int_{E} f .
\]

The relation $a \lesssim b$ indicates that $a \leq C b$, with a constant $C$ that depends neither on $a$, $b$ nor the discretization parameters. The value of $C$ might change at each occurrence.

\subsection{Weighted Sobolev spaces}
\label{sub:weighted_Sobolev}
A weight is an almost everywhere positive function $\omega \in L^1_{\textrm{loc}}(\R^n)$. In particular, we will be interested in the weights belonging to the so-called Muckenhoupt class $A_2$ \cite{Javier,FKS:82,Muckenhoupt,Turesson}.

\begin{definition}[Muckenhoupt class $A_2$]
 \label{def:Muckenhoupt}
Let $\omega$ be a weight. We say that $\omega \in A_2$ if there exists a positive constant $C_{\omega}$ such that
\begin{equation}
\label{A_pclass}
C_{\omega} = \sup_{B} \left( \fint_{B} \omega \right) \left( \fint_{B} \omega^{-1} \right)  < \infty,
\end{equation}
where the supremum is taken over all balls $B$ in $\R^n$. If $\omega$ belongs to the Muckenhoupt class $A_2$, we say that $\omega$ is an $A_2$-weight, and we call the constant $C_{\omega}$ in \eqref{A_pclass} the $A_2$-constant of $\omega$. 
\end{definition}

\AJS{For a measurable $E \subset \R^n$ and a weight $\omega$, we set $\omega(E) = \int_E \omega$. With this notation \eqref{A_pclass} can be rewritten as
\[
  \omega(B) \omega^{-1}(B) \lesssim |B|^2,
\]
for all balls $B \subset \R^n$.
}

Let us present an example of a weight that belongs to $A_2$, which will be essential in the analysis presented below. Let $x_0$ be an interior point of $\Omega$. Denote by $\dist(x)$ the Euclidean distance $\dist(x)=|x-x_0|$ to $x_0$ and define $\dist^\alpha(x) = \dist(x)^{\alpha}$. We have that $\dist^\alpha \in A_2$ if and only if $\alpha \in (-n,n)$. The main motivation to consider the weight $\dist^\alpha$ is that it plays a central role in the analysis of Poisson problems with Dirac measures such as the adjoint equation \eqref{eq:adjp}; see section \ref{sec:elliptic_problems} and \cite{AGM,DAngelo:SINUM,NOS2}. We refer the reader to \cite{Javier,NOS2,Turesson} for more examples of $A_2$-weights and their most important properties.

We now define the weighted Lebesgue space $L^2(\omega,\Omega)$.

\begin{definition}[weighted Lebesgue spaces]
\label{Lp_weighted} 
Let $\omega \in A_2$, and let $\Omega \subset \R^n$ be an open and bounded domain. We define the weighted Lebesgue space $L^2(\omega, \Omega)$ as the set of measurable functions $u$ on $\Omega$ \AJS{for which the norm}
\begin{equation*}
 \| u \|_{L^2(\omega, \Omega)} = \left( \int_{\Omega} |u|^2 \omega \right)^{\frac{1}{2}}
\end{equation*}
\AJS{is finite.}
\end{definition}

Since $L^2(\omega,\Omega) \subset L^1_{\textrm{loc}}(\Omega)$ \cite[Proposition 2.3]{NOS2}, it makes sense to talk about weak derivatives of functions in $L^2(\omega,\Omega)$. We define weighted Sobolev spaces as follows.

\begin{definition}[weighted Sobolev spaces]
\label{H_1_weighted} 
Let $\omega \in A_2$, and let $\Omega \subset \R^n$ be an open and bounded domain. We define the weighted Sobolev space $H^1(\omega, \Omega)$ as the set of \AJS{functions $u \in W^{1,1}(\Omega)$ for which the norm}
\begin{equation*}
 \|u \|_{H^1(\omega, \Omega)} = \left( \|u\|_{L^2(\omega, \Omega)}^2 + \|\nabla u\|_{L^2(\omega, \Omega)}^2 \right)^{\frac{1}{2}}
\end{equation*}
\AJS{is finite}. We also define $H_0^1(\omega, \Omega)$ as the closure of $C_0^{\infty}(\Omega)$ in $H^1(\omega, \Omega)$.
\end{definition}

If $\omega \in A_2$, then we have the following important consequence: the space $H^1(\omega,\Omega)$ is Hilbert and $H^1(\omega,\Omega) \cap C^{\infty}(\Omega)$ is dense in $H^1(\omega,\Omega)$ (cf.~\cite[Proposition 2.1.2, Corollary 2.1.6]{Turesson} and \cite[Theorem~1]{GU}).

The class $A_2$ has proven to be a fundamental tool in harmonic analysis. If $\omega \in A_2$, then any Calder\'on-Zygmund singular integral operator is bounded in the space $L^2(\omega,\Omega)$ \cite{Javier,Grafakos}. In spite of this, the use of these classes as a tool to derive and understand properties of discrete schemes is relatively new in numerical analysis; see \cite{AGM,NOS2}.

\AJS{
\subsection{The Poisson problem in Lipschitz polytopes}
\label{sub:reg_of_y}

Let us, for the sake of future reference, collect here some standard results concerning the regularity of the solution to the Poisson problem
\begin{equation}
\label{eq:Fish}
  - \Delta u = f \ \text{in } \Omega, \quad u=0 \ \text{on } \partial\Omega,
\end{equation}
where $\Omega$ is a bounded and Lipschitz, but not necessarily convex, polytope. We begin with a global higher integrability of $u$ and, as a consequence, its H\"older regularity; see \cite{Dauge:92, Grisvard, JK:95, JK:81, MR2641539, Savare:98}.

\begin{proposition}[higher integrability]
\label{prop:uisW1p}
Let $u \in H^1_0(\Omega)$ denote the unique solution of \eqref{eq:Fish} with $f \in L^2(\Omega)$. There is $q>n$ such that $u \in W^{1,q}(\Omega)$. Moreover,
\[
  \| u \|_{W^{1,q}(\Omega)} \lesssim \| f \|_{L^2(\Omega)},
\]
where the hidden constant is independent of $u$ and $f$. This, in particular, implies that for $\kappa = 1-n/q>0$ we have $u \in C^{0,\kappa}(\bar\Omega)$ with a similar estimate.
\end{proposition}

Next we comment that, whenever $f \in L^p(\Omega)$ with $p \in [2,\infty)$, we have  a local regularity result, whose proof can be found, for instance, in  \cite[Theorem 9.11]{GT} or \cite[Theorem 12.2.2]{MR3012036}.

\begin{proposition}[local regularity]
\label{prop:loclreg}
Let $u \in H^1_0(\Omega)$ denote the unique solution of \eqref{eq:Fish} with $f \in L^p(\Omega)$ and $p \in [2,\infty)$. If $D \Subset \Omega$ then $u \in W^{2,p}(D)$ and the following estimate holds
\[
  \| u \|_{W^{2,p}(D)} \lesssim \| u \|_{L^p(\Omega)} + \| f \|_{L^p(\Omega)},
\]
where the hidden constant depends on $\textup{dist}(D,\partial\Omega)$ but is independent of $u$ and $f$.
\end{proposition}

Notice that, since $\Omega$ is bounded, we can combine the estimates of Propositions~\ref{prop:uisW1p} and \ref{prop:loclreg} to obtain that, for every $D \Subset \Omega$,
\begin{equation}
\label{eq:W2p}
  \| u \|_{W^{2,p}(D)} \lesssim  \| f \|_{L^p(\Omega)},  
\end{equation}
where the hidden constant depends on $|\Omega|$ and $\textup{dist}(D,\partial\Omega)$ but is independent of $u$ and $f$.

Finally, we establish the weighted local integrability of $u$.

\begin{proposition}[weighted integrability]
\label{prop:uweight}
Let $u \in H^1_0(\Omega)$ denote the solution of \eqref{eq:Fish} with $f \in L^p(\Omega)$ and $p>n$. Let $y \in \Omega$, $r < \textup{dist}(y,\partial\Omega)$ and  $B$ denote the ball of radius $r$ and center $y$. If $\mu \in A_2$, then we have that $u \in H^1(\mu,B) $. Moreover, 
\[
  \| \nabla u \|_{L^2(\mu,B)} \lesssim \| f \|_{L^p(\Omega)},
\]
where the hidden constant depends on $\mu(B)$, $\textup{dist}(B,\partial\Omega)$, $r$ and $|\Omega|$, but is independent of $u$ and $f$.
\end{proposition}
\begin{proof}
The proof follows from the local regularity of Proposition~\ref{prop:loclreg} and an embedding result. Indeed, since $\textup{dist}(B,\partial\Omega)>0$, from \eqref{eq:W2p} we have
\[
  \| u \|_{W^{2,p}(B)} \lesssim  \| f \|_{L^p(\Omega)}.
\]
Moreover, since $p>n$, we have that $W^{1,p}(B) \hookrightarrow L^\infty(B)$ and
\[
  \| \nabla u \|_{L^\infty(B)} \lesssim \| u \|_{W^{2,p}(B)} \lesssim  \| f \|_{L^p(\Omega)}.
\]
With this estimate the $L^2(\mu,B)$-norm of $\nabla u$ can be bounded as
\[
  \int_B \mu |\nabla u |^2 \leq \mu(B) \| \nabla u \|_{L^\infty(B)}^2 \lesssim \| f \|_{L^p(\Omega)}^2,
\]
with a hidden constant that depends only on $\mu(B)$, $\textup{dist}(B,\partial\Omega)$, $r$ and $|\Omega|$. This concludes the proof.
\end{proof}
}

\section{Elliptic problems with Dirac sources}
\label{sec:elliptic_problems}
Since the analysis of the pointwise tracking optimal control problem involves \eqref{eq:adjp}, in this section we review the arguments developed in \cite{AGM,NOS2} to study Poisson problems with Dirac measures on the right hand side. The analysis hinges on the Muckenhoupt weighted Sobolev spaces introduced in section \ref{sub:weighted_Sobolev}. We also comment on the finite element approximation of such problems and conclude by reviewing the a posteriori error analysis recently developed in \cite{AGM}.

Let $x_0$ be an interior point of $\Omega$. Consider the following elliptic boundary value problem:
\begin{equation}
\label{-lap=delta}
    - \Delta u = \delta_{x_0} \text{ in } \Omega, \qquad u = 0 \text{ on } \partial\Omega,
\end{equation}
where $\delta_{x_0}$ denotes the Dirac delta supported at $x_0 \in \Omega$. The asymptotic behavior of $u$ near $x_0$ is dictated by
\begin{equation}\label{asympt-x0}
\nabla u(x) \approx |x-x_0|^{1-n}.
\end{equation}
On the basis of \eqref{asympt-x0}, a simple computation shows that $| \nabla u | \in L^2(\dist^{\alpha},\Omega)$ provided $\alpha \in (n-2,\infty)$. This heuristic suggests that we \RR{seek solutions} to problem \eqref{-lap=delta} in weighted Sobolev spaces. 

Let us make these considerations rigorous. Since $\dist^{\alpha}, \dist^{-\alpha} \in A_2$ for $\alpha \in (-n,n)$, we invoke 
Definition \ref{H_1_weighted} and \cite[Proposition 2.1.2]{Turesson} to conclude that the spaces $H^1_0(\dist^{\alpha},\Omega)$ and $H^1_0(\dist^{-\alpha},\Omega)$ are Hilbert. We then define the bilinear form
\begin{equation}
\label{eq:defofforma}
  a(w,v) = \int_\Omega \nabla w\cdot \nabla v
\end{equation}
and consider the following weak formulation of problem \eqref{-lap=delta}:
\begin{equation}
\label{eq:weak_form}
u \in H^1_0(\dist^{\alpha},\Omega): \quad a(u,v) = \delta_{x_0}(v) \quad \forall v \in H^1_0(\dist^{-\alpha},\Omega).
\end{equation}
The bilinear form $a$ satisfies an inf-sup condition on $H^1_0(\dist^{\alpha},\Omega) \times H^1_0(\dist^{-\alpha},\Omega)$ \cite[Theorem 2.3]{AGM}. In addition, $a$ is bounded 
as a consequence of H{\"o}lder's inequality. On the other hand, if $\alpha \in (n-2,n)$ we have that $\delta_{x_0} \in H_0^1(\dist^{-\alpha},\Omega)'$ \cite[Lemma 7.1.3]{KMR}. All these elements allow us to conclude that problem \eqref{eq:weak_form} is well posed on the weighted Sobolev spaces $H_0^1(\dist^{\alpha},\Omega)$ and $H_0^1(\dist^{-\alpha},\Omega)$ provided 
\begin{equation}
\label{alpha}
\alpha \in \textbf{I} := \left(n-2,n \right). 
\end{equation}
We refer the reader to \cite{NOS} for an alternative formulation on the weighted Sobolev spaces $H_0^1(\varpi,\Omega)$ and $H_0^1(\varpi^{-1},\Omega)$, where the weight $\varpi$ is defined as follows: if $d = \diam(\Omega)$ is the diameter of $\Omega$ and $\distnos(x)=\dist(x)/(2d)$, then
\begin{equation}
\label{eq:defofvarpi}
  \varpi(x) = \begin{dcases}
                \frac{ \distnos(x)^{n-2}} { \log^2 \distnos(x) }, & 0 < \distnos(x) < \frac12, \\
                \frac{2^{2-n}}{\log^2 2}, & \distnos(x) \geq \frac12.
              \end{dcases}
\end{equation}
The well-posedness of \eqref{-lap=delta} follows from \cite[Lemma 7.7]{NOS2}.

\AJS{We conclude with an embedding result for $H^1_0(\dist^\alpha,\Omega)$ that will be useful later.

\begin{lemma}[$H^1_0(\dist^\alpha,\Omega) \hookrightarrow L^2(\Omega)$]
\label{lem:restalpha}
If $\alpha \in (n-2,\RR{2})$ then $H^1_0(\dist^\alpha,\Omega) \hookrightarrow L^2(\Omega)$ and we have the following weighted Poincar\'e inequality
\[
  \| v \|_{L^2(\Omega)} \lesssim \| \nabla v \|_{L^2(\dist^\alpha,\Omega)}, \quad \forall v \in H^1_0(\dist^\alpha,\Omega)\RR{,}
\]
where the hidden constant depends only on $\Omega$.
\end{lemma}
\begin{proof}
It suffices to verify condition (6.2) of \cite{NOS2} for $p=q=2$, $\omega = \dist^\alpha$ and $\rho = 1$, \ie if $B_s \subset \Omega$ denotes a ball of radius $s$
\[
  \left( \frac{r}R \right)^{2+n} \frac{ \dist^\alpha(B_R) }{\dist^\alpha(B_r)} \lesssim 1,
\]
whenever $r \leq R$. Similarly to \cite[Lemma 7.6]{NOS2}, it suffices to verify this condition for balls centered at $x_0$. In this case
\[
  \dist^\alpha(B_s) \approx \int_0^s r^{\alpha+n-1} \diff r \approx s^{n+\alpha},
\]
which shows that, whenever $\alpha <2$, the embedding and Poincar\'e inequality hold.
\end{proof}}

\subsection{Finite element approximation and a priori error bounds}
\label{sec:a_priori}
Let $\T = \{T\}$ be a conforming partition of $\bar\Omega$ into simplices $T$ with size $h_T = \diam(T)$, and set $h_{\T} = \max_{T \in \T} h_T$.  We denote by $\Tr$ the collection of conforming and shape regular meshes that are refinements of an initial mesh $\T_0$. By shape regular we mean that there exists a constant $\sigma > 1$ such that
$
 \max \left\{ \sigma_T : T \in \T \right\} \leq \sigma
$
for all $\T \in \Tr$ \cite{CiarletBook,Guermond-Ern}. Here $\sigma_T = h_T/\rho_T$ denotes the shape coefficient of $T$  where $\rho_T$ is the diameter of the largest ball that can be inscribed in $T$.

Let $\Sides$ denote the set of internal interelement boundaries $S$ (or sides) and by $h_S$ we indicate the diameter of $S$. We define the \emph{star} or \emph{patch} associated with an element $T \in \T$ as
\[
  \Ne_T := \bigcup_{T' \in \T : T \cap T' \neq \emptyset} T'.
\]
Given a mesh $\T \in \Tr$, we define the finite element space of continuous piecewise polynomials of degree one as
\begin{equation}
\V(\T) = \left\{v_{\T} \in C^0( \bar \Omega): {v_{\T}}_{|T} \in \mathbb{P}_1(T),~ \ \forall~ T \in \T, \ v_{\T|\partial\Omega} = 0 \right\}.
\label{eq:defFESpace}
\end{equation}
Then, the corresponding Galerkin approximation to problem \eqref{-lap=delta} is given by
\begin{equation}
\label{eq:Gal_sol}
u_{\T} \in \V(\T): \quad a(u_{\T},v_{\T}) = \delta_{x_0} (v_{\T}) \quad \forall v_{\T} \in \V(\T).
\end{equation}
The well-posedness of problem \eqref{eq:Gal_sol} follows from \cite[Theorem 3.1]{AGM} and \cite[Lemma 7.8]{NOS2}. The a priori error analysis is due to Scott \cite{Scott:73} and Casas \cite{Casas:85}. The author of \cite{Scott:73} assumes that $\Omega$ is a smooth domain and the mesh $\T \in \Tr$ is quasiuniform with mesh size $h_{\T}$ and derives the following a priori error estimate:
\begin{equation}
\label{eq:optimal_a_priori_error_bound}
\| u - u_{\T} \|_{L^2(\Omega)} \lesssim h_{\T}^{2-n/2}.
\end{equation}
Using a different technique, Casas \cite{Casas:85} obtained the same result for polygonal or polyhedral domains and general regular Borel measures on the right-hand side. An analysis based on Muckenhoupt weighted Sobolev spaces has been recently developed in \cite{NOS2}: If $\Omega$ is convex and the mesh $\T \in \Tr$ is quasiuniform with mesh size $h_{\T}$, then
\begin{equation}
\label{eq:a_priori_error_bound}
 \| u - u_{\T} \|_{L^2(\Omega)} \lesssim h_{\T}^{2-n/2}| \log h_{\T} | \| \nabla u \|_{L^2(\varpi,\Omega)}.
\end{equation}
where $\varpi$ is defined in \eqref{eq:defofvarpi}. The analysis of \cite{Casas:85,NOS2} uses that $\Omega$ is a convex polyhedral domain, thus closing the regularity gap of \cite{Scott:73}; see \cite[Remark 3.1]{Seidman:12}.

\subsection{A posteriori error estimates}
\label{sec:a_posteriori}

The a priori upper bounds \eqref{eq:optimal_a_priori_error_bound} and \eqref{eq:a_priori_error_bound} are not computable and essentially provide only asymptotic information.
In addition, the limited regularity of the solution $u$ of problem \eqref{-lap=delta} does not allow the method to exhibit optimal rates of convergence. These facts motivate the analysis of an a posteriori error estimator driving AFEMs to solve problem \eqref{-lap=delta}. These methods are essential for the efficient approximation of \eqref{-lap=delta} with relatively modest computational resources; especially in three spatial dimensions and in the scenario of a non-convex domain $\Omega$.

We now recall the residual-type a posteriori error estimator introduced and analyzed in \cite{AGM}. To do this, given $\T \in \Tr$ and $T \in \T$, we define
 \begin{equation}
 \label{def:DT}
  D_T = \max_{x \in T} \dist (x).
 \end{equation}
We define the local a posteriori error indicator by
\begin{equation}
\label{eq:estimator_local}
\E_{\alpha}^2(u_{\T};T) = 
  \begin{dcases}
    h_T D_T^{\alpha} \| \llbracket \nu \cdot \nabla u_{\T} \rrbracket \|^2_{L^2(\partial T \setminus \partial \Omega)} + h_T^{\alpha + 2 - n},
    & x_0 \in T, \\
    h_T D_T^{\alpha} \| \llbracket \nu \cdot \nabla u_{\T}  \rrbracket \|^2_{L^2(\partial T \setminus \partial \Omega),} & x_0 \notin T ,
  \end{dcases}
\end{equation}
and the global error estimator $\E_{\alpha}^2(u_{\T};\T)= \sum_{T \in \T} \E_{\alpha}^2(u_{\T};T)$. In \eqref{eq:estimator_local}, the jump or interelement residual $\llbracket \nu \cdot \nabla u_{\T} \rrbracket$ is defined by
\begin{equation}
\label{eq:jump}
 \llbracket \nu \cdot \nabla u_{\T} \rrbracket = \nu^+ \cdot \nabla {u_{\T}}_{|{T^+}} + \nu^- \cdot \nabla {u_{\T}}_{|T^-}
\end{equation}
on the internal side $S \in \Sides$ shared by the distinct elements $T^+$, $T^{-} \in \T$. Here $\nu^+, \nu^-$ are unit normals on $S$ pointing towards $T^+$, $T^{-}$, respectively.  We comment that, in \eqref{eq:estimator_local}, the factor $h_T^{\alpha + 2 - n}$ appears, on the basis of \cite[Theorem 4.7]{AGM}, as a consequence of a local estimation of the Dirac delta $\delta_{x_0}$  applied to stars containing individually the delta points. 

The following result states the reliability of the global a posteriori error estimator. For a proof see \cite[Theorem 5.1]{AGM}.
 
\begin{proposition}[global reliability]
\label{pro:global_alpha}
Let $u \in H_0^1(\dist^{\alpha},\Omega)$ and $u_{\T} \in \V(\T)$ be the solutions to problems \RR{\eqref{eq:weak_form}} and \eqref{eq:Gal_sol}, respectively. If $\alpha \in \textbf{\emph{I}}$, then
\[
 \| \nabla(u - u_{\T}) \|_{L^2(\dist^{\alpha},\Omega)} \lesssim \E_{\alpha}(u_{\T};\T),
\]
where the hidden constant depends on the diameter of $\Omega$, the shape regularity constant $\sigma$, the parameter $\alpha$ and the inf-sup constant of $a$. In addition, the hidden constant blows up when $\alpha$ approaches $\partial\mathbf{I}$.
\end{proposition}

The local efficiency of the \RR{indicator} \eqref{eq:estimator_local} is as follows \cite[Theorem 5.3]{AGM}.

\begin{proposition}[local efficiency] 
Let $u \in H_0^1(\dist^{\alpha},\Omega)$ and $u_{\T} \in \V(\T)$ be the solutions to problems \RR{\eqref{eq:weak_form}} and \eqref{eq:Gal_sol}, respectively. If $\alpha \in \textbf{\emph{I}}$, then
\[
 \E_{\alpha}(u_{\T};T) \lesssim \| u - u_{\T} \|_{H^1(\dist^{\alpha},\Ne_T)}, 
\]
for all $T \in \T$, and where the hidden constant depends on the shape regularity constant $\sigma$ and the parameter $\alpha$. In addition, the hidden constant blows up if $\alpha$ approaches $n$.
\end{proposition}

We remark that given the specifics of our problem and discretization scheme---we are dealing with the Laplacian, the right hand side is only a Dirac mass and we are using lowest order finite elements---no oscillation terms appear in the local lower bound. In more general situations these must be taken into consideration; see \cite{AGM} for details.

We conclude this section by commenting on the alternative a posteriori error analysis developed in \cite{Rodolfo1,Rodolfo2}, which measures the error in the spaces $W^{l,p}(\Omega)$ with $l\in \{0,1\}$.

\section{Pointwise a posteriori error estimates for elliptic problems}
\label{sec:maximum}

In section~\ref{sec:pointwise_tracking} we propose and analyze an a posteriori error estimator for the pointwise tracking optimal control problem. The proposed indicator hinges on a suitable combination of the error estimator described in section \ref{sec:a_posteriori} for the adjoint equation \eqref{eq:adjp} and a pointwise a posteriori error estimator for the state equation \eqref{eq:defofPDEp}. In an effort to make this contribution self contained, in this section we briefly review results concerning the a posteriori error analysis for elliptic problems in the maximum norm.

Let $f \in L^{\infty}(\Omega)$ and $u$ be the weak solution to:
\begin{equation}
\label{-lap=f}
u \in H_0^1(\Omega): \quad a(u,v) = (f,v)_{L^2(\Omega)} \quad \forall v \in H_0^1(\Omega),
\end{equation}
where $a$ is defined in \eqref{eq:defofforma}. \AJS{Proposition~\ref{prop:uisW1p} yields} the H{\"o}lder continuity of the function $u$ solving \eqref{-lap=f}; see also \cite[Lemma~1]{demlow2014maximum}. In the setting of subsection \ref{sec:a_priori}, we define the Galerkin approximation to problem \eqref{-lap=f} as
\begin{equation}
\label{eq:Gal_sol_f}
u_{\T} \in \V(\T): \quad a(u_{\T},v_{\T}) = (f,v_{\T})_{L^2(\Omega)} \quad \forall v_{\T} \in \V(\T).
\end{equation}

We now present the pointwise a posteriori error estimator studied by Nochetto in \cite{rhn} for $n=2$. The analysis of this error indicator was subsequently extended to $n=3$ by Dari et al.~in \cite{deldia} and later improved in \cite{MR3022214,demlow2014maximum}. We introduce the local pointwise indicator
\begin{equation}
\label{eq:estimator_local_inf}
\E_\infty(u_{\T};T) = 
h_T^2 \| f \|_{L^\infty(T)}  + h_T \| \llbracket \nu\cdot\nabla u_{\T} \rrbracket \|_{L^\infty(\partial T \setminus \partial \Omega)}, 
\end{equation}
and the corresponding global pointwise estimator $\E_\infty(u_{\T};\T) = \max_{T \in\T} \E_\infty(u_{\T};T)$.

The reliability of the global indicator $\E_{\infty}$ is given below. To state it and for future reference, we define
\begin{equation}\label{eq:log}
\ell_\T = \left|\log\left( \max_{T \in \T} \frac{1}{h_{T}}\right)\right|.
\end{equation}
The earliest proof of reliability can be found in \cite[Lemma 4.1]{rhn} and \cite[Theorem 3.1]{deldia} for $n=2$ and $n=3$, respectively. These results were later improved in \cite{MR3022214} and \cite{demlow2014maximum} to the one given below.

\begin{proposition}[global reliability]
\label{pro:global_infty}
Let $u \in H_0^1(\Omega) \cap L^\infty(\Omega)$ and $u_{\T} \in \V(\T)$ be the solutions to problems \eqref{-lap=f} and \eqref{eq:Gal_sol_f}, respectively. Then
\[
 \| u - u_{\T} \|_{L^{\infty}(\Omega)} \lesssim  \ell_{\T}\E_{\infty}(u_{\T};\T),
\]
where the hidden constant depends on $\Omega$ but not on $u$ or the size of the elements in the mesh $\T$.
\end{proposition}

Denote by $\mathcal{P}_{\T}$ the $L^2$-projection operator onto functions that are piecewise constant over $\T$.
The local efficiency of the \RR{indicator} \eqref{eq:estimator_local_inf} is as follows. For a proof see \cite[Theorem 3.2]{deldia} and \cite[Lemma 4.2]{rhn} for $n=2$ and $n=3$, respectively; see also \cite[Section 3.4]{demlow2014maximum}.

\begin{proposition}[local efficiency] 
\label{pro:local_infty}
Let $u \in H_0^1(\Omega)\cap L^\infty(\Omega)$ and $u_{\T} \in \V(\T)$ be the solutions to problems \eqref{-lap=f} and \eqref{eq:Gal_sol_f}, respectively. Then
\[
 \E_{\infty}(u_{\T};T) \lesssim 
 \| u - u_{\T} \|_{L^{\infty}(\Ne_T)} + \max_{\RR{T'} \in \Ne_{T}} h_\RR{T'}^2 \| f - \mathcal{P}_{\T}f\|_{L^{\infty}(\RR{T'})} ,
\]
for all $T \in \T$, where the hidden constant is independent of \RR{$u$ and the size of the elements in the mesh $\T$.}
\end{proposition}

The proof of Propositions~\ref{pro:global_infty} and \ref{pro:local_infty} rely on asymptotic estimates for a regularized Green's function. Moreover, the original results required a fineness assumption on the initial mesh together with the existence of a constant $\gamma \geq 1$, such that
\begin{equation}
\label{eq:rhnorig}
  h_{\T}^{\gamma} \lesssim \min_{T \in \T} h_T, \quad \forall \T \in \mathbb{T}.
\end{equation}
These assumptions were later removed in \cite{MR2249676}, using a barrier function argument. The power of the factor $\ell_\T$ in these works was $2-\tfrac23(n-2)$, but it was later shown in \cite{MR3022214,demlow2014maximum} to be equal to one and that this cannot be improved.

\section{The pointwise tracking optimal control problem}
\label{sec:pointwise_tracking}

We follow \cite{AOS2} and invoke the results on weighted Sobolev spaces described in section \ref{sub:weighted_Sobolev} to precisely describe the pointwise tracking optimal control problem introduced in section \ref{sec:introduccion}. We start by defining the set of \emph{observable points} $\mathcal{Z} \subset \Omega$ with $\# \mathcal{Z} = l < \infty$. Since $\#\calZ < \infty$, we know that 
\[
  \AJS{d_{\calZ}=\min\left\{ \textup{dist}(\calZ,\partial\Omega), \min\left\{ \left|z-z'\right|: z,z' \in \calZ, \ z\neq z' \right\} \right\}> 0.}
\]
We then define the weight $\rho$ as follows: if $\# \calZ = 1$, then
\begin{equation}
\label{eq:defofoweight1}
  \rho(x) = \distz^{\alpha}(x),
\end{equation}
otherwise
\begin{equation}
\label{eq:defofoweight2}
  \rho(x) = \begin{dcases}
                \distz^{\alpha}(x), & \exists z \in \calZ: \distz(x) < \frac{d_{\calZ}}2, \\
                1, & \distz(x) \geq \frac{d_{\calZ}}2, \ \forall z \in \calZ.
              \end{dcases}
\end{equation}
Here $\distz(x) = |x-z|$ and $\alpha \in \mathbf{I}  = (n-2,n)$. \AJS{Owing to \cite{MR3215609}}, the weight $\rho$ belongs to the Muckenhoupt class $A_2$ introduced in Definition \ref{def:Muckenhoupt}. \EO{We now state an extension of Lemma \ref{lem:restalpha}, although we omit its proof for brevity.
\begin{lemma}[$H^1_0(\rho,\Omega) \hookrightarrow L^2(\Omega)$]
\label{lem:restalphagen}
If $\alpha \in (n-2,2)$ then $H^1_0(\rho,\Omega) \hookrightarrow L^2(\Omega)$ and we have the following weighted Poincar\'e inequality
\[
  \| v \|_{L^2(\Omega)} \lesssim \| \nabla v \|_{L^2(\rho,\Omega)}, \quad \forall v \in H^1_0(\rho,\Omega),
\]
where the hidden constant depends only on $\Omega$ \AJS{and $d_\calZ$}.
\end{lemma}}

For $\asf, \bsf \in \R$ with $\asf < \bsf$ we define the set of admissible controls
\begin{equation}
\label{eq:Uz}
\mathcal{U}_{\textrm{ad}} = \left\{ \usf \in L^2(\Omega): \asf \leq \usf(x) \le \bsf \RR{\mbox{ for almost every }} x \in \Omega \right\},
\end{equation}
which is a convex, closed and nonempty subset of $L^2(\Omega)$. We recall that the functional $J$ is defined in \eqref{eq:defofJp}. We then define the \emph{pointwise tracking optimal control problem} as follows: Find $\min J (\ysf,\usf)$ subject to the following weak formulation of problem \eqref{eq:defofPDEp}
\begin{equation}
\label{eq:statep}
\ysf\in H^1_0(\Omega): \quad   a(\ysf,\vsf) = (\fsf+\usf, \vsf )_{L^2(\Omega)} \quad \forall \vsf \in H^1_0(\Omega),
\end{equation}
and the control constraints $\usf \in \mathcal{U}_{\textrm{ad}}$. 
\EO{To analyze this optimal control problem, we introduce the so-called control-to-state map $\mathbf{S}:L^2(\Omega) \rightarrow H_0^1(\Omega)$ which, given a control $\usf$, associates to it the unique state $\ysf$ that solves problem \eqref{eq:statep}. With this operator at hand, we define the reduced cost functional
\begin{equation}
\label{eq:F_red}
j(\usf)=J(\mathbf{S}\usf,\usf) = \frac12 \sum_{z \in \calZ} | \mathbf{S}\usf(z) - \ysf_z |^2 + \frac\lambda2 \| \usf \|_{L^2(\Omega)}^2.
\end{equation}
We notice that, since $\fsf \in L^\infty(\Omega)$, $\usf \in\mathcal{U}_{\textrm{ad}} \subset L^{\infty}(\Omega)$ and $\partial \Omega$ is Lipschitz, the results of Proposition \ref{prop:uisW1p} imply that
$\ysf = \mathbf{S}\usf$ is H{\"o}lder continuous and then that the point evaluations of $ \ysf = \mathbf{S}\usf$ in \eqref{eq:F_red} are well defined. In view of the fact that $j$ is weakly lower semicontinuous and strictly convex ($\lambda>0$), we conclude the existence and uniqueness of an optimal control $\bar \usf$ and an optimal state $\bar \ysf$ that verify \eqref{eq:statep} \cite[Theorem 2.14]{Tbook}. In addition, we have that $\bar \usf$ satisfies the first order optimality condition $j'(\bar\usf)( \usf - \bar{\usf} )\geq 0$ for all $\usf \in \mathcal{U}_{\textrm{ad}}$ \cite[Lemma 2.21]{Tbook}. To explore this variational inequality, and to obtain optimality conditions, we begin by further characterizing the range of $\mathbf S$.}

\AJS{
\begin{lemma}[range of $\mathbf S$]
\label{lem:rangeS}
Let $\mathbf S$ denote the control-to-state map, \ie the solution operator to \eqref{eq:statep}. If $u \in \mathcal{U}_{\textrm{ad}}$, then 
$\mathbf{S}u \in W^{1,q}_0(\Omega) \cap H^1_0(\rho^{-1},\Omega)$, where $q>n$ is given by Proposition~\ref{prop:uisW1p}.
\end{lemma}
\begin{proof}
To shorten notation set $y = \mathbf{S}u$. Proposition~\ref{prop:uisW1p} immediately yields that $y \in W^{1,q}_0(\Omega)$ for some $q>n$. Let us now show that $y \in H^1_0(\rho^{-1},\Omega)$. To do so, for each $z \in \calZ$, let $B(z)$ denote the ball with center $z$ and radius $d_\calZ/2$. Set $D = \Omega \setminus \cup_{z \in \calZ} B(z)$ and compute
\[
  \int_\Omega \rho^{-1} | \nabla y|^2 = \sum_{z \in \calZ} \int_{B(z)} \rho^{-1} | \nabla y|^2 + \int_D \rho^{-1} | \nabla y|^2.
\]
By definition, there is $a>0$ such that $\rho(x) \geq a$ for every $x \in D$, thus
\[
  \int_D \rho^{-1} | \nabla y|^2 \lesssim \int_D |\nabla y|^2 \leq \| \fsf + u \|_{L^2(\Omega)}^2.
\]
To bound the integral near the support of the Dirac measures we note that $B(z) \Subset \Omega$, $\fsf + \mathcal{U}_{\textrm{ad}} \subset L^\infty(\Omega)$ and invoke Proposition~\ref{prop:uweight} to obtain
\[
  \int_{B(z)} \rho^{-1} | \nabla y|^2 \lesssim \| \fsf + u \|_{L^\infty(\Omega)}^2,
\]
where the hidden constant depends on $d_\calZ$. The fact that $\# \calZ$ is finite allows us to conclude.
\end{proof}

With this characterization at hand we can proceed to obtain optimality conditions. To do so}
we define the optimal adjoint variable $\bar\psf$ as the unique solution of
\begin{equation}
\label{eq:adjp2}
\bar\psf \in H^1_0(\rho,\Omega): \quad  a(\wsf, \bar\psf) = 
\sum_{z \in \calZ}  \left( \bar\ysf(z) - \ysf_z \right) \delta_z(\wsf) \quad \forall \wsf \in H^1_0(\rho^{-1},\Omega),
\end{equation}
with $\rho$ defined by \eqref{eq:defofoweight1}--\eqref{eq:defofoweight2}, the bilinear form $a$ as in \eqref{eq:defofforma} and $\bar{\ysf} = \mathbf{S} \bar \usf$. We notice that since $\delta_z \in H^1_0(\rho^{-1},\Omega)'$ \cite[Lemma 7.1.3]{KMR}, we invoke \cite[Theorem 2.3]{AGM} and conclude that the adjoint problem \eqref{eq:adjp2} is well posed.

\AJS{We are now in a position to show optimality conditions for our problem.

\begin{theorem}[optimality conditions]
\label{thm:optim_conds}
The pair $(\bar{\ysf},\bar{\usf}) \in H_0^1(\Omega) \times L^2(\Omega)$ is optimal for the pointwise tracking optimal control problem if and only if $\bar\usf \in \mathcal{U}_{\textrm{ad}}$, $\bar{\ysf} = \mathbf{S} \bar \usf$
and the optimal control $\bar{\usf}$ satisfies
\begin{equation}
\label{eq:VIp}
  ( \bar{\psf} +\lambda \bar{\usf}, \usf - \bar{\usf} )_{L^2(\Omega)}  \geq 0 \quad \forall \usf \in \mathcal{U}_{\textrm{ad}},
\end{equation}
where the optimal adjoint state $\bar \psf \in H^1_0(\rho,\Omega)$ solves \eqref{eq:adjp2}\RR{.}
\end{theorem}
\begin{proof}
The first order optimality condition, which characterizes $\bar \usf$ reads, for every $\usf \in \mathcal{U}_{\textrm{ad}}$
\[
  0 \leq j'(\bar \usf)(\usf - \bar \usf) = \sum_{z \in \calZ} \left( \mathbf{S}\bar\usf(z) - \ysf_z \right) \mathbf{S}(\usf - \bar \usf)(z) 
  + \lambda \left( \bar\usf, \usf - \bar \usf \right)_{L^2(\Omega)}.
\]
We now focus on the first term of this inequality. Set $\ysf = \mathbf{S}\usf$ and $\bar\ysf = \mathbf{S}\bar\usf$ and notice that, by Lemma~\ref{lem:rangeS}, we can set $\wsf = \ysf - \bar\ysf \in H^1_0(\rho^{-1},\Omega) \cap C(\bar\Omega)$ in \eqref{eq:adjp2} to obtain
\begin{equation}
\label{eq:dualwithyybar}
  a(\ysf - \bar\ysf, \bar\psf) = \sum_{z \in \calZ} \left( \bar \ysf(z) - \ysf_z \right) \left( \ysf(z) - \bar\ysf(z) \right).
\end{equation}
We would like to set $\vsf = \bar\psf$ in the equation that $\ysf-\bar\ysf$ solves (see \eqref{eq:statep}). If that were possible, we would obtain
\begin{equation}
\label{eq:cheating}
  a(\ysf-\bar\ysf, \bar\psf) = (\usf - \bar\usf, \bar\psf)_{L^2(\Omega)}.
\end{equation}
A combination of \eqref{eq:dualwithyybar}, \eqref{eq:cheating} and the variational inequality would then allow us to conclude. However, $\bar\psf \in H^1_0(\rho,\Omega)\setminus H^1_0(\Omega)$ so that \eqref{eq:cheating} must be justified by different means.

Let $\{p_n\}_{n \in \mathbb N} \subset C_0^\infty(\Omega)$ be such that $p_n \to \bar\psf$ in $H^1_0(\rho,\Omega)$. Setting $\vsf = p_n$ in \eqref{eq:statep} yields
\[
  a(\ysf-\bar\ysf, p_n) = (\usf - \bar\usf, p_n)_{L^2(\Omega)}.
\]
Since $H^1_0(\rho,\Omega) \hookrightarrow L^1(\Omega)$ and $\usf - \bar \usf \in L^\infty(\Omega)$ the right hand side of this expression converges to $(\usf - \bar\usf, \bar \psf)_{L^2(\Omega)}$. The continuity of $a$ in $H^1_0(\rho^{-1}, \Omega) \times H^1_0(\rho, \Omega)$, together with $\ysf-\bar\ysf \in H^1_0(\rho^{-1}, \Omega)$ then yield \eqref{eq:cheating}.
\end{proof}}

We recall the so-called projection formula: $\bar{\usf}$ solves \eqref{eq:VIp} if and only if \cite[section 3.6.3]{Tbook}, \cite[section 2.1]{MR0271512}
\begin{equation}
\label{eq:Pi}
  \bar{\usf} = \Pi\left(-\frac{1}{\lambda}\bar{\psf}\right),
\end{equation}
where the projection operator $\Pi: \AJS{L^1(\Omega)} \rightarrow \mathcal{U}_{\textrm{ad}}$ is defined by 
\begin{equation} 
\label{def:Pi}
  \Pi(v) = \min\{\bsf, \max\{\asf, v\}\},
\end{equation}
and gives the best approximation of $v$ in $\mathcal{U}_{\textrm{ad}}$.

We now recall the finite element approximation of the pointwise tracking optimal control problem proposed and analyzed in \cite{AOS2}. The approximation of the optimal control $\bar{\usf}$ is done by piecewise constant functions: $\bar\usf_{\T} \in \U_{\textrm{ad}}(\T)$, where
\begin{equation}
\label{eq:Upc}
\U_{\textrm{ad}}(\T) = \U(\T) \cap \mathcal{U}_{\textrm{ad}}, ~~ \U(\T) = \left\{ v_\T \in L^\infty(\Omega): v_{\T|T} \in \mathbb{P}_0(T),~\forall~T\in\T \right\},
\end{equation}
with $\mathcal{U}_{\textrm{ad}}$ defined in \eqref{eq:Uz}. The optimal state and adjoint state are discretized using the finite element space $\V(\T)$ defined in \eqref{eq:defFESpace}. In this setting, the discrete counterpart of \eqref{eq:defofJp}--\eqref{eq:adjp} reads: Find
$
\min J (\ysf_{\T},\usf_{\T})
$
subject to the discrete state equation
\begin{equation}
\label{eq:abstateh}
\ysf_{\T} \in \V(\T): \quad a(\ysf_{\T}, \vsf_{\T} ) = (\fsf + \usf_{\T}, \vsf_{\T})_{L^2(\Omega)} \quad \forall \vsf_{\T} \in \V(\T),
\end{equation}
and the discrete control constraints
$
\usf_{\T} \in \U_{\textrm{ad}}(\T).
$
The pair $(\bar{\ysf}_{\T},\bar{\usf}_{\T})$ is optimal for the previous discrete optimal control problem if and only $\bar{\ysf}_{\T}$ solves \eqref{eq:abstateh} and
\begin{equation}
\label{eq:VIpd}
  ( \bar{\psf}_{\T} +\lambda \bar{\usf}_{\T}, \usf_{\T} - \bar{\usf}_{\T} )_{L^2(\Omega)}  \geq 0 \quad \forall \usf_{\T} \in \U_{\textrm{ad}}(\T),
\end{equation}
where $\bar{{\psf}}_{\T}$ solves the discrete counterpart of \eqref{eq:adjp2}, that is
\begin{equation}
\label{eq:adjointh}
  \psf_\T \in \V(\T): \quad  a(\wsf_\T,\psf_\T) = \sum_{z \in \calZ} ( \ysf_\T(z) - \ysf_z) \delta_z( \wsf_\T) \quad \forall \wsf_\T \in \V(\T).
\end{equation}

Exploiting the function-space setting based on the Muckenhoupt weight $\varpi$, defined in \eqref{eq:defofvarpi}, the following a priori error estimate was derived in \cite{AOS2}: If $\Omega$ is convex and the mesh $\T$ is quasiuniform with mesh size $h_{\T}$, then
\begin{equation}
\label{eq:tracking_estimate}
  \| \bar \usf - \bar{\usf}_\T \|_{L^2(\Omega)} \lesssim  h_{\T}^{2-n/2} |\log h_{\T}|^{n-1},
\end{equation}
where the hidden constant is independent of $\T$ and $\bar{\usf}$. 

\subsection{A posteriori error analysis}
\label{sub:tracking_aposteriori}

The a priori error estimate \eqref{eq:tracking_estimate} is suboptimal in terms of approximation. The reduced regularity of the optimal adjoint state $\bar{\psf}$ does not allow the method to exhibit its optimal rate of convergence. In addition, the a priori error estimate \eqref{eq:tracking_estimate} requires the mesh $\T$ to be quasiuniform and the domain $\Omega$ to be convex. However, as it was commented in section \ref{sec:pointwise_tracking}, the well-posedness of the pointwise tracking optimal control problem relies only on the Lipschitz property of $\partial \Omega$.  The need for convexity and the lack of regularity properties of $\bar{\psf}$ when $n=3$ motivate the study of AFEM to solve the optimal control problem \eqref{eq:defofJp}--\eqref{eq:adjp}. We now proceed to propose and analyze the key ingredient of any AFEM: an a posteriori error estimator.

The derivation and analysis of an a posteriori error estimator for problem \eqref{eq:defofJp}--\eqref{eq:adjp} is \AJS{far from} \RR{trivial}. The pointwise tracking optimal control problem involves: pointwise evaluations of the optimal state $\bar{\ysf}$, an elliptic equation with point sources \eqref{eq:adjp2} and an intrinsic nonlinearity introduced by the constraints on the optimal control, \ie $\bar{\usf} \in \mathcal{U}_{\textrm{ad}}$. We incorporate these main features in an a posteriori error estimator, which is defined as the sum of three contributions:
\begin{equation}
 \label{eq:defofEocp}
 \RR{\E_{\textrm{ocp}}^2}(\bar{\ysf}_{\T},\bar{\psf}_{\T},\bar{\usf}_{\T}; \T) = \RR{\E_{\ysf}^2}(\bar{\ysf}_{\T},\bar{\usf}_{\T}; \T) + \RR{\E_{\psf}^2}(\bar{\psf}_{\T},\bar{\ysf}_{\T}; \T) + \RR{\E_{\usf}^2}(\bar{\usf}_{\T},\bar{\psf}_{\T}; \T),
\end{equation}
where $\T \in \Tr$ and $\bar \ysf_{\T}$, $\bar \usf_{\T}$ and $\bar \psf_{\T}$ denote the optimal variables \AJS{that solve \eqref{eq:abstateh}--\eqref{eq:adjointh}}.

We now define and describe each contribution in \eqref{eq:defofEocp} separately. First, on the basis of section \ref{sec:maximum}, we define the local pointwise indicator associated with the state equation \eqref{eq:defofPDEp} as
\begin{equation}
 \label{eq:defofEy}
 \E_{\ysf}(\bar{\ysf}_{\T},\bar{\usf}_{\T}; T)= h_T^2 \| \fsf + \bar{\usf}_{\T} \|_{L^{\infty}(T)} + h_T \| \llbracket \nu \cdot \nabla \bar{\ysf}_{\T} \rrbracket\|_{L^{\infty}(\partial T \setminus \partial \Omega)},
\end{equation}
where $\llbracket \nu \cdot \nabla \bar{\ysf}_{\T} \rrbracket$ denotes the jumps of the normal derivative of $\bar{\ysf}_{\T}$ across interelement sides as defined in \eqref{eq:jump}. The global pointwise estimator is then defined by
\begin{equation}
\label{eq:defofEyg}
\E_{\ysf}(\bar{\ysf}_{\T},\bar{\usf}_{\T}; \T) = \max_{T \in \T } \E_{\ysf}(\bar{\ysf}_{\T},\bar{\usf}_{\T}; T).
\end{equation}

To introduce an a posteriori error indicator associated with the adjoint equation \eqref{eq:adjp2} we assume that: 
\begin{equation}\label{eq:patch}
\forall T \in \T, \ \# (\Ne_T \cap \calZ) \leq 1,
\end{equation}
that is, for every element $T \in \T$ its patch $\Ne_T$ contains at most one observable point.
This is not a restrictive assumption, as it can always be satisfied by starting with a suitably refined mesh.
Define
\begin{equation}
\label{eq:DTapost}
  D_T = \min_{z \in \calZ} \left\{  \max_{x \in T} |x-z| \right\}.
\end{equation}
On the basis of the results presented in section \ref{sec:a_posteriori}, we define the local error indicator as
\begin{multline}
 \label{eq:defofEp}
 \E_{\psf}(\bar{\psf}_{\T},\bar{\ysf}_{\T}; T)= 
    \bigg( h_T D_T^{\alpha} \| \llbracket \nu \cdot \nabla \bar{\psf}_{\T} \rrbracket\|^2_{L^2(\partial T\setminus \partial \Omega)} 
    \\
    + \sum_{z \in \calZ \cap T} h_T^{\alpha + 2 - n} |\bar{\ysf}_{\T}(z) - \ysf_{z}|^2 \bigg)^{1/2}.
\end{multline}
The global error estimator is thus defined by
\begin{equation}
\label{eq:defofEpg}
  \E_{\psf}(\bar{\psf}_{\T},\bar{\ysf}_{\T}; \T) = \left(\sum_{T \in \T}  \E^2_{\psf}(\bar{\psf}_{\T},\bar{\ysf}_{\T}; T) \right)^{1/2}. 
\end{equation}
Finally, we define a global error estimator associated with the optimal control, as follows:
\begin{equation}
\label{eq:defofEug}
  \E_{\usf}(\bar{\usf}_{\T},\bar{\psf}_{\T}; \T) = \left( \sum_{T \in \T}  \E^2_{\usf}(\bar{\usf}_{\T},\bar{\psf}_{\T}; T) \right)^{1/2},
\end{equation}
with the local error indicators
\begin{equation}\label{eq:defofEu}
\E_{\usf}(\bar{\usf}_{\T},\bar{\psf}_{\T}; T) = \|\bar{\usf}_{\T} - \Pi (-\tfrac{1}{\lambda} \bar{\psf}_{\T})\|_{L^2(T)}.
\end{equation}
In \eqref{eq:defofEu}, $\Pi$ denotes the nonlinear projection operator defined by formula \eqref{def:Pi}.

\subsubsection{Error estimator: reliability}
\label{subsub:reliable}
We now proceed to derive the global reliability of the error indicator $\E_{\textrm{ocp}}$ defined by \eqref{eq:defofEocp}.

\begin{theorem}[global reliability property of $\E_{\textrm{ocp}}$]
Let $(\bar{\usf},\bar{\ysf},\bar{\psf}) \in L^2(\Omega) \times H_0^1(\Omega) \times H_0^1(\rho,\Omega)$ be the solution to the optimality system \AJS{\eqref{eq:statep}, \eqref{eq:adjp2} and \eqref{eq:VIp}} and $(\bar{\usf}_{\T},\bar{\ysf}_{\T},\bar{\psf}_{\T}) \in  \U_{\textrm{ad}}(\T) \times \V(\T) \times \V(\T)$ its numerical approximation \AJS{given by \eqref{eq:abstateh}--\eqref{eq:adjointh}. If $\alpha \in (n-2,2)$, then}
\RR{\begin{equation}
\label{eq:reliability}
 \begin{aligned}
\| \bar{\usf} - \bar{\usf}_{\T}\|_{L^2(\Omega)}^2 + \| \bar{\ysf} - \bar{\ysf}_{\T} \|_{L^{\infty}(\Omega)}^2 
+ \| \nabla( \bar{\psf} - \bar{\psf}_{\T}) \|_{L^2(\rho,\Omega)}^2 \lesssim \ell_\T^2 \E_{\ysf}^2(\bar{\ysf}_{\T},\bar{\usf}_{\T}; \T) 
\\
+ \E_{\psf}^2(\bar{\psf}_{\T},\bar{\ysf}_{\T}; \T) + \E_{\usf}^2(\bar{\usf}_{\T},\bar{\psf}_{\T}; \T) \lesssim 
(1+ \ell_\T)^2
\E^2_{\textrm{\emph{ocp}}} (\bar{\ysf}_{\T},\bar{\psf}_{\T},\bar{\usf}_{\T}; \T),
\end{aligned}
\end{equation}}
where $\ell_{\T}$ is defined in \eqref{eq:log} and the hidden constant is independent of $\bar{\ysf}$, $\bar{\usf}$, $\bar{\psf}$, the size of elements in the mesh $\T$ and $\#\T$. \AJS{The constant, however, blows up as $\lambda \downarrow 0$}.
\label{TH:global_reliability}
\end{theorem}

\begin{proof}
We proceed in seven steps.

\noindent \framebox{Step 1.} We define $\tilde{\usf} = \Pi (-\tfrac{1}{\lambda}\bar{\psf}_{\T})$ which, \AJS{owing to \cite[Lemma 2.26]{Tbook},} can be equivalently characterized by
 \begin{equation}
  \label{tildeu}
  (\lambda \tilde{\usf} + \bar{\psf}_{\T}, \usf - \tilde{\usf})_{L^2(\Omega)} \geq 0 \quad \forall \usf \in \mathcal{U}_{\textrm{ad}}.
 \end{equation}
With this definition at hand, a simple application of the triangle inequality yields
\begin{equation}
 \label{u-uh-utilde}
 \|  \bar{\usf} -\bar{\usf}_{\T} \|_{L^2(\Omega)} \leq \|  \bar{\usf} -\tilde{\usf} \|_{L^2(\Omega)} + \|  \tilde{\usf} -\bar{\usf}_{\T} \|_{L^2(\Omega)}.
\end{equation}
Notice that, by the definition of $\tilde{\usf}$, the second term on the right hand side of \eqref{u-uh-utilde} is the global error estimator associated with the optimal control $\E_{\usf}(\bar{\usf}_{\T},\bar{\psf}_{\T}; \T)$ given in \eqref{eq:defofEug}. 

\noindent \framebox{Step 2.} Let us now focus on the first term on the right hand side of \eqref{u-uh-utilde}. Set $\usf = \tilde{\usf}$ in \eqref{eq:VIp} and $\usf = \bar{\usf}$ in \eqref{tildeu} and add the obtained inequalities to arrive at
\begin{equation}
\label{eq:usf-tildeusf}
 \lambda \| \bar{\usf} - \tilde{\usf} \|^2_{L^2(\Omega)} \leq (\bar{\psf} - \bar{\psf}_{\T}, \tilde{\usf} - \bar{\usf})_{L^2(\Omega)}.
\end{equation}
To control the right hand side of \eqref{eq:usf-tildeusf} we define an auxiliary adjoint state via the solution to the following problem:
\begin{equation}
\label{eq:adjq}
\qsf \in H_0^1(\rho,\Omega): \quad a(\wsf,\qsf) = \sum_{z \in \calZ} ( \bar{\ysf}_{\T}(z) - \ysf_z ) \delta_z(\wsf) \quad \forall \wsf \in H^1_0(\rho^{-1},\Omega).
\end{equation}
We then write $\bar{\psf} - \bar{\psf}_{\T} = (\bar{\psf}  - \qsf) + (\qsf - \bar{\psf}_{\T})$ in \eqref{eq:usf-tildeusf} to obtain
\begin{equation}
\label{I+II}
 \lambda \| \bar{\usf} - \tilde{\usf} \|^2_{L^2(\Omega)} \leq (\bar{\psf}  - \qsf, \tilde{\usf} - \bar{\usf})_{L^2(\Omega)} + (\qsf - \bar{\psf}_{\T}, \tilde{\usf} - \bar{\usf})_{L^2(\Omega)} = \textrm{I} + \textrm{II}.
\end{equation}
To control the term $\textrm{II}$, we exploit the fact that $\qsf$ solves \eqref{eq:adjq} and $\bar{\psf}_{\T}$ corresponds to its Galerkin approximation. Now, an adaption of the arguments developed in \cite{AGM} provide the estimate
\begin{equation}
\label{eq:qest}
  \| \nabla (\qsf  - \bar{\psf}_{\T}) \|_{L^2(\rho,\Omega)} \lesssim \E_{\psf}(\bar{\psf}_{\T},\bar{\ysf}_{\T}; \T),
\end{equation}
where $\E_{\psf}$ denotes the a posteriori error estimator for problem \eqref{eq:adjq} defined in \eqref{eq:defofEp}--\eqref{eq:defofEpg}. For brevity we skip details and only remark that \eqref{eq:qest} is valid because of assumption \eqref{eq:patch}. Consequently, \RR{Lemma~\ref{lem:restalphagen}}
allows us to arrive at
\begin{equation*}
 |\textrm{II}| \lesssim \| \nabla (\qsf  - \bar{\psf}_{\T}) \|_{L^2(\rho,\Omega)}\| \tilde{\usf} - \bar{\usf} \|_{L^2(\Omega)} ,
\end{equation*}
which together with a Young's inequality and \eqref{eq:qest}, yield
\begin{equation}
\label{II}
 |\textrm{II}| - \frac{\lambda}{4}\| \tilde{\usf} - \bar{\usf} \|^2_{L^2(\Omega)} \lesssim \E^2_{\psf}(\bar{\psf}_{\T},\bar{\ysf}_{\T}; \T ) .
\end{equation}

\noindent \framebox{Step 3.} We now estimate the term $\textrm{I}$. To do this, we introduce another auxiliary adjoint state $\rsf$ via the solution to the following problem:
\begin{equation}
\label{eq:adjr}
\rsf \in H^1_0(\rho,\Omega): \quad a(\wsf,\rsf) = \sum_{z \in \calZ} ( \tilde{\ysf}(z) - \ysf_z) \delta_z(\wsf) \quad \forall \wsf \in H^1_0(\rho^{-1},\Omega),
\end{equation}
where $\tilde{\ysf} \in H_0^1(\Omega)$ solves $a(\tilde{\ysf},\vsf) = (\fsf + \tilde{\usf},\vsf)$ for all $\vsf \in H_0^1(\Omega)$. We recall that $\tilde{\usf} = \Pi (-\tfrac{1}{\lambda}\bar{\psf}_{\T})$. With the definition of the state $\rsf$ at hand, we write $\bar{\psf}  - \qsf = (\bar{\psf} - \rsf) + (\rsf - \qsf)$ and then, the term $\textrm{I}$ in \eqref{I+II} becomes 
\[
 (\bar{\psf}  - \qsf, \tilde{\usf} - \bar{\usf})_{L^2(\Omega)} =  (\bar{\psf} - \rsf, \tilde{\usf} - \bar{\usf})_{L^2(\Omega)} + (\rsf - \qsf, \tilde{\usf} - \bar{\usf})_{L^2(\Omega)} = \textrm{I}_1 + \textrm{I}_2.
\]
We now proceed to control the term $\textrm{I}_1$. \AJS{Since $\bar{\usf} - \tilde{\usf} \in L^\infty(\Omega)$, and $\bar{\ysf} - \tilde{\ysf}$ solves $a(\bar{\ysf} - \tilde{\ysf},\vsf) = (\bar{\usf} - \tilde{\usf},\vsf)$ for all $\vsf \in H_0^1(\Omega)$, Lemma~\ref{lem:rangeS} allows us to conclude that $\bar{\ysf} - \tilde{\ysf} \in H_0^1(\rho^{-1},\Omega)$ and then, by setting $\wsf = \bar{\ysf} - \tilde{\ysf}$ in \eqref{eq:adjp2} and \eqref{eq:adjr}, we derive 
\[
  a(\bar{\ysf} - \tilde{\ysf},\bar{\psf} - \rsf) =  \sum_{z \in \calZ} |\bar{\ysf}(z) - \tilde{\ysf}(z)|^2.
\]
On the other hand, a similar approximation argument to that used in the proof of Theorem~\ref{thm:optim_conds} shows that
\[
  a(\bar\ysf - \tilde\ysf, \bar\psf - \rsf) = (\bar{\usf} - \tilde{\usf}, \bar{\psf} - \rsf )_{L^2(\Omega)}.
\]
In conclusion,
\[
  \textrm{I}_1 =  - \sum_{z \in \calZ} |\bar{\ysf}(z) - \tilde{\ysf}(z)|^2 \leq 0.
\]}

\noindent \framebox{Step 4.} We now estimate the term $\textrm{I}_2$. To accomplish this we first notice that the function $\rsf - \qsf \in H^1_0(\rho,\Omega)$ solves
\begin{equation*}
\quad a(\wsf,\rsf - \qsf) = \sum_{z \in \calZ} ( \tilde{\ysf}(z)- \bar{\ysf}_{\T}(z)) \delta_z(\wsf) \quad \forall \wsf \in H^1_0(\rho^{-1},\Omega).
\end{equation*}
Consequently, \AJS{the embedding of} \RR{Lemma~\ref{lem:restalphagen}} in conjunction with the stability of the problem above yield
\begin{equation}
\label{r-q}
 \| \rsf - \qsf \|_{L^2(\Omega)} \lesssim  \| \nabla (\rsf - \qsf) \|_{L^2(\rho,\Omega)} \lesssim \| \tilde{\ysf} - \bar{\ysf}_{\T}  \|_{L^{\infty}(\Omega)}.
\end{equation}
To control the right hand side of the previous expression, we use the triangle inequality to obtain $\| \tilde{\ysf} -\bar{\ysf}_{\T}  \|_{L^{\infty}(\Omega)} \lesssim \| \tilde{\ysf} - \ysf^*  \|_{L^{\infty}(\Omega)} + \| \ysf^* -\bar{\ysf}_{\T} \|_{L^{\infty}(\Omega)}$ where $\ysf^* \in H_0^1(\Omega)$ solves $a(\ysf^*,\vsf) = (\fsf+\bar{\usf}_{\T},\vsf)$ for all $\vsf \in H_0^1(\Omega)$.
\AJS{The results of Proposition~\ref{prop:uisW1p} imply that, for some $q>n$,
\begin{equation}
\label{eq:W1pestimate}
  \| \tilde{\ysf} - \ysf^*  \|_{L^{\infty}(\Omega)} \lesssim \| \tilde{\ysf} - \ysf^* \|_{W^{1,q}(\Omega)} \lesssim \| \tilde{\usf} - \bar{\usf}_{\T} \|_{L^2(\Omega)} = \E_{\usf}(\bar{\usf}_{\T},\bar{\psf}_{\T}; \T).
\end{equation}}
Since $\bar{\ysf}_{\T}$ is the Galerkin approximation of $\ysf^*$, the second term is estimated by invoking the global reliability of the error estimator $\E_{\ysf}$ defined in \eqref{eq:defofEy}--\eqref{eq:defofEyg}: $\| \ysf^* -\bar{\ysf}_{\T} \|_{L^{\infty}(\Omega)} \lesssim \ell_\T \E_{\ysf}(\bar{\ysf}_{\T},\bar{\usf}_{\T};\T)$. Applying these estimates to \eqref{r-q} we obtain
\[
  |\textrm{I}_2| - \frac{\lambda}{4}\| \tilde{\usf} - \bar{\usf} \|^2_{L^2(\Omega)} \lesssim 
  \ell_\T^2 \E^2_{\ysf}(\bar{\ysf}_{\T},\bar{\usf}_{\T};\T) + \E^2_{\usf}(\bar{\usf}_{\T},\bar{\psf}_{\T}; \T).
\]
The fact that $\textrm{I}_1 \leq 0$, implies that a similar estimate is valid for the term $\textrm{I} = \textrm{I}_1 + \textrm{I}_{2}$. This, together with \eqref{u-uh-utilde}, \eqref{I+II} and \eqref{II}, implies that
\RR{\begin{equation}
\label{u-uT}
\| \bar{\usf} - \bar{\usf}_{\T} \|_{L^2(\Omega)}^2 \lesssim  \ell_\T^2 \E_{\ysf}^2(\bar{\ysf}_{\T},\bar{\usf}_{\T};\T) + \E_{\psf}^2(\bar{\psf}_{\T},\bar{\ysf}_{\T};\T) + \E_{\usf}^2(\bar{\usf}_{\T},\bar{\psf}_{\T}; \T).
\end{equation}}

\noindent \framebox{Step 5.} The goal of this step is to control the error $\bar{\ysf} - \bar{\ysf}_{\T}$ in the $L^{\infty}$--norm in terms of $\E_{\textrm{ocp}}$. To do this, we write $\bar{\ysf} - \bar{\ysf}_{\T} = (\bar{\ysf} - \ysf^*) + (\ysf^*- \bar{\ysf}_{\T})$ and estimate each term separately. To control the first term we invoke a similar argument to the one that gives \eqref{eq:W1pestimate}: 
\[
 \| \bar{\ysf} - \ysf^* \|_{L^{\infty}(\Omega)} \lesssim \| \bar{\ysf} - \ysf^* \|_{W^{1,\AJS{q}}(\Omega)} \lesssim \| \bar{\usf} - \bar{\usf}_{\T} \|_{L^2(\Omega)}. 
\]
On the other hand, the estimate 
\[
 \| \ysf^*- \bar{\ysf}_{\T} \|_{L^{\infty}(\Omega)} \lesssim \ell_\T \E_{\ysf}(\bar{\ysf}_{\T},\bar{\usf}_{\T};\T) 
\]
follows from the fact that $\bar{\ysf}_{\T}$ is the Galerkin approximation of $\ysf^*$ and the global reliability of $\E_{\ysf}$ (Proposition \ref{pro:global_infty}). Collecting the derived estimates and invoking \eqref{u-uT} we obtain 
\RR{\begin{equation}
\label{y-yT}
 \| \bar{\ysf}- \bar{\ysf}_{\T} \|_{L^{\infty}(\Omega)}^2 \lesssim \ell_{\T}^2 \E_{\ysf}^2(\bar{\ysf}_{\T},\bar{\usf}_{\T};\T) + \E_{\psf}^2(\bar{\psf}_{\T},\bar{\ysf}_{\T};\T) + \E_{\usf}^2(\bar{\usf}_{\T},\bar{\psf}_{\T}; \T).
\end{equation}}

\noindent \framebox{Step 6.}  In this step we control the error $\bar{\psf} - \bar{\psf}_{\T}$ in the $H^1(\rho,\Omega)$-seminorm. A simple application of the triangle inequality yields
\[
\| \nabla(\bar{\psf} - \bar{\psf}_{\T}) \|_{L^2(\rho,\Omega)} \lesssim \| \nabla(\bar{\psf} - \qsf) \|_{L^2(\rho,\Omega)} + \| \nabla(\qsf - \bar{\psf}_{\T}) \|_{L^2(\rho,\Omega)}.
\]
To control the first term, we invoke a stability result and conclude that 
\[
\| \nabla(\bar{\psf} - \qsf) \|_{L^2(\rho,\Omega)}\lesssim\| \bar{\ysf} - \bar{\ysf}_{\T}  \|_{L^{\infty}(\Omega)}.
\]
Consequently, \eqref{eq:qest} and \eqref{y-yT} yield that
\RR{\begin{equation}
\label{p-pT}
\| \nabla(\bar{\psf} - \bar{\psf}_{\T}) \|_{L^2(\rho,\Omega)}^2 \lesssim \ell_\T^2 \E_{\ysf}^2(\bar{\ysf}_{\T},\bar{\usf}_{\T};\T) + \E_{\psf}^2(\bar{\psf}_{\T},\bar{\ysf}_{\T};\T) + \E_{\usf}^2(\bar{\usf}_{\T},\bar{\psf}_{\T}; \T). 
\end{equation}}
Notice that here, once again, we need to rely on assumption \eqref{eq:patch} for estimate \eqref{eq:qest} to be valid.

\noindent \framebox{Step 7.} Finally, collecting \eqref{u-uT}, \eqref{y-yT} and \eqref{p-pT} we obtain \eqref{eq:reliability}, which concludes the proof.
\end{proof}

\begin{remark}[\AJS{range of $\alpha$}]\rm
\AJS{Notice that although problem \eqref{-lap=delta} is well-posed for $\alpha \in \mathbf{I}=(n-2,n)$ and Proposition~\ref{pro:global_alpha} asserts the reliability of $\E_\alpha$ for the same range, in three dimensions we must further restrict the range of $\alpha$ to guarantee the embedding of} \RR{Lemma~\ref{lem:restalphagen}.}
\end{remark}

\subsubsection{Error estimator: efficiency}
\label{subsub:efficient}
We now analyze the efficiency properties of the a posteriori error estimator $\E_{\textrm{ocp}}$ defined in \eqref{eq:defofEocp} by examining each of its contributions separately. 
Before proceeding with such analysis, we introduce the following notation: for an edge, triangle or tetrahedron $G$, let $\mathcal{V}(G)$ be the set of vertices of $G$. With this notation at hand, we 
introduce two smooth functions $\varphi_T$ and $\varphi_S$ as follows. 
Given $T \in \T$, we consider a non-negative and smooth function $\varphi_{T}$ with the following properties:
\begin{equation}
 \label{eq:propertiesofeta}
 |T| \lesssim \int_{T} \varphi_{T}, \quad \supp \varphi_{T} = T, \quad \|\nabla^{k} \varphi_{T} \|_{L^{\infty}(T)} \lesssim h_{T}^{-k}, \ k=0,1,2.
\end{equation}
For example, following \cite{Verfurth2,Verfurth}, we may take $\varphi_T$ given by
\[
 \varphi_{T|T} = (n+1)^{n+1} \prod_{\vero\in\mathcal{V}(T)} \xi_{\vero}, 
\]
where $\xi_{\vero}$ are the barycentric coordinates of $T$. We now define the function $\varphi_{S}$ for $S \in \Sides$. To do this, let $\Ne_{S}$ be the patch composed of the two elements of $\T$ sharing $S$. We then define $\varphi_S$ as a smooth cut-off function with the properties
\begin{equation}
 \label{eq:propertiesofetaS}
 |S| \lesssim \int_{S} \varphi_{S}, \quad \supp \varphi_{S} = \Ne_S,  \quad 
 \|\nabla^{k} \varphi_{S} \|_{L^{\infty}(T)} \lesssim h_{S}^{-k}, \ k=0,1,2.
\end{equation}
For example, we may take $\varphi_S$ given by
\[
 \varphi_{S|T} = n^{n} \prod_{\vero\in\mathcal{V}(S)} \xi_{\vero}, 
\]
where $T \subset \Ne_S$. This choice is due to Verf\"{u}rth \cite{Verfurth2,Verfurth}.

Recall that $\mathcal{P}_{\T}$ denotes the $L^2$-projection operator onto piecewise constant, over $\T$, functions. With these elements at hand we derive the local efficiency of $\E_{\ysf}$.

\begin{lemma}[local efficiency of $\E_{\ysf}$] 
\RR{Let} $(\bar{\usf},\bar{\ysf},\bar{\psf}) \in L^2(\Omega) \times H_0^1(\Omega) \times H_0^1(\rho,\Omega)$ \RR{be} the solution to the optimality system \eqref{eq:statep}--\eqref{eq:adjp2} and $(\bar{\usf}_{\T},\bar{\ysf}_{\T},\bar{\psf}_{\T}) \in  \U_{\textrm{ad}}(\T) \times \V(\T) \times \V(\T)$ \RR{be} its numerical approximation defined \AJS{by \eqref{eq:abstateh}--\eqref{eq:adjointh}}\RR{. If $\alpha \in \mathbf{I}$}, then 
\begin{equation}
\label{eq:efficiencyy}
 \begin{aligned}
\E_{\ysf} (\bar{\ysf}_{\T},\bar{\usf}_{\T}; T) & \lesssim \|  \bar{\ysf} - \bar{\ysf}_{\T} \|_{L^{\infty}(\Ne_T)} + h_{T}^{2-n/2} \| \bar{\usf} - \bar{\usf}_{\T}\|_{L^2(\Ne_T)}\\
& + h_{T}^{2} \| \fsf - \mathcal{P}_{\T}\fsf\|_{L^{\infty}(\Ne_T)},
\end{aligned}
\end{equation}
where the hidden constant is independent of the optimal variables, their approximations, the size of the elements in the mesh $\T$ and $\#\T$. 
\end{lemma}
\begin{proof}
Let us consider a function $\vsf \in H^1_0(\Omega)$ which is such that $\vsf_{|T} \in C^2(T)$ for all $T \in \T$. We then invoke the fact that $\bar{\ysf} \in H^1_0(\Omega)$ solves \eqref{eq:statep} and integration by parts to obtain
\begin{equation*}
\label{eq:res1}
\int_{\Omega} \nabla (\bar{\ysf} - \bar{\ysf}_{\T}) \cdot \nabla \vsf = \int_{\Omega} \left(\fsf+\bar{\usf}\right) \vsf - \int_{\Omega} \nabla \bar{\ysf}_{\T} \cdot \nabla  \vsf = \sum_{T \in \T} \int_{T} \left(\fsf+\bar{\usf}\right) \vsf + \sum_{S \in \Sides } \int_{S } \llbracket \nu\cdot\nabla \bar{\ysf}_{\T} \rrbracket  \vsf.
\end{equation*}
On the other hand, since on each element $T$ we have that $\vsf \in C^2(T)$, integrating by parts again allows us to derive
\begin{equation*}
\label{eq:res2}
\int_{\Omega} \nabla (\bar{\ysf} - \bar{\ysf}_{\T}) \cdot \nabla \vsf = - \sum_{T \in \T} \int_{T} (\bar{\ysf} - \bar{\ysf}_{\T}) \Delta \vsf  + \sum_{S \in \Sides } \int_{S} \llbracket \nu\cdot\nabla \vsf \rrbracket (\bar{\ysf} - \bar{\ysf}_{\T}).
\end{equation*}
Consequently, the right hand sides of the previous two expressions coincide and, in particular, by setting $\vsf = \varphi_{T}$ 
defined through properties  \eqref{eq:propertiesofeta}, we deduce that
\begin{equation}
\label{eq:res3}
  \begin{aligned}
    \int_{T} \left(\mathcal{P}_{\T}\fsf+\bar{\usf}_{\T}\right) \varphi_{T}  & = \int_{T} \left(\mathcal{P}_{\T}\fsf- \fsf \right) \varphi_{T} + \int_{T} \left(\bar{\usf}_{\T}-\bar{\usf}\right) \varphi_{T}
    \\
    &- \int_{T} (\bar{\ysf} - \bar{\ysf}_{\T}) \Delta \varphi_{T}  + \sum_{S \in \Sides_T } \int_{S} \llbracket \nu\cdot\nabla \varphi_T \rrbracket (\bar{\ysf} - \bar{\ysf}_{\T}),
  \end{aligned}
\end{equation}
where $\Sides_T$ is the subset of $\Sides$ which contains the sides in $\Sides$ which are sides of $T$. Similarly, if we consider $\vsf = \varphi_S$, defined through properties \eqref{eq:propertiesofetaS}, we obtain that
\begin{multline}
\label{eq:res4}
    \int_{S} \llbracket \nu\cdot\nabla \bar{\ysf}_{\T} \rrbracket \varphi_S  = -\int_{\Ne_S} \left(\fsf+\bar{\usf}_{\T}\right)  \varphi_S - \int_{\Ne_S} (\bar{\usf} - \bar{\usf}_{\T})  \varphi_S
    \\
    - \sum_{T\in\T:\,T\subset\Ne_S}\int_{T} (\bar{\ysf} - \bar{\ysf}_{\T}) \Delta \varphi_{S} + \sum_{S'\in\Sides:\,S'\subset\Ne_S}\int_{S'} \llbracket \nu\cdot\nabla \varphi_S \rrbracket (\bar{\ysf} - \bar{\ysf}_{\T}).
\end{multline}
With these ingredients we proceed to divide the proof in two steps. 

\noindent \framebox{Step 1.} We estimate the term $h_T^2 \| \fsf+\bar\usf_{\T}\|_{L^{\infty}(T)}$ in \eqref{eq:defofEy}. A simple application of the triangle inequality allows us to derive
\begin{equation}
\label{eq:Pfaux}
h_T^2 \| \fsf+\bar\usf_{\T}\|_{L^{\infty}(T)} \lesssim h_T^2 \| \mathcal{P}_{\T}\fsf+\bar\usf_{\T}\|_{L^{\infty}(T)}+ h_T^2\| \fsf-\mathcal{P}_{\T}\fsf\|_{L^{\infty}(T)}.
\end{equation}
To control the first term on the right hand side of the previous expression we invoke identity \eqref{eq:res3}. This, in view of \eqref{eq:propertiesofeta} and the fact that $(\mathcal{P}_\T \fsf + \bar\usf_{\T})_{|T} \in \mathbb{P}_0\AJS{(T)}$, allows us to obtain
\begin{align*}
 \left| \mathcal{P}_{\T}\fsf+\bar\usf_{\T} \right| ~ | T | & \lesssim \left| \int_{T} (\bar{\ysf} - \bar{\ysf}_{\T}) \Delta \varphi_{T} \right|  + \sum_{S' \in \Sides_T} \left| \int_{S'} \llbracket \nu\cdot \nabla \varphi_T \rrbracket (\bar{\ysf} - \bar{\ysf}_{\T}) \right|
\\
& + \left|  \int_{T} (\fsf-\mathcal{P}_{\T}\fsf) \varphi_{T} \right|+\left|  \int_{T} (\bar{\usf}-\bar{\usf}_{\T}) \varphi_{T} \right|
\\
& \lesssim \left( h_{T}^{-2} |T| + h_{T}^{-1}\sum_{S'\in\Sides_T}|S'| \right) \| \bar{\ysf} - \bar{\ysf}_{\T}\|_{L^{\infty}(T)}
\\
& + |T|\| \fsf-\mathcal{P}_{\T}\fsf\|_{L^{\infty}(T)} + |T|^{\tfrac{1}{2}}\| \bar{\usf} - \bar{\usf}_{\T } \|_{L^2(T)}.
\end{align*}
We now recall that, for any $T \in \T$, we have that $|T| \approx h_{T}^n$ and $|T|/|S| \approx h_{T}$. Consequently, in view of \eqref{eq:Pfaux} and the previous estimate, we derive
\[
h_T^2 \| \fsf+\bar\usf_{\T}\|_{L^{\infty}(T)} \lesssim \| \bar{\ysf} - \bar{\ysf}_{\T}\|_{L^{\infty}(T)} + h_T^{2-\tfrac{n}{2}}\| \bar{\usf} - \bar{\usf}_{\T } \|_{L^2(T)} + h_T^2 \| \fsf-\mathcal{P}_{\T}\fsf\|_{L^{\infty}(T)}.
\]

\noindent \framebox{Step 2.} We estimate the term $h_T \| \llbracket \nu\cdot\nabla \bar{\ysf}_{\T} \rrbracket \|_{L^{\infty}(\partial T)}$ in \eqref{eq:defofEy}. We first invoke \eqref{eq:propertiesofetaS} and derive
$
|S| \| \llbracket \nu\cdot\nabla \bar{\ysf}_{\T} \rrbracket \|_{L^{\infty}(S)} \lesssim  \left| \int_{S} \llbracket \nu\cdot\nabla \bar{\ysf}_{\T} \rrbracket \varphi_S \right|.
$
In light of this estimate, we use identity \eqref{eq:res4} and control each term with the help of \eqref{eq:propertiesofetaS}. This yields
\begin{align*}
\left| \int_{S} \llbracket \nu\cdot\nabla \bar{\ysf}_{\T} \rrbracket  \varphi_S \right|  & \lesssim |T| \| \fsf+\bar{\usf}_{\T}\|_{L^{\infty}(\Ne_{S})} + |T|^{\tfrac{1}{2}} \|\bar{\usf} - \bar{\usf}_{\T} \|_{L^2(\Ne_S)}
\\
& + \left( h_{T}^{-2} |T| + h_{T}^{-1}\sum_{S'\in\Sides:\,S'\subset\Ne_{S}}|S'| \right) \| \bar{\ysf} - \bar{\ysf}_{\T}\|_{L^{\infty}(\Ne_{S})}.
\end{align*}
The collection of the estimates derived in Steps 1 and 2 concludes the proof.
\end{proof}

\AJS{
\begin{remark}[piecewise linear control] \rm
Notice that the only place where we use that $\Uad(\T)$ consists of piecewise constant functions is in Step 1 of the previous result. If we were to use, say, piecewise linear functions and suitably redefine $\calP_\T$ all that is necessary for our considerations to hold is to obtain a suitable bound on $\| \calP_\T \fsf + \bar\usf_\T\|_{L^\infty(\Omega)}$.
\end{remark}
}

We now proceed to analyze the local efficiency properties of the contribution $\E_{\psf}$ defined by \eqref{eq:defofEp}--\eqref{eq:defofEpg}. An important ingredient in the analysis that follows is the so-called \emph{residual} $\Res_{\psf} = \Res_{\psf}(\bar{\psf}_{\T}) \in H_0^1(\rho^{-1},\Omega)'$ defined by
\[
\langle \Res_{\psf} (\bar{\psf}_{\T}),\wsf \rangle := a(\wsf,\bar{\psf} - \bar{\psf}_{\T}) = \sum_{z \in \calZ} ( \bar{\ysf}(z) - \ysf_z) \delta_z(\wsf) - a(\wsf,\bar{\psf}_{\T}) \quad \forall \wsf \in H^1_0(\rho^{-1},\Omega),
\]
where $a$ is defined in \eqref{eq:defofforma}, $\bar \ysf$ solves \eqref{eq:statep} and $\bar\psf_{\T}$ solves \eqref{eq:adjointh}. We now explore an abstract estimate for the residual. Let $\Orara$ be a subdomain of $\Omega$ and $\wsf \in H_0^1(\rho^{-1},\Orara)$. An argument based on \cite[Remark 6.1]{MR2238754} and \cite[Section 5.2]{AGM} provides the estimate
\[
 | \langle \Res_{\psf} (\bar{\psf}_{\T}),\wsf \rangle | = |a(\wsf,\bar \psf - \bar{\psf}_{\T})| \leq \| \nabla ( \bar \psf - \bar{\psf}_{\T}) \|_{L^2(\rho,\Orara)}
  \| \nabla \wsf \|_{L^2(\rho^{-1},\Orara)}.
\]
Consequently,
\begin{equation}
\| \Res_{\psf} \|_{H_0^1(\rho^{-1},\Orara)'} \leq \| \nabla ( \bar \psf - \bar{\psf}_{\T}) \|_{L^2(\rho,\Orara)}.
\label{eq:estimate_Resp}
\end{equation}

We now utilize the residual estimation techniques developed in \cite{AGM}. We start by introducing a smooth function $\psi_S$ whose construction we owe to \cite{AGM}. In this setting, $\psi_S$ plays an analogous role to the one of the smooth function $\varphi_S$ defined by \eqref{eq:propertiesofetaS}. Given $S \in \Sides$, we recall that $\Ne_S = T \cup T'$ denotes the patch composed of the two elements of $\T$ sharing $S$. The construction of $\psi_S$ is as follows: we divide each edge of $T$ and $T'$ into four equal segments. The vertices of $S$ and the segments having a vertex in $S$ determine $n$ patches of adjacent simplices, which we denote by $\{\calP_i\}_{i=1}^n$. We now choose $\calP_0 \in \{\calP_i\}_{i=1}^n$ as the patch that is further away from $\calZ$. Notice that, first of all, $\calP_0 = T_{*} \cup T'_{*}$ with $T_{*} \subset T$ and $T'_{*} \subset T'$, respectively (see Figure \ref{Fig:bubble}). Moreover, owing to assumption \eqref{eq:patch}, we have that
\RR{\[
  D_T \lesssim \min_{z \in \calZ} \left\{ \min_{x \in T_*} |x-z| \right\}
  \mbox{ and }
    D_T \lesssim \min_{z \in \calZ} \left\{ \min_{x \in T'_*} |x-z| \right\}.
\]}
Scaling and translation arguments applied to a standard bubble function yield the existence of a smooth function $\psi_S$ with the properties
\begin{equation}
 \label{eq:propertiesofetaSdelta}
 \begin{dcases}
    \psi_S(z) = 0, \quad \forall z \in \calZ, & |S| \lesssim \int_{S} \psi_{S},  \\ 
    \supp  \psi_{S} \subset T_* \cup T'_* \subset \Ne_S,  &  
    \|\nabla^{k} \psi_{S} \|_{L^{\infty}(\Ne_S)} \lesssim h_{S}^{-k}, \ k=0,1,2,
 \end{dcases}
\end{equation}
see \cite{AGM} for details.
After these preparations we are ready to derive the local efficiency of $\E_{\psf}$.
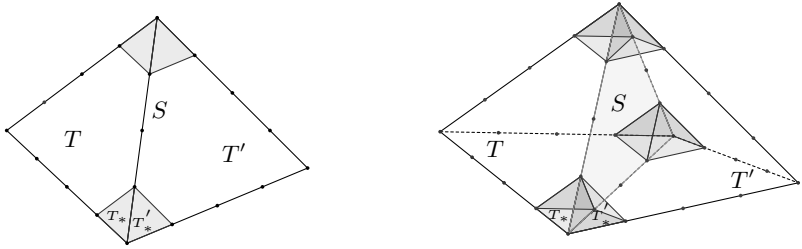
\begin{figure}[!h]
\begin{center}
\definecolor{tttttt}{rgb}{0.2,0.2,0.2}
\definecolor{uququq}{rgb}{0.25,0.25,0.25}
\begin{tikzpicture}[line cap=round,line join=round,>=triangle 45,x=.2cm,y=.25cm]
\clip(0.37,1.68) rectangle (25.13,14.14);
\fill[color=tttttt,fill=tttttt,fill opacity=0.1] (8,3.5) -- (10,2) -- (10.5,5) -- cycle;
\fill[color=tttttt,fill=tttttt,fill opacity=0.1] (10,2) -- (13,3) -- (10.5,5) -- cycle;
\fill[color=tttttt,fill=tttttt,fill opacity=0.1] (9.5,12.5) -- (11.5,11) -- (12,14) -- cycle;
\fill[color=tttttt,fill=tttttt,fill opacity=0.1] (11.5,11) -- (14.5,12) -- (12,14) -- cycle;
\draw (2,8)-- (12,14);
\draw (12,14)-- (10,2);
\draw (10,2)-- (2,8);
\draw (10,2)-- (22,6);
\draw (22,6)-- (12,14);
\draw [color=tttttt] (8,3.5)-- (10,2);
\draw [color=tttttt] (10,2)-- (10.5,5);
\draw [color=tttttt] (10.5,5)-- (8,3.5);
\draw [color=tttttt] (10,2)-- (13,3);
\draw [color=tttttt] (13,3)-- (10.5,5);
\draw [color=tttttt] (10.5,5)-- (10,2);
\draw [color=tttttt] (9.5,12.5)-- (11.5,11);
\draw [color=tttttt] (11.5,11)-- (12,14);
\draw [color=tttttt] (12,14)-- (9.5,12.5);
\draw [color=tttttt] (11.5,11)-- (14.5,12);
\draw [color=tttttt] (14.5,12)-- (12,14);
\draw [color=tttttt] (12,14)-- (11.5,11);
\draw (5.2,8.5) node[anchor=north west] {\small $T$};
\draw (15.7,7.8) node[anchor=north west] {\small $T'$};
\draw (11,10) node[anchor=north west] {\small $S$};
\draw (8.0,4.2) node[anchor=north west] {\tiny $T_{*}$};
\draw (9.7,4.3) node[anchor=north west] {\tiny $T_{*}^{'}$};
\begin{scriptsize}
\fill [color=black] (2,8) circle (0.7pt);
\fill [color=black] (12,14) circle (0.7pt);
\fill [color=black] (10,2) circle (0.7pt);
\fill [color=black] (7,11) circle (0.7pt);
\fill [color=uququq] (4.5,9.5) circle (0.7pt);
\fill [color=black] (9.5,12.5) circle (0.7pt);
\fill [color=black] (6,5) circle (0.7pt);
\fill [color=black] (4,6.5) circle (0.7pt);
\fill [color=black] (8,3.5) circle (0.7pt);
\fill [color=black] (11,8) circle (0.7pt);
\fill [color=black] (10.5,5) circle (0.7pt);
\fill [color=black] (11.5,11) circle (0.7pt);
\fill [color=black] (22,6) circle (0.7pt);
\fill [color=black] (16,4) circle (0.7pt);
\fill [color=black] (13,3) circle (0.7pt);
\fill [color=black] (19,5) circle (0.7pt);
\fill [color=black] (17,10) circle (0.7pt);
\fill [color=black] (19.5,8) circle (0.7pt);
\fill [color=black] (14.5,12) circle (0.7pt);
\end{scriptsize}
\end{tikzpicture}
\definecolor{zzzzzz}{rgb}{0.6,0.6,0.6}
\definecolor{uququq}{rgb}{0.25,0.25,0.25}
\definecolor{tttttt}{rgb}{0.2,0.2,0.2}
\begin{tikzpicture}[line cap=round,line join=round,>=triangle 45,x=0.34cm,y=0.34cm]
\clip(0.04,0.41) rectangle (19.56,10.5);
\fill[color=tttttt,fill=tttttt,fill opacity=0.1] (11.12,4.83) -- (10.59,6.12) -- (8.84,4.87) -- cycle;
\fill[color=tttttt,fill=tttttt,fill opacity=0.1] (10.59,6.12) -- (11.12,4.83) -- (12.34,4.37) -- cycle;
\fill[color=tttttt,fill=tttttt,fill opacity=0.1] (8.84,4.87) -- (10.09,3.87) -- (10.59,6.12) -- cycle;
\fill[color=tttttt,fill=tttttt,fill opacity=0.1] (10.09,3.87) -- (12.34,4.37) -- (10.59,6.12) -- cycle;
\fill[color=tttttt,fill=tttttt,fill opacity=0.1] (7.25,8.75) -- (9.53,8.71) -- (9,10) -- cycle;
\fill[color=tttttt,fill=tttttt,fill opacity=0.1] (7.25,8.75) -- (8.5,7.75) -- (9,10) -- cycle;
\fill[color=tttttt,fill=tttttt,fill opacity=0.1] (8.5,7.75) -- (9.53,8.71) -- (9,10) -- cycle;
\fill[color=tttttt,fill=tttttt,fill opacity=0.1] (8.5,7.75) -- (10.75,8.25) -- (9,10) -- cycle;
\fill[color=tttttt,fill=tttttt,fill opacity=0.1] (9.53,8.71) -- (10.75,8.25) -- (9,10) -- cycle;
\fill[color=tttttt,fill=tttttt,fill opacity=0.1] (5.75,2) -- (8.03,1.96) -- (7.5,3.25) -- cycle;
\fill[color=tttttt,fill=tttttt,fill opacity=0.1] (5.75,2) -- (7,1) -- (7.5,3.25) -- cycle;
\fill[color=tttttt,fill=tttttt,fill opacity=0.1] (7,1) -- (9.25,1.5) -- (7.5,3.25) -- cycle;
\fill[color=tttttt,fill=tttttt,fill opacity=0.1] (8.03,1.96) -- (9.25,1.5) -- (7,1) -- cycle;
\fill[color=tttttt,fill=tttttt,fill opacity=0.1] (7,1) -- (8.03,1.96) -- (7.5,3.25) -- cycle;
\fill[color=zzzzzz,fill=zzzzzz,fill opacity=0.1] (7,1) -- (11.12,4.83) -- (9,10) -- cycle;
\draw (2,5)-- (7,1);
\draw [dash pattern=on 1pt off 1pt] (7,1)-- (11.12,4.83);
\draw [dash pattern=on 1pt off 1pt] (11.12,4.83)-- (9,10);
\draw (9,10)-- (7,1);
\draw (9,10)-- (2,5);
\draw [dash pattern=on 1pt off 1pt] (2,5)-- (11.12,4.83);
\draw [dash pattern=on 1pt off 1pt] (11.12,4.83)-- (16,3);
\draw (16,3)-- (7,1);
\draw (9,10)-- (16,3);
\draw [color=tttttt] (11.12,4.83)-- (10.59,6.12);
\draw [color=tttttt] (10.59,6.12)-- (8.84,4.87);
\draw [color=tttttt] (8.84,4.87)-- (11.12,4.83);
\draw [color=tttttt] (10.59,6.12)-- (11.12,4.83);
\draw [color=tttttt] (11.12,4.83)-- (12.34,4.37);
\draw [color=tttttt] (12.34,4.37)-- (10.59,6.12);
\draw [color=tttttt] (8.84,4.87)-- (10.09,3.87);
\draw [color=tttttt] (10.09,3.87)-- (10.59,6.12);
\draw [color=tttttt] (10.59,6.12)-- (8.84,4.87);
\draw [color=tttttt] (10.09,3.87)-- (12.34,4.37);
\draw [color=tttttt] (12.34,4.37)-- (10.59,6.12);
\draw [color=tttttt] (10.59,6.12)-- (10.09,3.87);
\draw [color=tttttt] (7.25,8.75)-- (9.53,8.71);
\draw [color=tttttt] (9.53,8.71)-- (9,10);
\draw [color=tttttt] (9,10)-- (7.25,8.75);
\draw [color=tttttt] (7.25,8.75)-- (8.5,7.75);
\draw [color=tttttt] (8.5,7.75)-- (9,10);
\draw [color=tttttt] (9,10)-- (7.25,8.75);
\draw [color=tttttt] (8.5,7.75)-- (9.53,8.71);
\draw [color=tttttt] (9.53,8.71)-- (9,10);
\draw [color=tttttt] (9,10)-- (8.5,7.75);
\draw [color=tttttt] (8.5,7.75)-- (10.75,8.25);
\draw [color=tttttt] (10.75,8.25)-- (9,10);
\draw [color=tttttt] (9,10)-- (8.5,7.75);
\draw [color=tttttt] (9.53,8.71)-- (10.75,8.25);
\draw [color=tttttt] (10.75,8.25)-- (9,10);
\draw [color=tttttt] (9,10)-- (9.53,8.71);
\draw (3.45,5) node[anchor=north west] {\small $T$};
\draw (13.0,3.9) node[anchor=north west] {\small $T'$};
\draw (8.3,6.82) node[anchor=north west] {\small $S$};
\draw (5.85,2.3) node[anchor=north west] {\tiny $T_{*}$};
\draw (7.5,2.6) node[anchor=north west] {\tiny $T_{*}^{'}$};
\draw [color=tttttt] (5.75,2)-- (8.03,1.96);
\draw [color=tttttt] (8.03,1.96)-- (7.5,3.25);
\draw [color=tttttt] (7.5,3.25)-- (5.75,2);
\draw [color=tttttt] (5.75,2)-- (7,1);
\draw [color=tttttt] (7,1)-- (7.5,3.25);
\draw [color=tttttt] (7.5,3.25)-- (5.75,2);
\draw [color=tttttt] (7,1)-- (9.25,1.5);
\draw [color=tttttt] (9.25,1.5)-- (7.5,3.25);
\draw [color=tttttt] (7.5,3.25)-- (7,1);
\draw [color=tttttt] (8.03,1.96)-- (9.25,1.5);
\draw [color=tttttt] (9.25,1.5)-- (7,1);
\draw [color=tttttt] (7,1)-- (8.03,1.96);
\draw [color=tttttt] (7,1)-- (8.03,1.96);
\draw [color=tttttt] (8.03,1.96)-- (7.5,3.25);
\draw [color=tttttt] (7.5,3.25)-- (7,1);
\draw [color=zzzzzz] (7,1)-- (11.12,4.83);
\draw [color=zzzzzz] (11.12,4.83)-- (9,10);
\draw [color=zzzzzz] (9,10)-- (7,1);
\begin{scriptsize}
\fill [color=tttttt] (2,5) circle (0.7pt);
\fill [color=tttttt] (7,1) circle (0.7pt);
\fill [color=tttttt] (11.12,4.83) circle (0.7pt);
\fill [color=tttttt] (9,10) circle (0.7pt);
\fill [color=tttttt] (16,3) circle (0.7pt);
\fill [color=tttttt] (4.5,3) circle (0.7pt);
\fill [color=uququq] (5.75,2) circle (0.7pt);
\fill [color=tttttt] (3.25,4) circle (0.7pt);
\fill [color=uququq] (6.56,4.91) circle (0.7pt);
\fill [color=tttttt] (4.28,4.96) circle (0.7pt);
\fill [color=uququq] (8.84,4.87) circle (0.7pt);
\fill [color=uququq] (10.06,7.41) circle (0.7pt);
\fill [color=uququq] (9.53,8.71) circle (0.7pt);
\fill [color=uququq] (10.59,6.12) circle (0.7pt);
\fill [color=uququq] (12.5,6.5) circle (0.7pt);
\fill [color=uququq] (14.25,4.75) circle (0.7pt);
\fill [color=uququq] (10.75,8.25) circle (0.7pt);
\fill [color=uququq] (11.5,2) circle (0.7pt);
\fill [color=uququq] (13.75,2.5) circle (0.7pt);
\fill [color=tttttt] (9.25,1.5) circle (0.7pt);
\fill [color=uququq] (9.06,2.91) circle (0.7pt);
\fill [color=uququq] (8.03,1.96) circle (0.7pt);
\fill [color=uququq] (10.09,3.87) circle (0.7pt);
\fill [color=uququq] (8,5.5) circle (0.7pt);
\fill [color=uququq] (8.5,7.75) circle (0.7pt);
\fill [color=uququq] (5.5,7.5) circle (0.7pt);
\fill [color=uququq] (3.75,6.25) circle (0.7pt);
\fill [color=uququq] (7.25,8.75) circle (0.7pt);
\fill [color=uququq] (13.56,3.91) circle (0.7pt);
\fill [color=uququq] (14.78,3.46) circle (0.7pt);
\fill [color=uququq] (12.34,4.37) circle (0.7pt);
\fill [color=uququq] (7.5,3.25) circle (0.7pt);
\end{scriptsize}
\end{tikzpicture}
\end{center}
\caption{\AAF{Simplices $T$ and $T'$, that share the side $S$, with their sub-simplices $T_{*}$ and $T_{*}^{'}$, respectively, obtained after dividing their edges into four equal segments, in two dimensions (left) and three dimensions (right).}}
\label{Fig:bubble}
\end{figure}
\begin{lemma}[local efficiency of $\E_{\psf}$] 
Let $(\bar{\usf},\bar{\ysf},\bar{\psf}) \in L^2(\Omega) \times H_0^1(\Omega) \times H_0^1(\rho,\Omega)$ be the solution to the optimality system \eqref{eq:statep}--\eqref{eq:adjp2} and $(\bar{\usf}_{\T},\bar{\ysf}_{\T},\bar{\psf}_{\T}) \in  \U_{\textrm{ad}}(\T) \times \V(\T) \times \V(\T)$ be its numerical approximation defined \AJS{by \eqref{eq:abstateh}--\eqref{eq:adjointh}. If $\alpha \in \mathbf{I}$}, then
\begin{equation}
\E_{\psf} (\bar{\psf}_{\T},\bar{\ysf}_{\T}; T) \lesssim \| \nabla( \bar{\psf} - \bar{\psf}_{\T}) \|_{L^{2}(\rho,\Ne_T)} 
+ \#\left(T \cap \calZ \right) h_T^{ \frac{\alpha}{2} + 1 - \frac{n}{2} } \|  \bar{\ysf} - \bar{\ysf}_{\T} \|_{L^{\infty}(T)},
\label{eq:efficiencyp}
\end{equation}
where the hidden constant is independent of the optimal variables, their approximations, the size of the elements in the mesh $\T$ and $\#\T$. 
\end{lemma}
\begin{proof}
Let $T \in \T$ and $S \in \Sides$ be a side of $T$. We will estimate each one of the terms that comprise definition \eqref{eq:defofEp} separately.

We begin with the term $h_T D_T^{\alpha} \| \llbracket \nu\cdot\nabla \bar{\psf}_{\T} \rrbracket\|^2_{L^2(\partial T\setminus \partial \Omega)}$. If $\psi_S$ denotes the smooth function defined by \eqref{eq:propertiesofetaSdelta}, then 
\[
 \| \llbracket \nu\cdot\nabla \bar{\psf}_{\T} \rrbracket \|^2_{L^2(S)} \lesssim \int_{S} \llbracket \nu\cdot\nabla \bar{\psf}_{\T} \rrbracket^2 \psi_S = \int_{S} \llbracket \nu\cdot\nabla \bar{\psf}_{\T} \rrbracket \phi_S,
\]
where $\phi_S = \llbracket \nu\cdot\nabla \bar{\psf}_{\T} \rrbracket \psi_S$. Since $\phi_S(z) = 0$ for all $z \in \calZ$ and $\supp \phi_{S} \subset T_* \cup T'_* \subset \Ne_S$, we invoke the definition of $\Res_{\psf}$ and estimate the right hand side of the previous identity as follows:
\begin{align*}
\int_{S} \llbracket \nu\cdot \nabla \bar{\psf}_{\T} \rrbracket \phi_S = \langle \Res_{\psf}, \phi_S \rangle & \leq \| \Res_{\psf} \|_{H_0^1(\rho^{-1},\Ne_S)'} \| \nabla \phi_S \|_{L^2(\rho^{-1},\Ne_S)} 
 \\
 &\lesssim 
 h_T^{-1/2} D_T^{-\frac{\alpha}{2}} \| \nabla ( \bar \psf - \bar{\psf}_{\T}) \|_{L^2(\rho,\Ne_S)} \| \llbracket \nu\cdot\nabla \bar{\psf}_{\T} \rrbracket \|_{L^2(S)},
\end{align*}
where we have used \eqref{eq:estimate_Resp} and \cite[equation (5.9)]{AGM}. The previous computations yield the estimate
\begin{equation}
\label{eq:Ep_first}
h_T D_T^{\alpha} \| \llbracket \nu \cdot \nabla \bar{\psf}_{\T} \rrbracket\|^2_{L^2(S)} \lesssim \| \nabla ( \bar \psf - \bar{\psf}_{\T}) \|^2_{L^2(\rho,\Ne_S)}.
\end{equation}

From \eqref{eq:patch} it follows that $T \cap \calZ$ is either empty or consists of exactly one point. If $T \cap \calZ = \emptyset$, then estimate \eqref{eq:Ep_first} immediately yields \eqref{eq:efficiencyp}. If $T \cap \calZ = \{ z \}$, then the local indicator $\E^2_{\psf}$, defined in \eqref{eq:defofEp}, contains the term
\[
h_T^{\alpha + 2 - n} |\bar{\ysf}_{\T}(z) - \ysf_z|^2.
\]
To control this term, we follow the arguments developed in the proof of \cite[Theorem 5.3]{AGM}, which yield the existence of a smooth function $\chi$ such that 
\begin{equation}
\label{eq:chi}
\chi(z) = 1, \quad \| \chi \|_{L^{\infty}(\Omega)} = 1, \quad \| \nabla \chi \|_{L^{\infty}(\Omega)} = h_T^{-1}, \quad \supp \chi \subset \Ne_T.
\end{equation}
In addition, $\| \chi \|_{L^2(S)} \lesssim h_{T}^{\frac{n-1}{2}}$ and $\| \nabla \chi \|_{L^2(\rho^{-1},\Ne_{T})} \lesssim h_T^{\frac{n-2}{2} - \frac{\alpha}{2}}$. With these elements we proceed to control $h_T^{\alpha + 2 - n}|\bar{\ysf}_{\T}(z) - \ysf_z|^2$ as follows. We start with a simple application of the triangle inequality: 
\begin{equation}
\label{eq:aux1}
h_T^{\frac{\alpha}{2} + 1 - \frac{n}{2}}| \bar{\ysf}_{\T}(z) - \ysf_z | \leq 
h_T^{\frac{\alpha}{2} + 1 - \frac{n}{2}} | \bar{\ysf}_{\T}(z) - \bar{\ysf}(z) | + 
h_T^{\frac{\alpha}{2} + 1 - \frac{n}{2}}| \bar{\ysf}(z) - \ysf_z|. 
\end{equation}
The first term is bounded by $h_T^{\frac{\alpha}{2} + 1 - \frac{n}{2}} \|  \bar{\ysf} - \bar{\ysf}_{\T} \|_{L^{\infty}(T)}$. To control the second term we 
employ that $\chi(z) = 1$, $\supp \chi \subset \Ne_T$ and integration by parts. Since $\bar{\psf}$ solves \eqref{eq:adjp2}, we obtain
\begin{align*}
 | \bar{\ysf}(z) - \ysf_z | & = |\left(\bar{\ysf}(z) - \ysf_z\right)\chi(z)| = |a(\chi,\bar{\psf})| \leq  |a(\chi,\bar{\psf} - \bar{\psf}_{\T})| + |a(\chi,\bar{\psf}_{\T})|
 \\
 & \leq \|\nabla(\bar{\psf} - \bar{\psf}_{\T}) \|_{L^2(\rho,\Ne_{T})} \| \nabla \chi \|_{L^2(\rho^{-1},\Ne_{T})} 
\\ 
& + \sum_{S\in\Sides:\,S \subset \Ne_T\setminus\partial\Ne_T} \| \llbracket \nu\cdot \nabla \bar{\psf}_{\T} \rrbracket \|_{L^2(S)} \| \chi \|_{L^2(S)}.
\end{align*}
The previous estimate, combined with the properties satisfied by $\chi$, allows us to derive
\begin{equation}
\label{eq:aux3}
  \begin{aligned}
    | \bar{\ysf}(z) - \ysf_z |  &\lesssim h_T^{\frac{n-2}{2} - \frac{\alpha}{2}} \|\nabla(\bar{\psf} - \bar{\psf}_{\T}) \|_{L^2(\rho,\Ne_{T})}  
    + \sum_{S\in\Sides:\,S \subset \Ne_T\setminus\partial\Ne_T} h_{T}^{\frac{n-1}{2}} \| \llbracket \nu\cdot\nabla \bar{\psf}_{\T} \rrbracket\|_{L^2(S)} \\
    & \lesssim h_T^{\frac{n-2}{2} - \frac{\alpha}{2}} \Bigg{(}
      \|\nabla(\bar{\psf} - \bar{\psf}_{\T}) \|_{L^2(\rho,\Ne_{T})} \\
    &+ \sum_{T'\in\T:\,T'\subset\Ne_T}\sum_{\RR{S\in\Sides_{T'}:\,S \not\subset\partial\Ne_T}} h_{T'}^{1/2} D_{T'}^{\alpha/2} \| \llbracket \nu\cdot\nabla \bar{\psf}_{\T} \rrbracket\|_{L^2(S)}
    \Bigg{)}.
  \end{aligned}
\end{equation}
In light of \eqref{eq:aux3}, equation \eqref{eq:aux1} yields
\begin{align*}
 h_T^{\alpha + 2 - n}|\bar{\ysf}_{\T}(z) - \ysf_z|^2  & \lesssim h_T^{ \alpha + 2 - n } \|  \bar{\ysf} - \bar{\ysf}_{\T} \|_{L^{\infty}(T)}^2 + \|\nabla(\bar{\psf} - \bar{\psf}_{\T}) \|_{L^2(\rho,\Ne_{T})}^2
 \\
& + \sum_{T'\in\T:\,T'\subset\Ne_T}\sum_{\RR{S\in\Sides_{T'}:\,S \not\subset\partial\Ne_T}} h_{T'} D_{T'}^{\alpha} \| \llbracket \nu \cdot \nabla \bar{\psf}_{\T} \rrbracket\|^2_{L^2(S)},
\end{align*}
which, via an application of \eqref{eq:Ep_first}, yields \eqref{eq:efficiencyp} and concludes the proof.
\end{proof}

\begin{remark}[range of $\alpha$] \rm
\label{rk:alpha}
Notice that, since $\alpha \in \textbf{I}$, we have that
\[
  \alpha + 2 -n >0,
\]
so that \eqref{eq:efficiencyp} is indeed an efficiency bound. 
\end{remark}

We now conclude with the global efficiency of the error estimator $\E_{\textrm{ocp}}$ defined in \eqref{eq:defofEocp}. To derive such a result, we define
\[
 \textrm{osc}_{\T}(\fsf,\T) = \max_{T \in \T} h_{T}^2\| \fsf - \mathcal{P}_{\T} \fsf \|_{L^{\infty}(T)}.
\]

\begin{theorem}[global efficiency of $\E_{\textrm{ocp}}$] 
\RR{Let} $(\bar{\usf},\bar{\ysf},\bar{\psf}) \in L^2(\Omega) \times H_0^1(\Omega) \times H_0^1(\rho,\Omega)$ \RR{be} the solution to the optimality system \eqref{eq:statep}--\eqref{eq:adjp2} and $(\bar{\usf}_{\T},\bar{\ysf}_{\T},\bar{\psf}_{\T}) \in  \U_{\textrm{ad}}(\T) \times \V(\T) \times \V(\T)$ \RR{be} its numerical approximation defined \AJS{by \eqref{eq:abstateh}--\eqref{eq:adjointh}. If $\alpha \in (n-2,2)$}, then 
\begin{equation}
\label{eq:efficiencyglobal}
\begin{aligned}
\E_{\textrm{ocp}} (\bar{\ysf}_{\T},\bar{\psf}_{\T},\bar{\usf}_{\T}; \T) & \lesssim \|  \bar{\ysf} - \bar{\ysf}_{\T} \|_{L^{\infty}(\Omega)} + \| \nabla( \bar{\psf} - \bar{\psf}_{\T}) \|_{L^{2}(\rho,\Omega)} 
\\
& + \| \bar{\usf} - \bar{\usf}_{\T}\|_{L^2(\Omega)} + \mathrm{osc}_{\T}(\fsf,\T),
\end{aligned}
\end{equation}
where the hidden constant is independent of the size of the elements in the mesh $\T$ and $\#\T$ but depends on linearly on $\sqrt{\#\calZ}$ and $\diam(\Omega)^{\frac\alpha2+1-\frac{n}2}$.
\end{theorem}
\begin{proof}
We start invoking the definition of the global pointwise indicator $\E_{\ysf}$ given by \eqref{eq:defofEyg} and the local efficiency estimate \eqref{eq:efficiencyy} to arrive at
\[
 \E_{\ysf}(\bar{\ysf}_{\T},\bar{\usf}_{\T}; \T) \lesssim \| \bar{\ysf} -\bar{\ysf}_{\T} \|_{L^{\infty}(\Omega)} +  \textrm{diam}(\Omega)^{2-n/2} \|\bar{\usf} -\bar{\usf}_{\T}  \|_{L^2(\Omega)} +  \textrm{osc}_{\T}(\fsf,\T).
\]
Now, in view of \eqref{eq:defofEpg}, the local efficiency estimate \eqref{eq:efficiencyp} provides the bound
\[
 \E_{\psf}(\bar{\psf}_{\T},\bar{\ysf}_{\T}; \T) \lesssim \| \nabla( \bar{\psf} -\bar{\psf}_{\T}) \|_{L^{2}(\rho,\Omega)} + 
 \left( \sum_{T \in \T: T\cap\calZ \neq \emptyset } h_T^{ \alpha+ 2 - n } \right)^{\frac{1}{2}}\| \bar{\ysf} -\bar{\ysf}_{\T} \|_{L^{\infty}(\Omega)}. 
\]
Remark~\ref{rk:alpha} and the fact that $\# \calZ < \infty$ imply the bound
\[
  \left( \sum_{T \in \T: T\cap\calZ \neq \emptyset}  h_T^{ \alpha+ 2 - n } \right)^{\frac12} \leq \sqrt{\# \calZ} \diam(\Omega)^{\frac\alpha2+1-\frac{n}2},
\]
which, firstly, is independent of  $\#\T$ and, secondly, is where the linear dependence on $\sqrt{\# \calZ}$ and $\diam(\Omega)^{\frac\alpha2+1-\frac{n}2}$ comes from. Finally, a trivial application of the triangle inequality yields
\[
 \E_{\usf}(\bar{\usf}_{\T},\bar{\psf}_{\T}; \T) \leq \| \bar \usf_{\T} - \Pi(-\tfrac{1}{\lambda}\bar \psf) \|_{L^2(\Omega)} + \| \Pi(-\tfrac{1}{\lambda}\bar \psf) - \Pi(-\tfrac{1}{\lambda}\bar \psf_{\T}) \|_{L^2(\Omega)},
\]
where $\Pi$ is defined in \eqref{def:Pi}. This, in conjunction with the Lipschitz continuity of $\Pi$ and \RR{Lemma~\ref{lem:restalphagen}}, implies 
\[
 \E_{\usf}(\bar{\usf}_{\T},\bar{\psf}_{\T}; \T) \leq \| \bar \usf_{\T} - \bar{\usf} \|_{L^2(\Omega)} + \| \nabla( \bar \psf - \bar \psf_{\T} ) \|_{L^2(\rho,\Omega)}.
\]
Gathering all the obtained estimates concludes the proof.
\end{proof}

\section{Numerical examples}
\label{sec:numex}
In this section we conduct a series of numerical examples that illustrate the performance of the error estimator, even \AAF{when} we violate \AAF{the assumption of} homogeneous Dirichlet boundary condition\RR{s}. These have been carried out with the help of a code that we implemented using \texttt{C++}. All matrices have been assembled exactly. The right hand sides and approximation errors are computed by a quadrature formula which is exact for polynomials of degree 19 for two dimensional domains and degree 14 for three dimensional domains. All linear systems were solved using the multifrontal massively parallel sparse direct solver (MUMPS) \cite{MUMPS1,NUMPS2}. 

For a given partition $\T$ we seek $(\bar{\ysf}_\T,\bar{\psf}_\T,\bar{\usf}_\T)\in \V(\T)\times \V(\T)\times \Uad(\T)$ that solves \eqref{eq:abstateh}--\eqref{eq:adjointh}. We solve the ensuing nonlinear system of equations using a Newton-type primal-dual active set strategy \cite[\S 2.12.4]{Tbook}. Once a discrete solution is obtained, we use the error indicator \RR{
\begin{equation}\label{eq:totalInd}
\E_{\mathsf{ocp};T}^{2}:=
\E_{\ysf}^{2}(\bar{\ysf}_{\T},\bar{\usf}_{\T};T)+
\E_{\psf}^{2}(\bar{\psf}_{\T},\bar{\ysf}_{\T};T)+
\E_{\usf}^{2}(\bar{\usf}_{\T},\bar{\psf}_{\T};T),
\end{equation}}
which is defined in terms of \eqref{eq:defofEy}, \eqref{eq:defofEp} and \eqref{eq:defofEu}, to drive the adaptive procedure described in \textbf{Algorithm 1}. For the numerical results, we define the total number of degrees of freedom $\Ndof = 2 \dim(\V(\T)) + \dim(\U(\T))$, and to assess the accuracy of the approximation, the error is measured in the norm\RR{
\[
  \|(e_{\ysf},e_{\psf},e_{\usf})\|_{\Omega}^{2} = \|e_\ysf\|_{L^{\infty}(\Omega)}^{2}+
  \|\nabla e_\psf\|_{L^{2}(\rho,\Omega)}^{2}+
  \|e_\usf\|_{L^{2}(\Omega)}^{2}
\]
where $e_\ysf=\bar{\ysf}-\bar{\ysf}_{\T}$, $e_\psf=\bar{\psf}-\bar{\psf}_{\T}$ and $e_\usf=\bar{\usf}-\bar{\usf}_{\T}$.}
\AAF{For all our numerical examples we fix $\lambda=1$, and consider problems with homogeneous boundary conditions whose exact solutions are not known, and problems where we violate such a requirement by constructing exact solutions by taking the optimal adjoint to be}
\begin{equation}
\label{exact_adjoint}
\bar{\psf}(x) = \begin{dcases}
                  -\frac{1}{2\pi}\sum_{z\in\calZ}\log|x-z|, & \textrm{if}~\Omega\subset\mathbb{R}^{2},\\
                  \frac{1}{4\pi}\sum_{z\in\calZ}\frac{1}{|x-z|}, & \textrm{if}~\Omega\subset\mathbb{R}^{3},
                \end{dcases}
\end{equation}
\AAF{and fixing an optimal state, from which the right hand side $\fsf$ can be computed accordingly.}

The initial meshes for our numerical examples are shown in Figure~\ref{Fig0}.

\begin{table}[!h]
\begin{flushleft}
\scalebox{.9}
{
\begin{tabular}{l l} 
\multicolumn{2}{l}{\textbf{Algorithm 1:  Adaptive Primal-Dual Active Set Algorithm.}} 
\vspace{0.15cm}\\
\toprule
\multicolumn{2}{l}{\textbf{Input:} Initial mesh $\T_{0}$, set of observation points $\calZ$,  set of desired point states $\{\ysf_{z}\}_{z\in\calZ}$,}\\
\multicolumn{2}{l}{ box-constrain\RR{t} constants $\asf<\bsf$, control cost $\lambda$ and external force $\fsf$.}
\vspace{0.1cm} \\
\textbf{Set:}  & $i=0$.
\vspace{0.1cm}\\
\multicolumn{2}{l}{\textbf{Active set strategy:}}
\vspace{0.1cm}\\
\textbf{1:}    &  Choose initial guesses for the control $\usf_\T^0\in \U(\T)$ and $\mu_\T^0 \in \U(\T)$.\\
               &  Compute $[\bar{\ysf}_{\T},\bar{\psf}_{\T},\bar{\usf}_{\T}]=\textrm{\textbf{Active-Set}}[\T_{i},\usf_\T^0,\mu_\T^{0},\lambda,\asf,\bsf,\fsf,\ysf_{z},\calZ]$. \\
               &  $\textrm{\textbf{Active-Set}}$ implements the active set strategy of \cite[\S 2.12.4]{Tbook}. \\
\multicolumn{2}{l}{\textbf{Adaptive loop:}}
\vspace{0.1cm}\\
\textbf{2:}    &  For each $T \in \T$ compute the local error indicator $\E_{\mathsf{ocp};T}$ given in \eqref{eq:totalInd}.
 \\
\textbf{3:}    & Mark an element $T$ for refinement using either a:\\
& - Maximum strategy: $\displaystyle\E_{\mathsf{ocp};T}^{2}> 0.5\max_{T'\in\T}\E_{\mathsf{ocp};T'}^{2}$.\\
& - Bulk criterion: see \cite{MR1393904}.\\
& - Average strategy: $\displaystyle\E_{\mathsf{ocp};T}^{2}\geq \frac{1}{\#\T}\sum_{T'\in\T}\E_{\mathsf{ocp};T'}^{2}$.\\
\textbf{4:}    & From step \textbf{3}, construct a new mesh, using a longest edge bisection algorithm. \\& Set $i\leftarrow i+1$, and go to step \textbf{1}.\\
\bottomrule
\end{tabular}}
\vspace{-0.3cm}
\end{flushleft}
\end{table}

\begin{figure}[!h]
\begin{center}
\scalebox{0.18}{\includegraphics{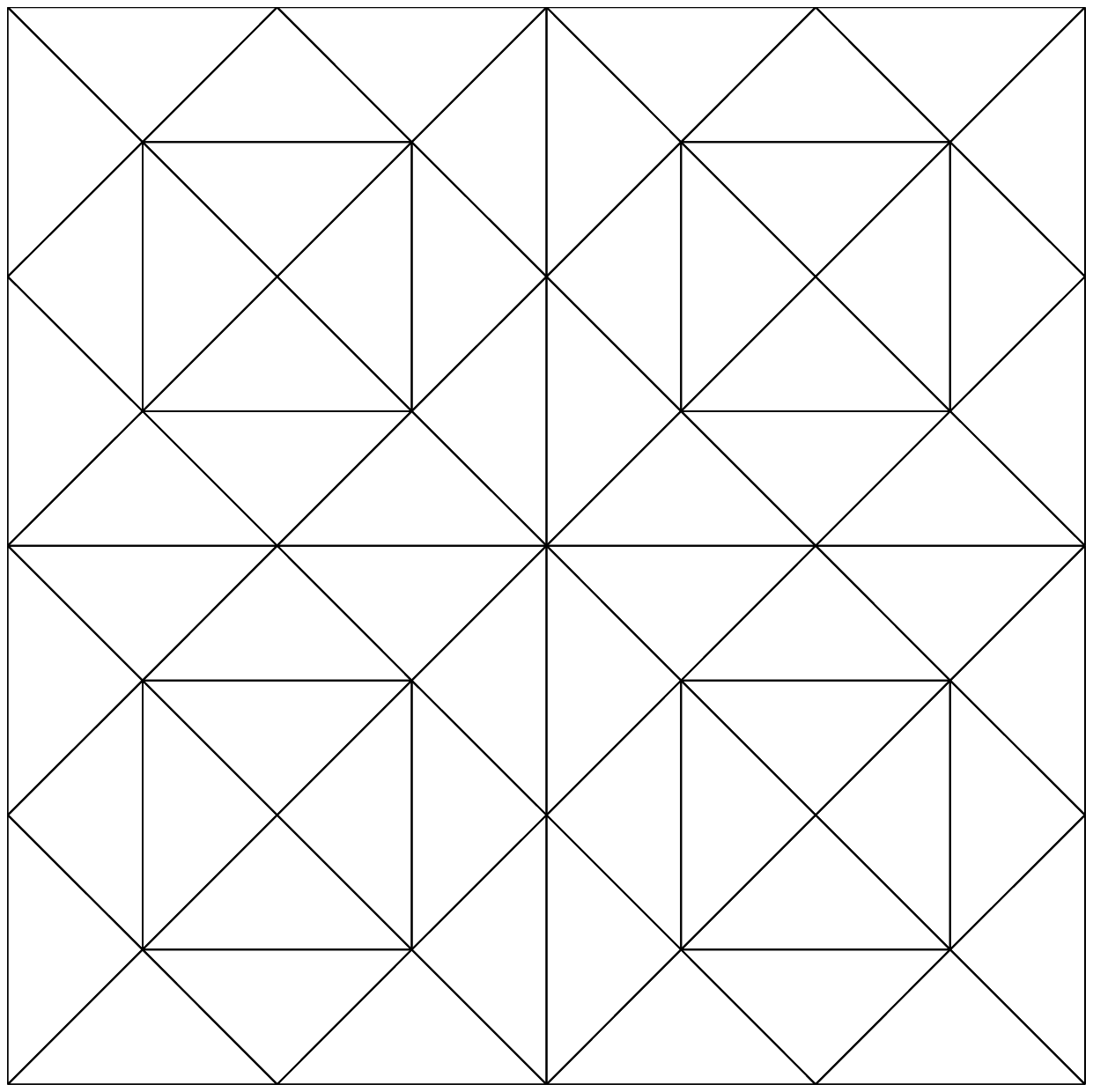}}
\scalebox{0.18}{\includegraphics{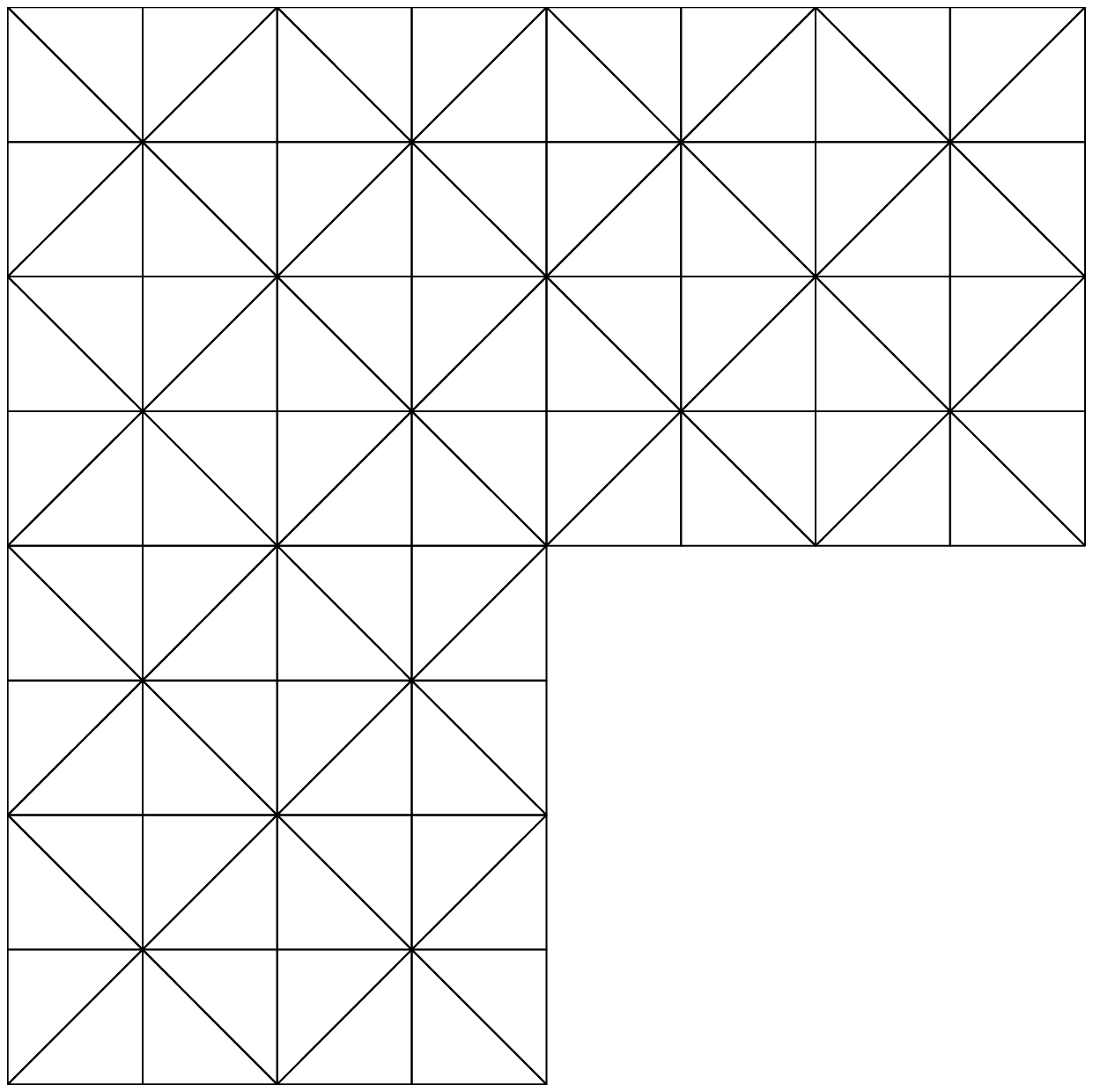}}
\scalebox{0.18}{\includegraphics{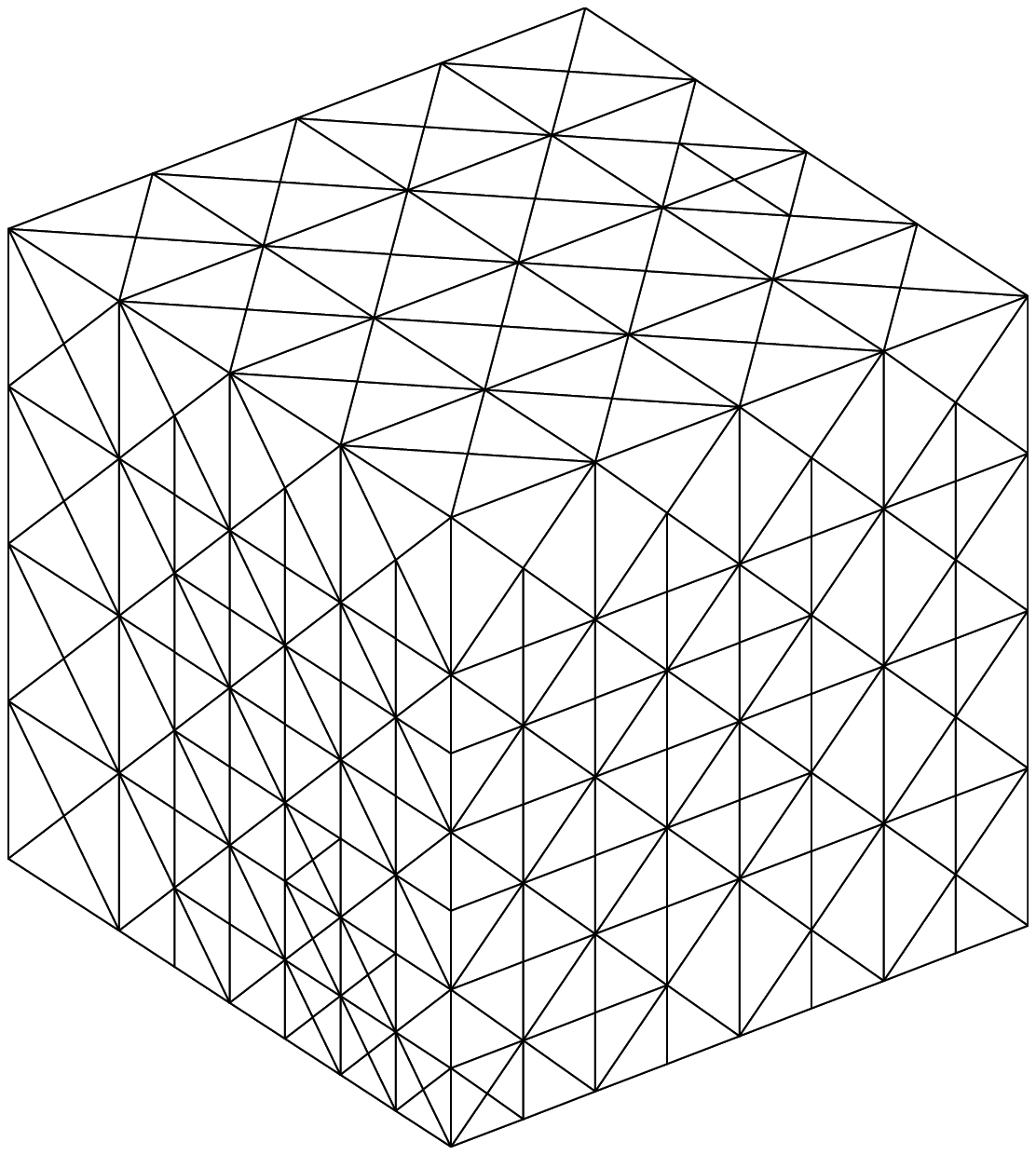}}
\end{center}
\caption{Initial meshes for Examples 1--3 (left), \RR{Example 4 (center) and Examples 5--7 (right)}.}
\label{Fig0}
\end{figure}
\subsection{Two dimensional examples}
\label{2d_examples}
\AAF{First, we consider a series of two dimensional examples for convex and nonconvex domains, with homogeneous and inhomogeneous Dirichlet boundary conditions and different numbers of source points. }
~\\~\\
\noindent \textbf{Example 1:} We consider a problem with homogeneous Dirichlet boundary conditions, letting $\Omega=(0,1)^{2}$, and setting $\asf = -0.5$, $\bsf =0.5$, and
\begin{gather*}
\fsf(x_{1},x_{2})=\sin(2\pi x_1)\cos(2\pi x_2)x_{1}^{3},\quad \calZ=\{(0.75,0.75),(0.25,0.25)\},  \\ 
 \ysf_{(0.75,0.75)}=1, ~ \ysf_{(0.25,0.25)}=-1.
\end{gather*}
\noindent \textbf{Example 2:} We let $\Omega=(0,1)^{2}$, and set the exact optimal adjoint state as in \eqref{exact_adjoint}, $\asf=-0.4$, $\bsf=-0.2$ and
\begin{eqnarray*}
\bar{\ysf}(x_{1},x_{2}) = 32x_1x_2(1-x_1)(1-x_2), ~ \calZ=\{(0.5,0.5)\}, ~ \ysf_{(0.5,0.5)}=1.
\end{eqnarray*}
\noindent \textbf{Example 3:} We let $\Omega=(0,1)^{2}$, and set the exact optimal adjoint state as in \eqref{exact_adjoint}, $\asf=-1.2$, $\bsf=-0.7$ and 
\begin{gather*}
\bar{\ysf}(x_{1},x_{2})=2.75-2x_1-2x_2+4x_1x_2, \\
\calZ=\{(0.75,0.75),(0.75,0.25),(0.25,0.75),(0.25,0.25)\}, \\
\ysf_{(0.75,0.75)}=1,\qquad \ysf_{(0.25,0.25)}=1,\qquad \ysf_{(0.75,0.25)}=0.5 \qquad \ysf_{(0.25,0.75)}=0.5.
\end{gather*}
\noindent \textbf{Example \AAF{4}:} We let $\Omega=(-1,1)^2\setminus [0,1)\times(-1,0]$ an L-shaped domain, and set the exact optimal adjoint state as in \eqref{exact_adjoint}, $\asf=-0.4$, $\bsf=-0.2$, and 
\begin{gather*}
\bar{\ysf}(x_{1},x_{2}):=r^{2/3}\sin(2\theta/3),\quad\textrm{with}~\theta\in[0,3\pi/2],\\
\calZ=\{(0.5,0.5)\},\quad \ysf_{(0.5,0.5)}=2^{1/3}\sin(\pi/6)-1.
\end{gather*}
\AAF{In Figure \ref{Fig:Ex1} we present convergence rates for the total error estimator and its individual contributions, with uniform and adaptive refinement, for Example 1. Figure \ref{Fig:Ex2}, presents convergence rates for the total error, error estimator and effectivity indices, considering $\alpha\in\{0.5,1,1.5\}$ and different marking strategies, for Example 2. We mention that we conducted such experiments for all the other problems, and came to the same conclusion. Figure \ref{Fig:Ex3-4} presents the finite element solutions for the optimal state $\bar{\ysf}_{\T}$, adjoint state $\bar{\psf}_{\T}$, and optimal control $\bar{\usf}_{\T}$, on adaptively refined meshes, and also convergence rates for the total error $\|(e_{\ysf},e_{\psf},e_{\usf})\|_{\Omega}$, error estimator $\E_{\mathsf{ocp}}$, their individual contributions and effectivity indices, for Examples 3 and 4.}
\psfrag{eta-Omega,for alpha=0.5}{\large $\E_{\mathsf{ocp}}$, for $\alpha=0.25$}
\psfrag{Ndofs}{\large Ndof}
\psfrag{adaptive-refinement}{$\E_{\mathsf{ocp}}$-Adaptive}
\psfrag{uniform-refinement}{$\E_{\mathsf{ocp}}$-Uniform}
\psfrag{ordenh1}{Ndof$^{-1/2}$}
\psfrag{ordenh2}{Ndof$^{-3/8}$}
\psfrag{ordenh-3/8}{Ndof$^{-3/8}$}
\psfrag{ordenh-1/2}{Ndof$^{-1/2}$}
\psfrag{ordenh-1}{Ndof$^{-1}$}
\psfrag{}{}
\psfrag{adaptive-refinement-log}{$\E_{\mathsf{ocp},\mathsf{log}}$}
\psfrag{adaptive-refinement-}{$\E_{\mathsf{ocp}}$}
\psfrag{yh}{\huge $\bar{\mathsf{y}}_{h}$}
\psfrag{estimador-refinamien-U}{$\E_{\mathsf{ocp}}$-Adaptive}
\psfrag{estimador-refinamien-A}{$\E_{\mathsf{ocp}}$-Uniform}
\begin{figure}[!h]
\psfrag{estimador-y-adapti}{$\E_{\mathsf{y}}$-Adaptive}
\psfrag{estimador-p-adapti}{$\E_{\mathsf{p}}$-Adaptive}
\psfrag{estimador-u-adapti}{$\E_{\mathsf{u}}$-Adaptive}
\psfrag{estimador-y-unifor}{$\E_{\mathsf{y}}$-Uniform}
\psfrag{estimador-p-unifor}{$\E_{\mathsf{p}}$-Uniform}
\psfrag{estimador-u-unifor}{$\E_{\mathsf{u}}$-Uniform}
\psfrag{FIGURAAAAAAAAAAAAAAAA4AAAAAAAAAA}{$\E_{\mathsf{ocp}}$ Adaptive vs Uniform for $\alpha=1.5$}
\includegraphics[width=4cm,height=4cm,scale=0.55]{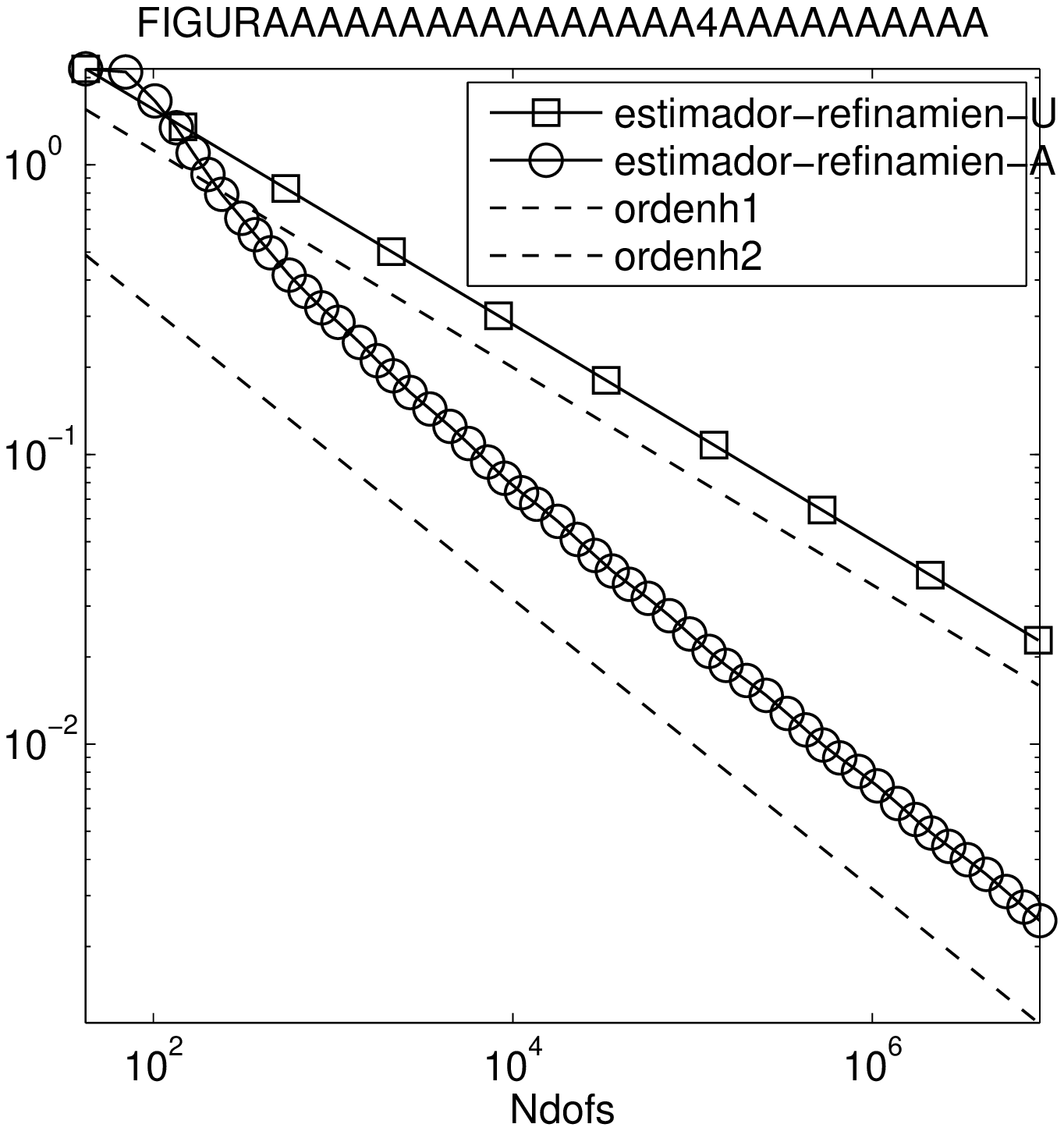}
\psfrag{FIGURAAAAAAAAAAAAAAAA4}{Estimator contributions for $\alpha=1.5$}
\includegraphics[width=4cm,height=4cm,scale=0.55]{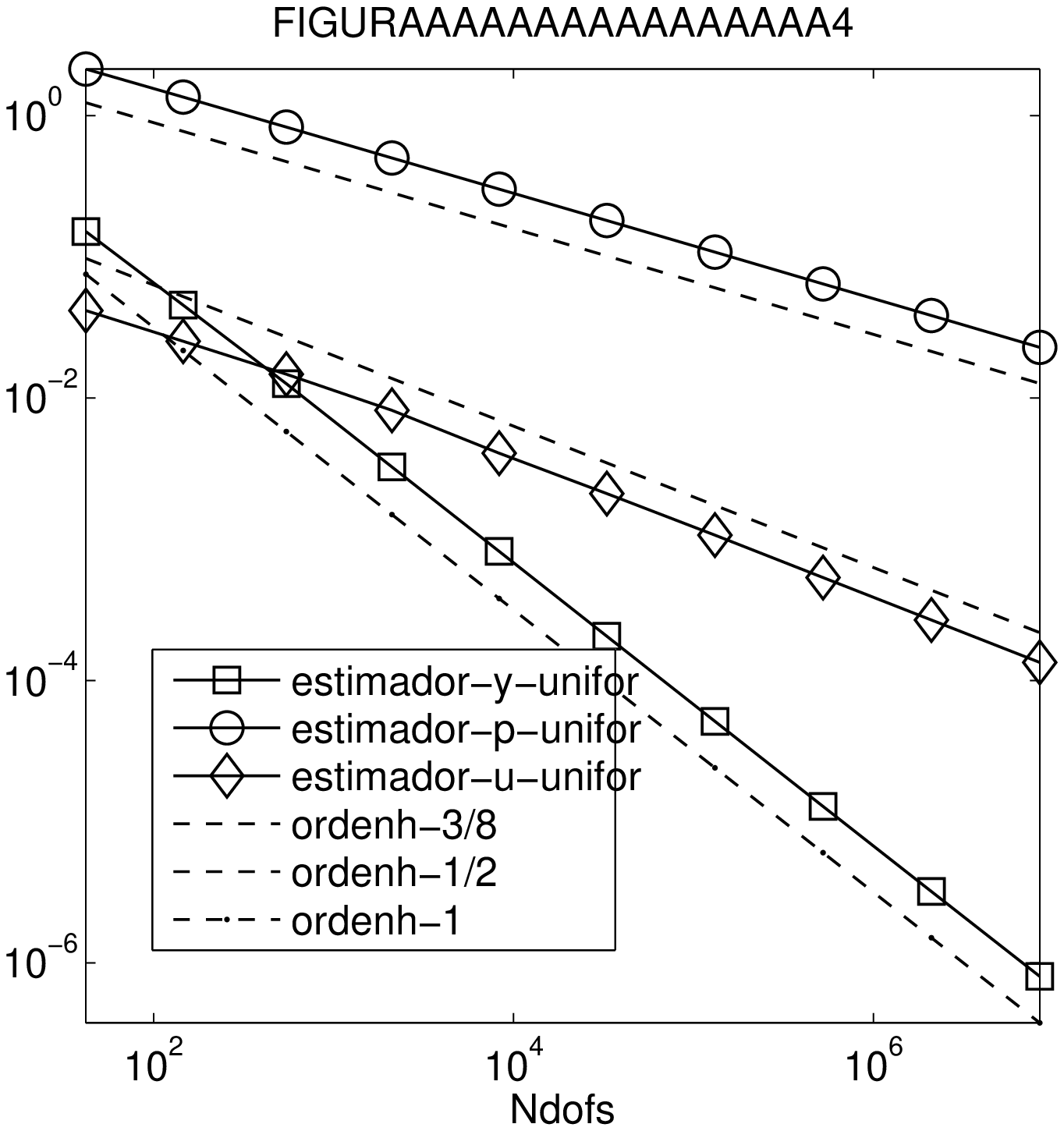}
\psfrag{FIGURAAAAAAAAAAAAAAAA4}{Estimator contributions for $\alpha=1.5$}
\includegraphics[width=4cm,height=4cm,scale=0.55]{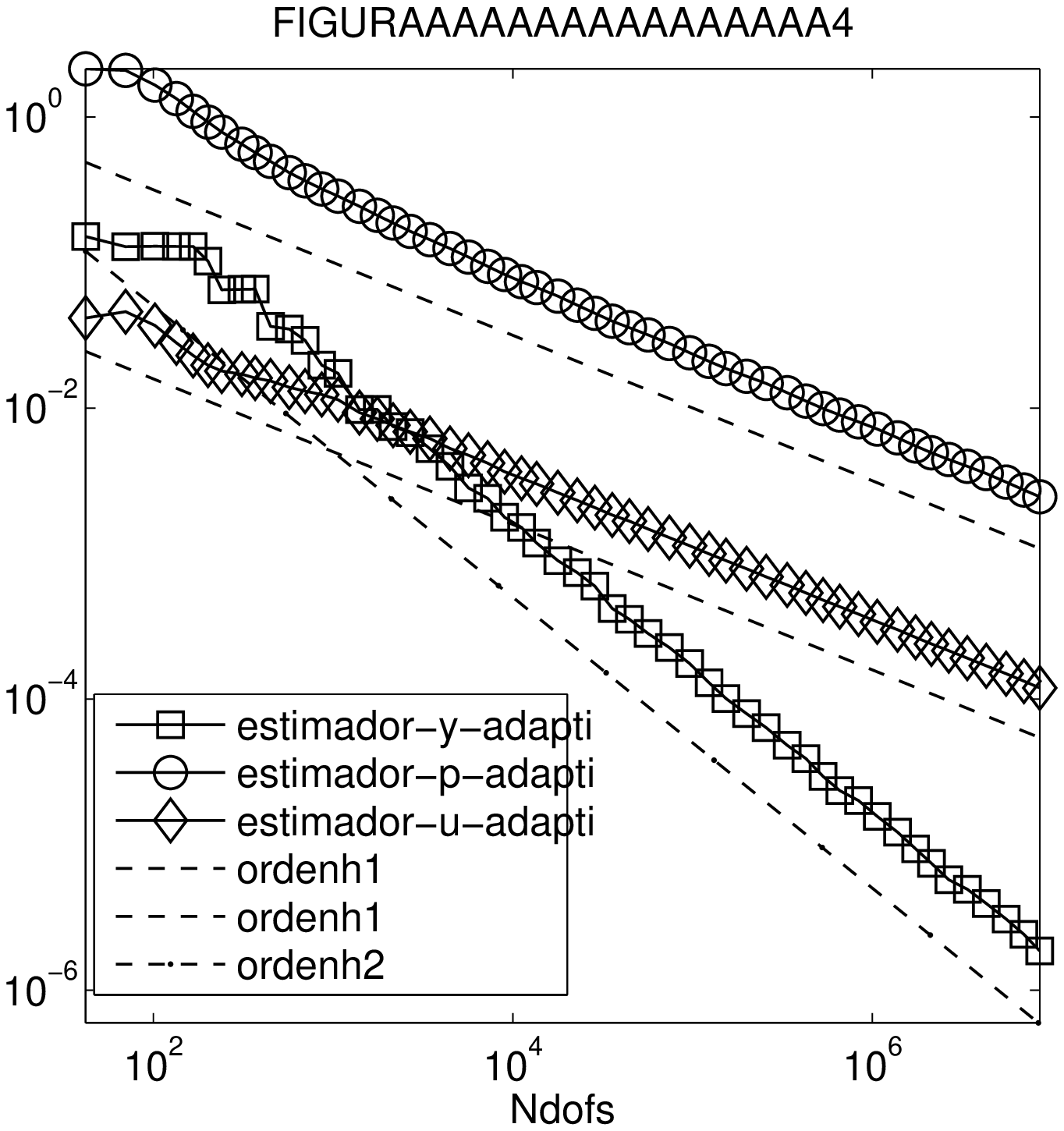}
\caption{\AAF{Example 1: For $\alpha = 1.5$ and based on a maximum refinement strategy we show convergence rates for the total error estimator with adaptive and uniform refinement (left), all the different contributions with uniform refinement (center), and adaptive refinement (right).}}
\label{Fig:Ex1}
\end{figure}
\psfrag{error para alpha=0.5----------------}{\Huge $\|(e_{\ysf},e_{\psf},e_{\usf})\|_{\Omega}$ for $\alpha=0.5$}
\psfrag{error para alpha=1----------------}{\Huge $\|(e_{\ysf},e_{\psf},e_{\usf})\|_{\Omega}$ for $\alpha=1$}
\psfrag{error para alpha=1.5----------------}{\Huge $\|(e_{\ysf},e_{\psf},e_{\usf})\|_{\Omega}$ for $\alpha=1.5$}

\psfrag{estimador para alpha=0.5}{\Huge $\E_{\mathsf{ocp}}$ for $\alpha=0.5$}
\psfrag{estimador para alpha=1}{\Huge $\E_{\mathsf{ocp}}$ for $\alpha=1$}
\psfrag{estimador para alpha=1.5}{\Huge $\E_{\mathsf{ocp}}$ for $\alpha=1.5$}

\psfrag{eficiencia para alpha=0.5---------------}{\Huge $\E_{\mathsf{ocp}}/\|(e_{\ysf},e_{\psf},e_{\usf})\|_{\Omega}$ for $\alpha=0.5$}
\psfrag{eficiencia para alpha=1---------------}{\Huge $\E_{\mathsf{ocp}}/\|(e_{\ysf},e_{\psf},e_{\usf})\|_{\Omega}$ for $\alpha=1$}
\psfrag{error para alpha=1.5}{\Huge $\|(e_{\ysf},e_{\psf},e_{\usf})\|_{\Omega}$ for $\alpha=1.5$}

\psfrag{eficiencia para alpha=1.5---------------}{\Huge $\E_{\mathsf{ocp}}/\|(e_{\ysf},e_{\psf},e_{\usf})\|_{\Omega}$ for $\alpha=1.5$}
\psfrag{Maximum----}{\huge Maximum}
\psfrag{Bulk----}{\huge Bulk}
\psfrag{Average----}{\huge Average}
\psfrag{ordenh1}{\Large $\textrm{Ndof}^{-1/2}$}
\psfrag{Ndofs}{\huge Ndof}
\begin{figure}[!h]
\begin{center}
\includegraphics[width=4cm,height=3.8cm,scale=1]{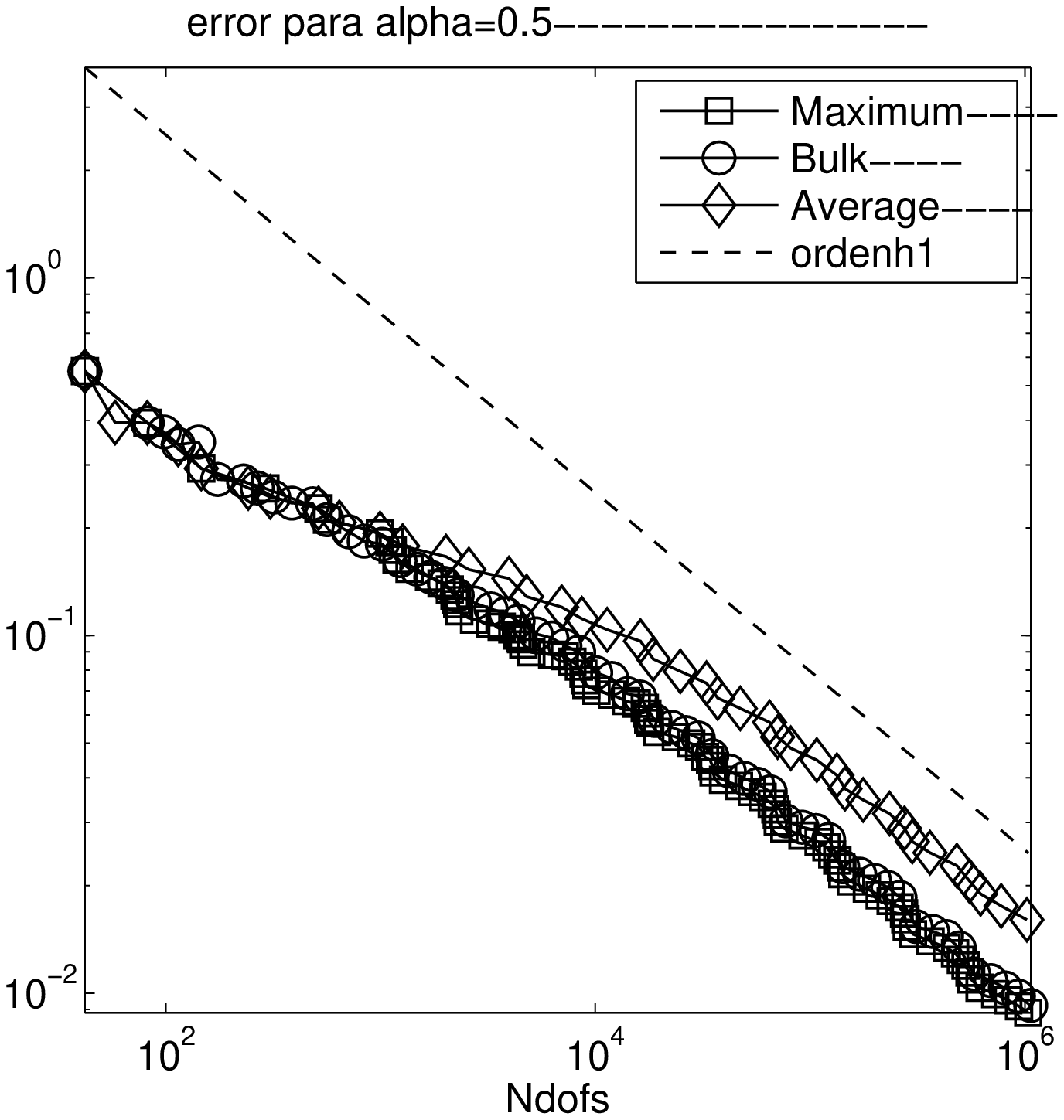}
\includegraphics[width=4cm,height=3.8cm,scale=1]{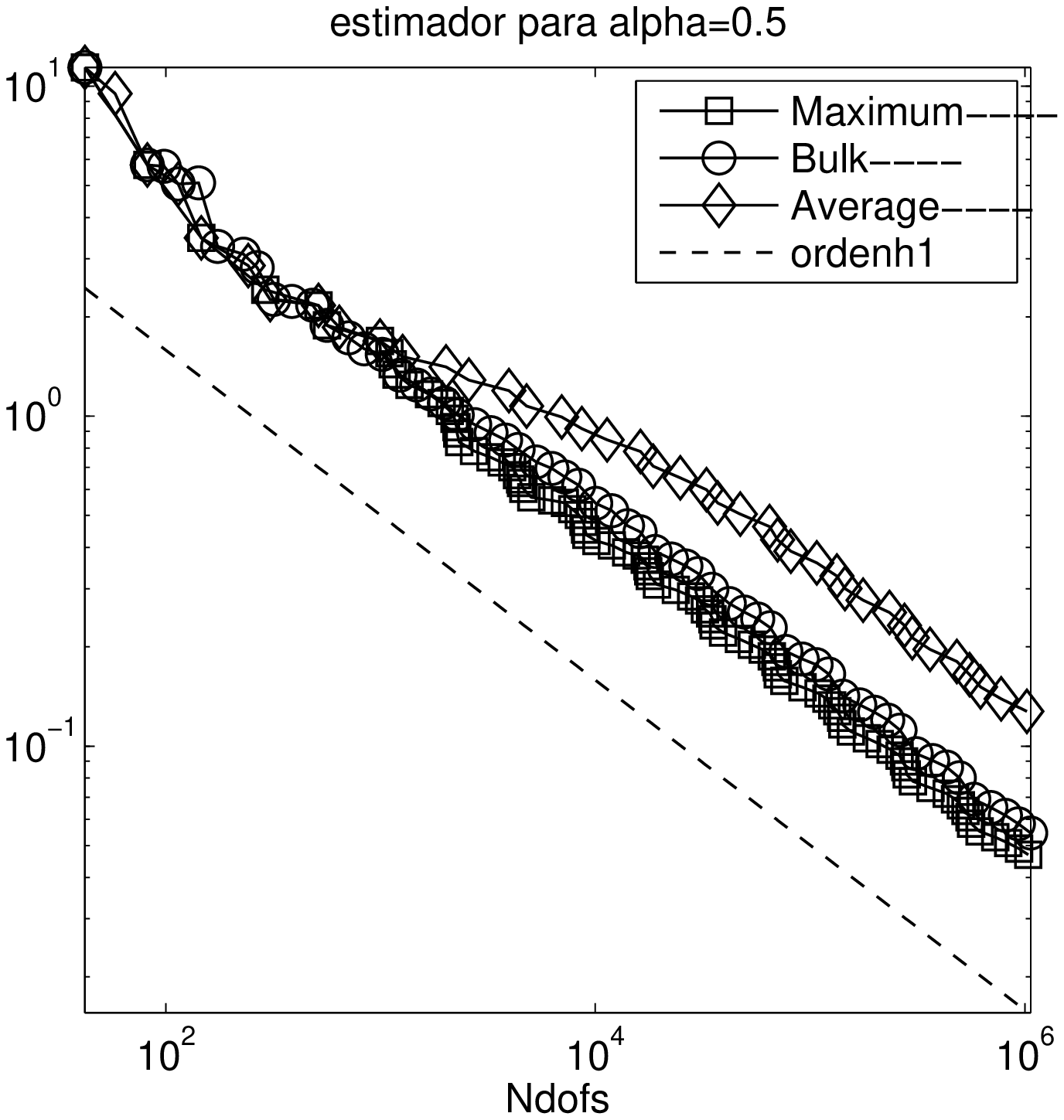}
\includegraphics[width=4cm,height=3.8cm,scale=1]{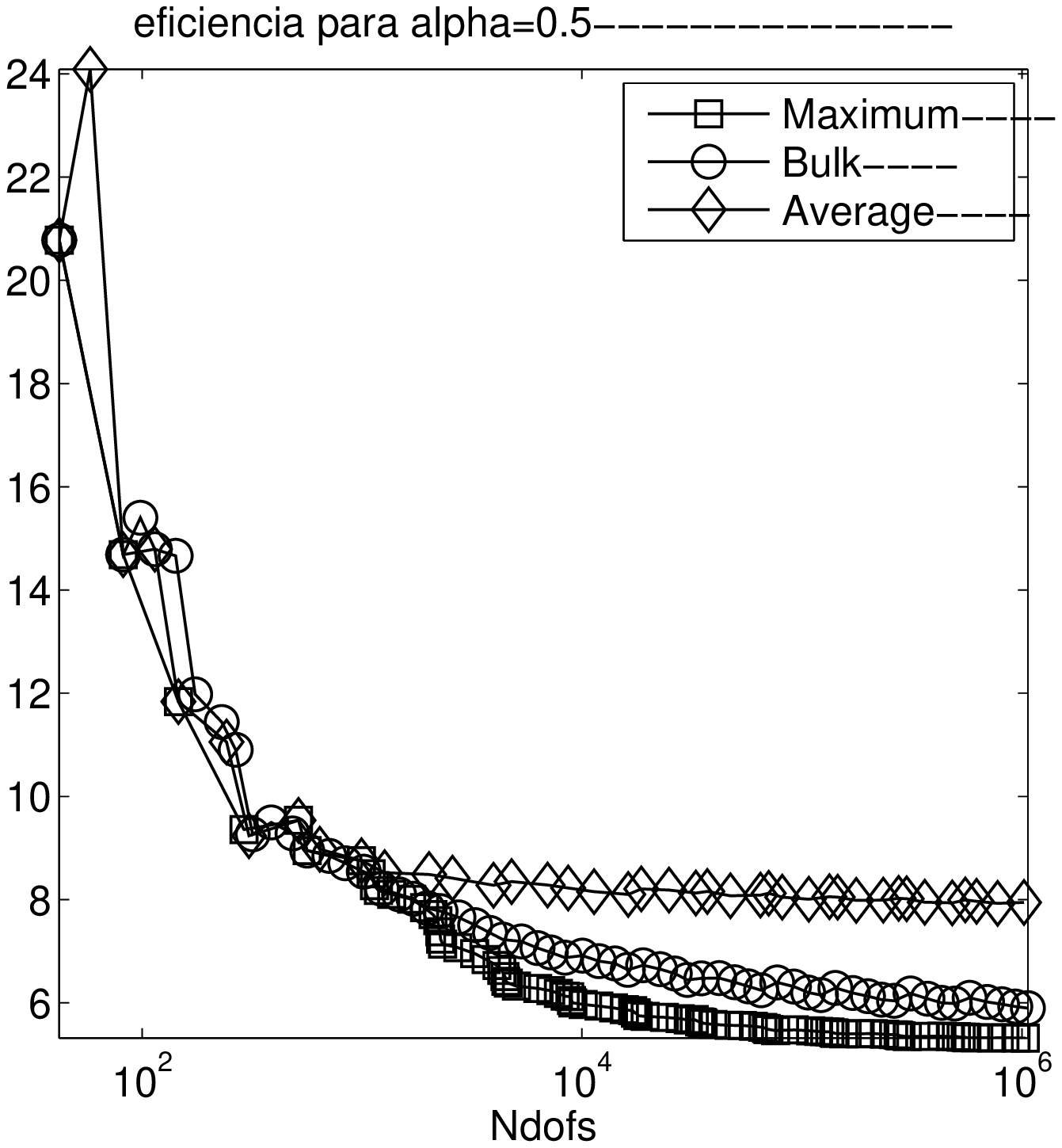}\\
\includegraphics[width=4cm,height=3.8cm,scale=1]{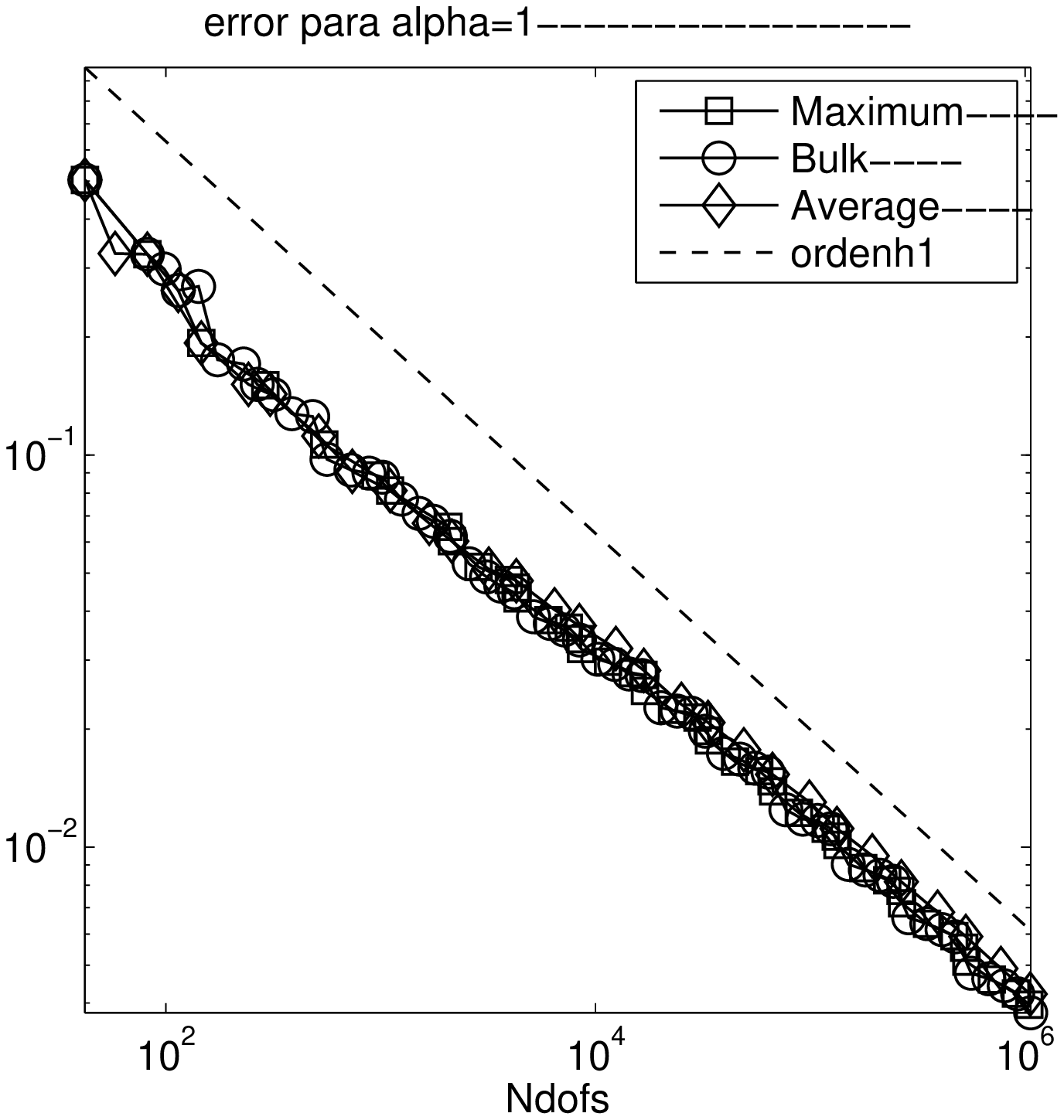}
\includegraphics[width=4cm,height=3.8cm,scale=1]{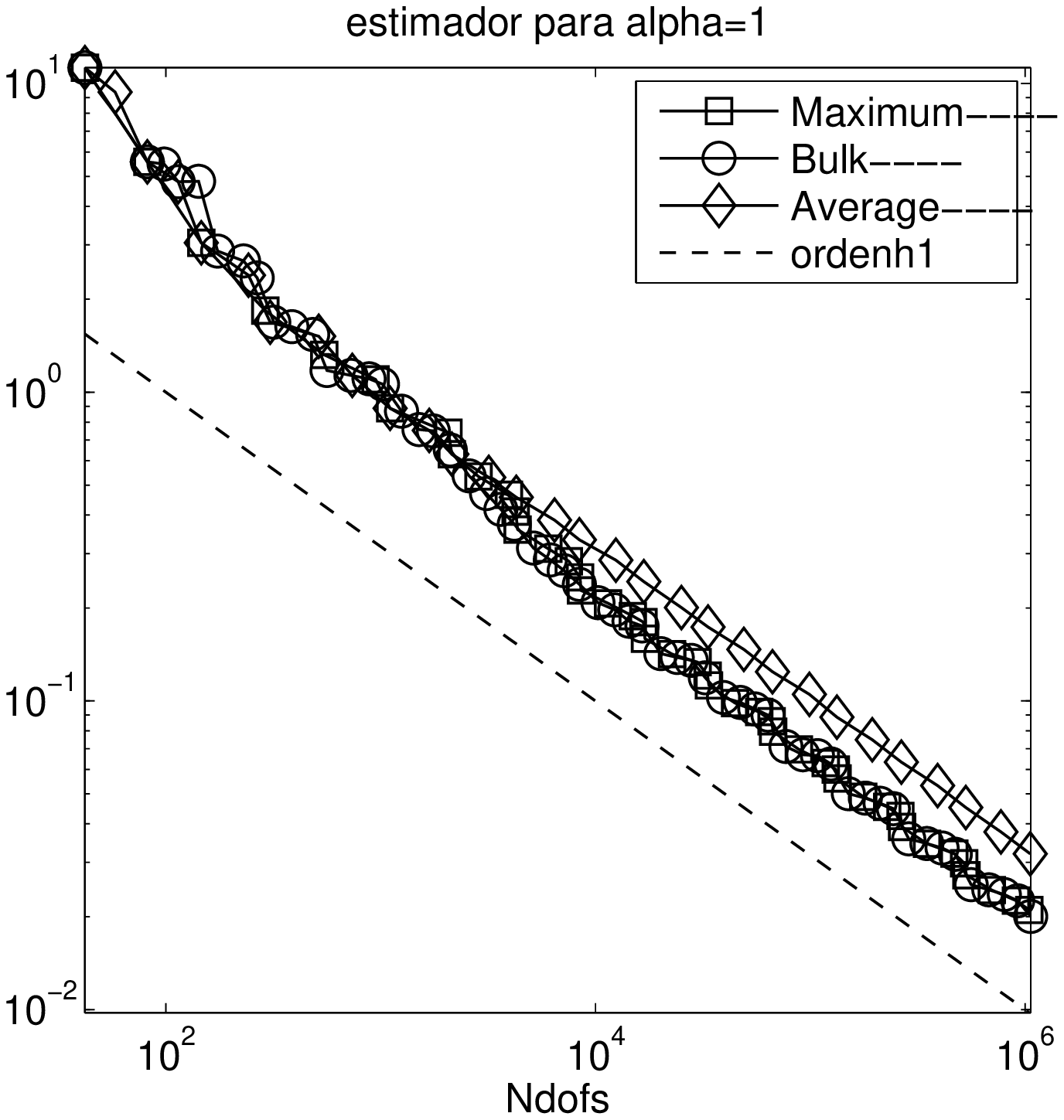}
\includegraphics[width=4cm,height=3.8cm,scale=1]{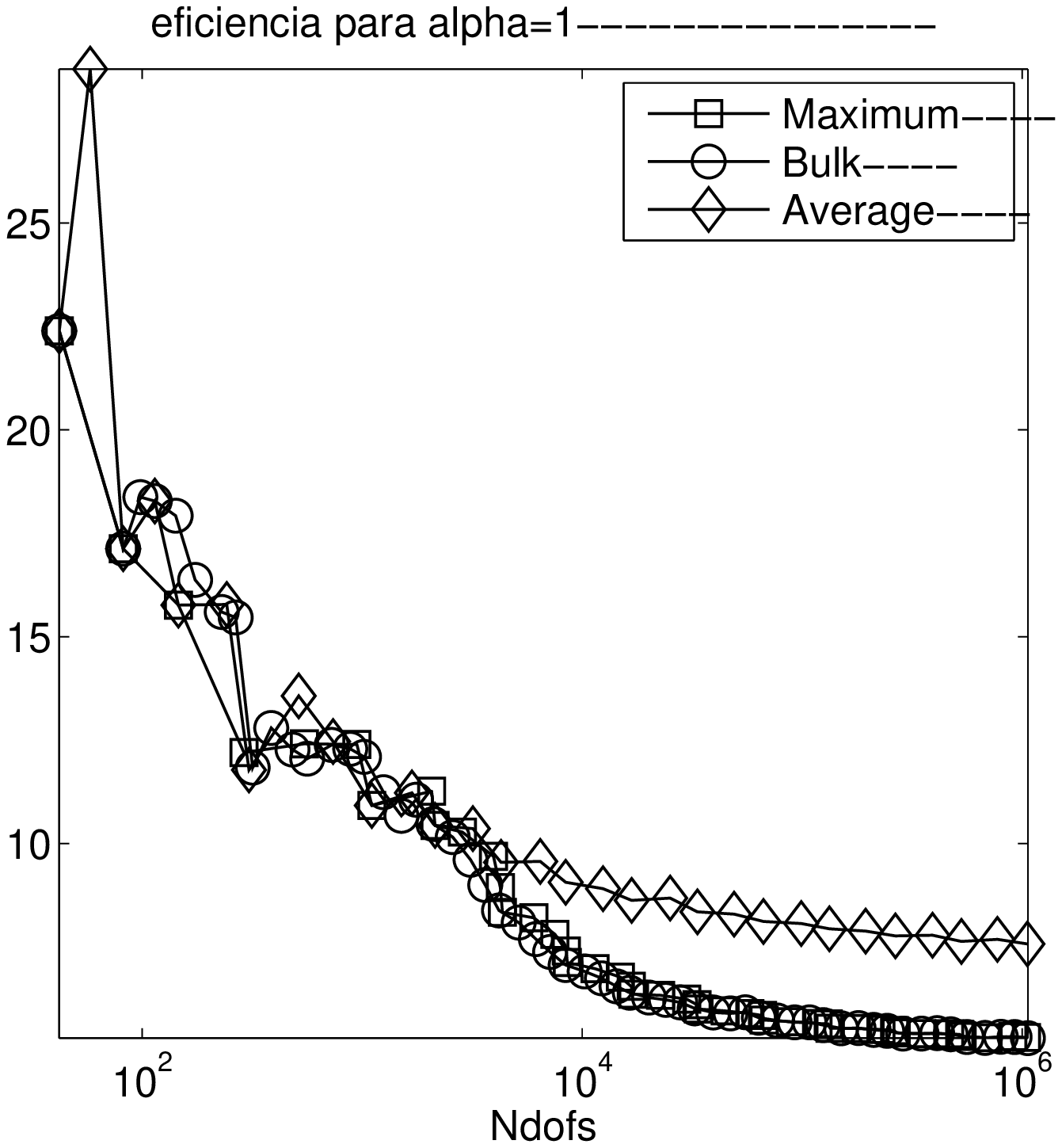}\\
\includegraphics[width=4cm,height=3.8cm,scale=1]{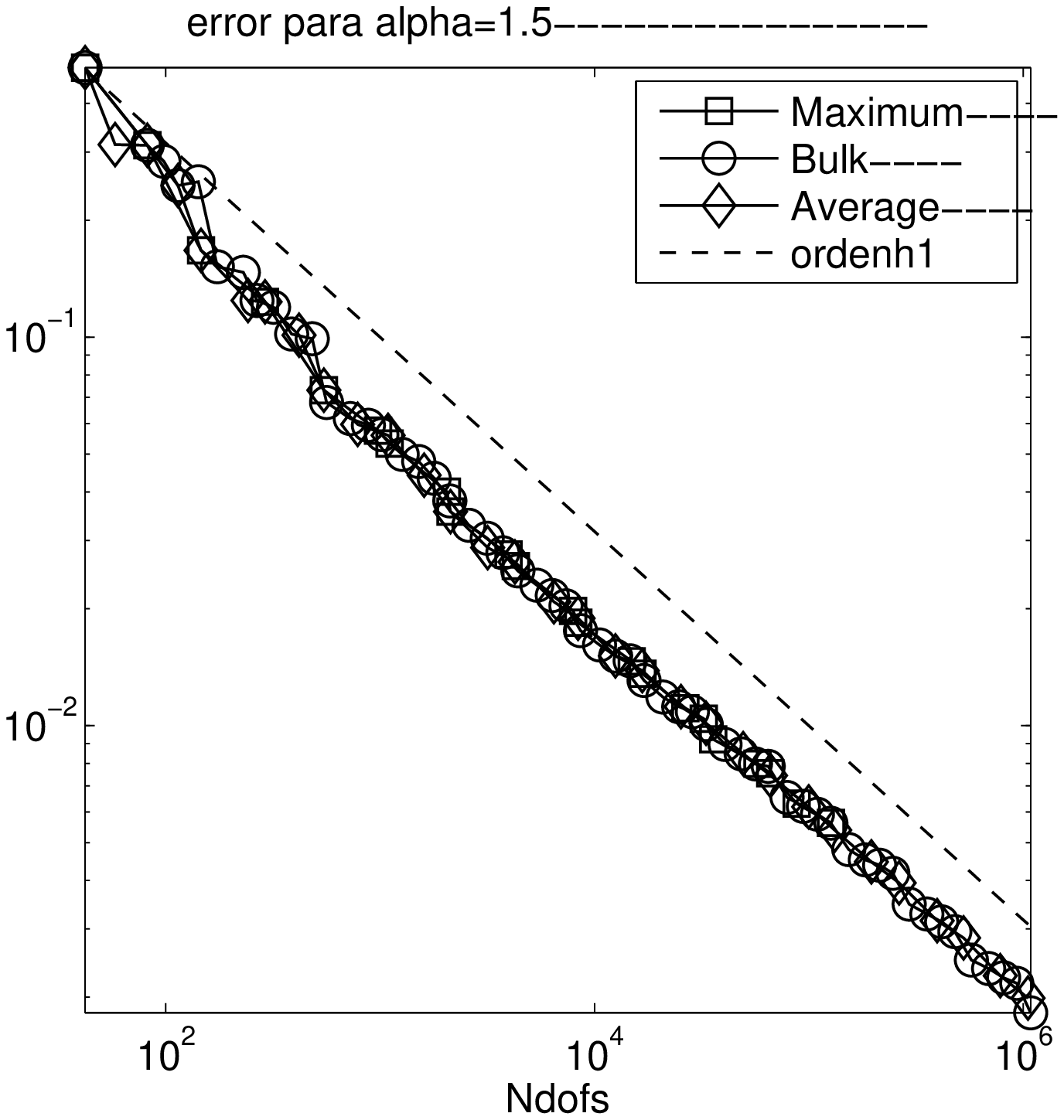}
\includegraphics[width=4cm,height=3.8cm,scale=1]{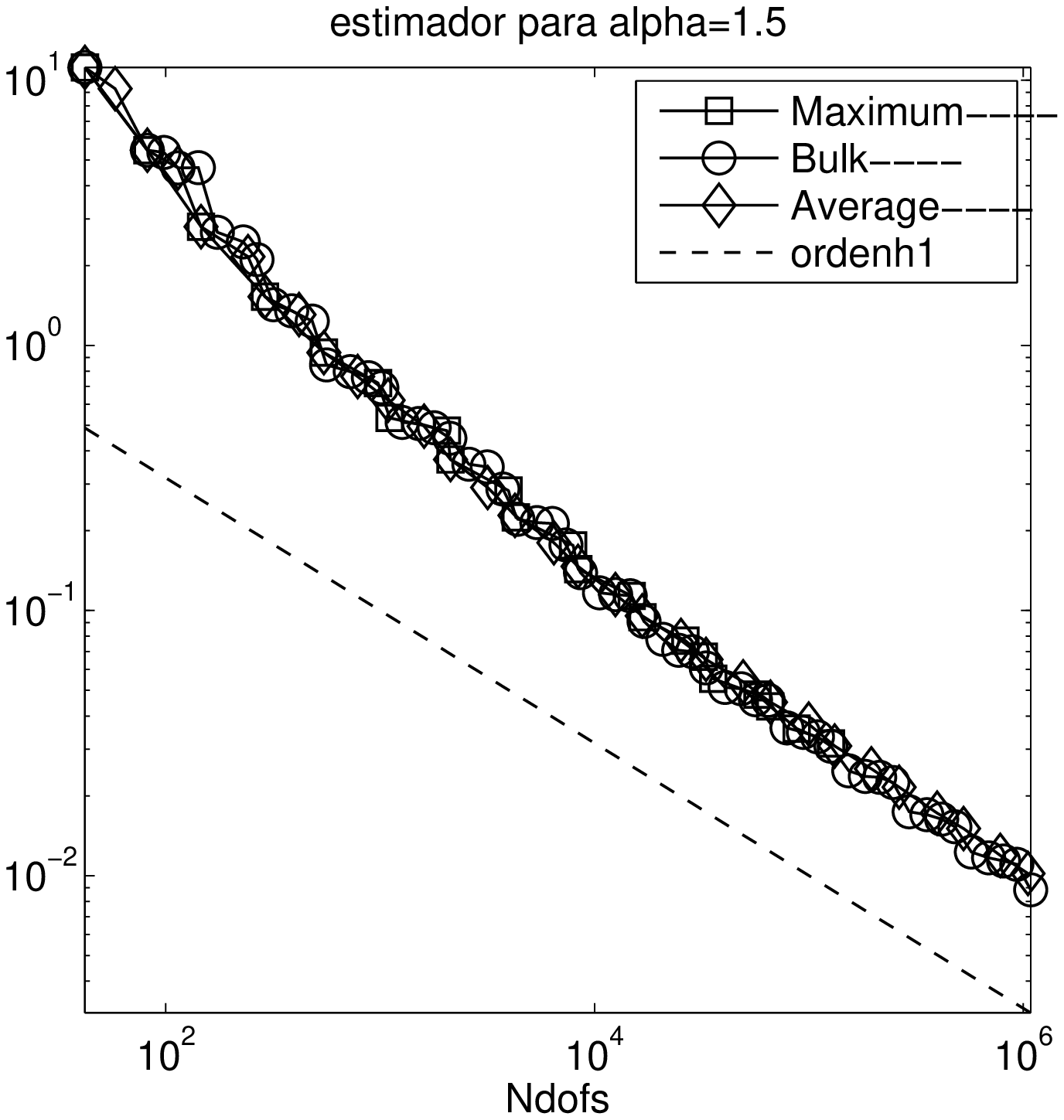}
\includegraphics[width=4cm,height=3.8cm,scale=1]{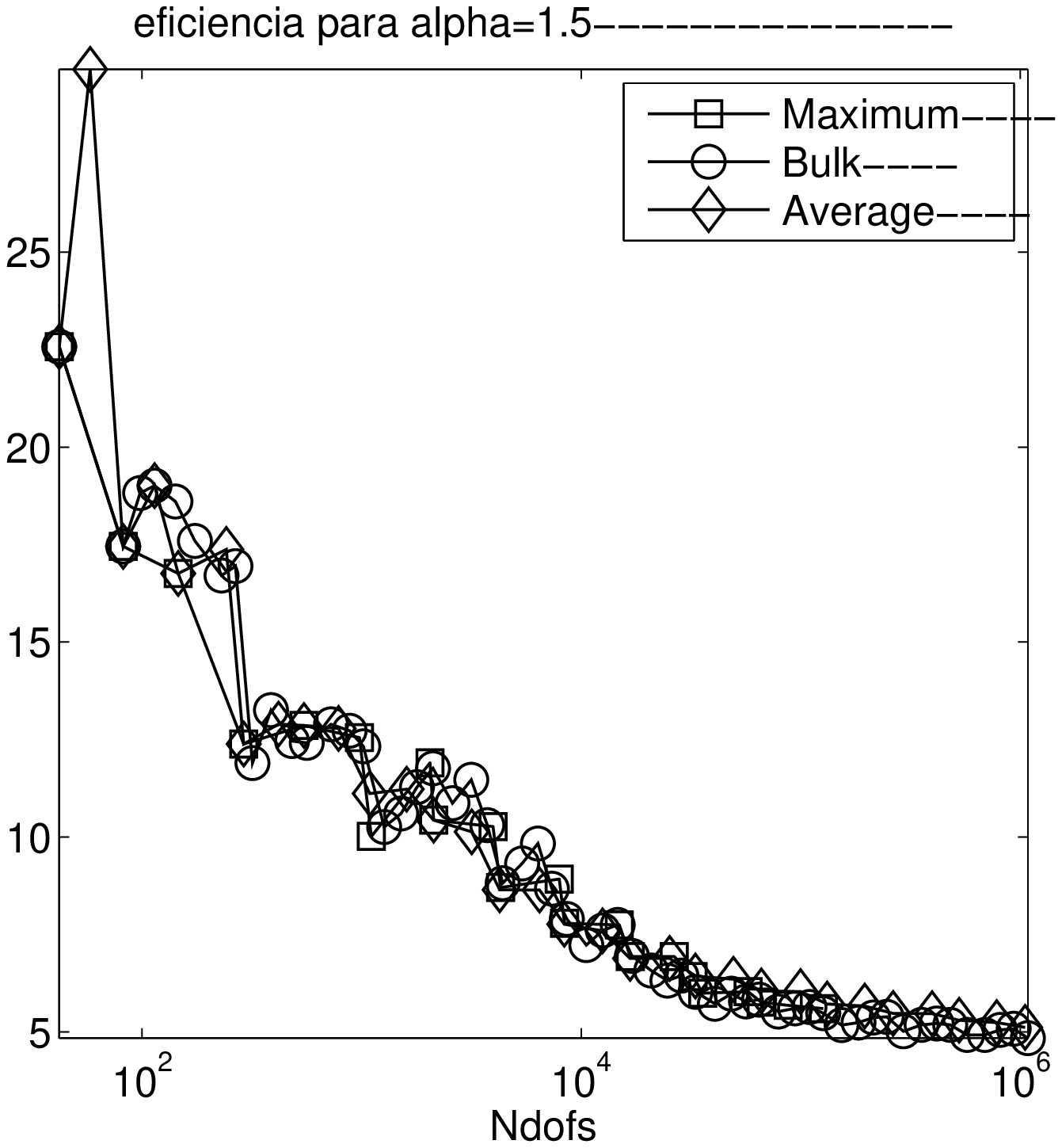}
\end{center}
\caption{\AAF{Example 2: Convergence rates for the total error (left), error estimator (center), and effectivity indices (right), with $\alpha\in\{0.5,1,1.5\}$ and different marking strategies.}}
\label{Fig:Ex2}
\end{figure}
\begin{figure}[!h]
\psfrag{FIGURAAAAAAAAAAAAAAAA5}{\huge $\E_{\ysf}$, $\E_{\psf}$ and $\E_{\usf}$ for $\alpha=1.5$}
\psfrag{error para y}{\huge $e_{\ysf}$}
\psfrag{error para p}{\huge $e_{\psf}$}
\psfrag{error para u}{\huge $e_{\usf}$}
\psfrag{estimador para alpha=1.5}{\huge $\E_{\mathsf{ocp}}$}
\psfrag{estimador para y}{\huge $\E_{\ysf}$}
\psfrag{estimador para p}{\huge $\E_{\psf}$}
\psfrag{estimador para u}{\huge $\E_{\usf}$}
\psfrag{maximum---}{\huge Maximum}
\psfrag{bulk---}{\huge Bulk}
\psfrag{average---}{\huge Average}
\psfrag{ordenh1}{\Large $\textrm{Ndof}^{-1/2}$}
\psfrag{Ndofs}{\huge Ndof}
\begin{subfigure}[b]{0.24\textwidth}
\centering $\bar{\mathsf{y}}_{h}$\\
\includegraphics[width=2.8cm,height=2.8cm,scale=1.5]{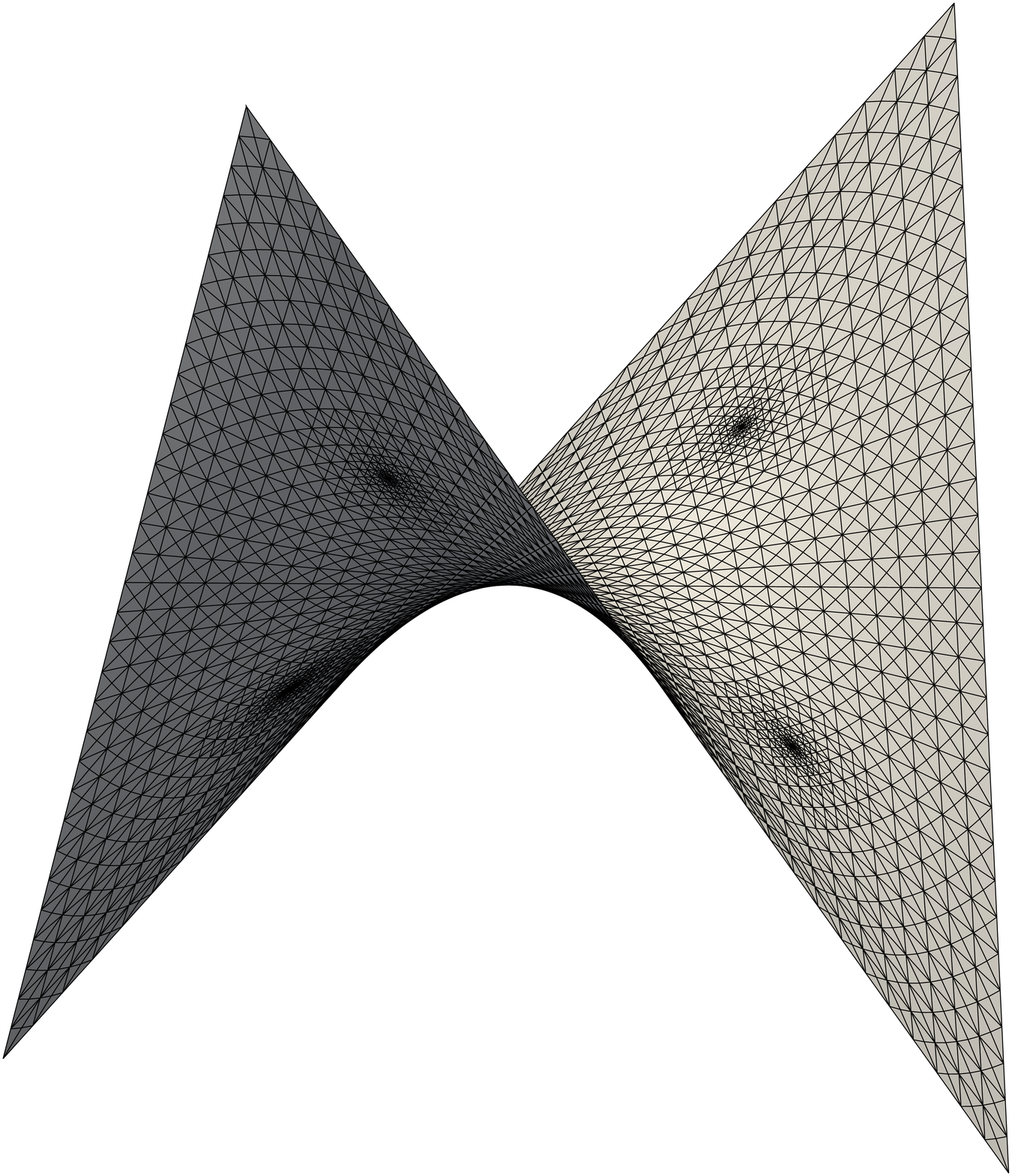}\\
\end{subfigure}
\begin{subfigure}[b]{0.26\textwidth}
\centering $\bar{\mathsf{p}}_{h}$\\
\includegraphics[width=2.8cm,height=2.8cm,scale=0.5]{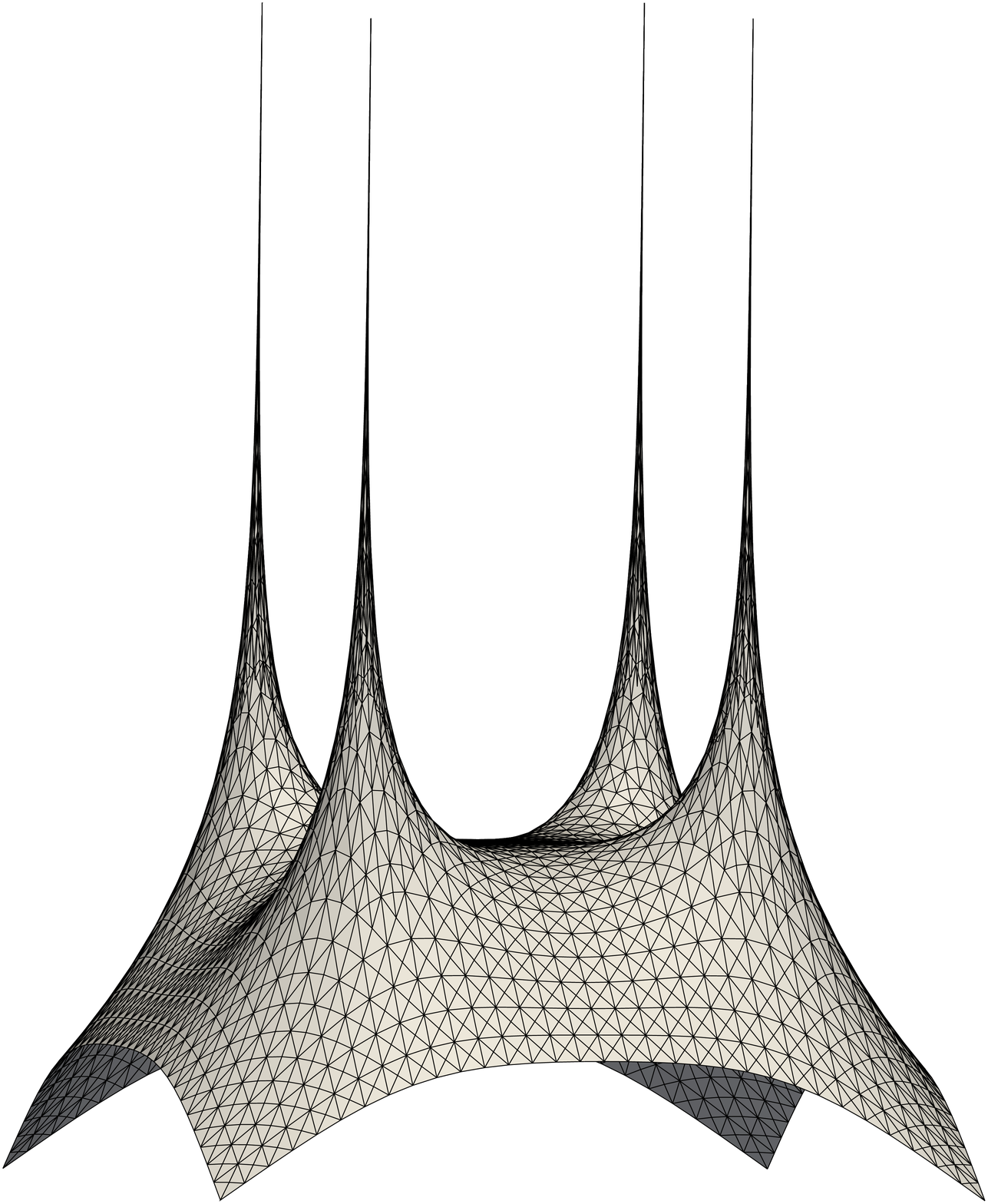}
\end{subfigure}
\begin{subfigure}[b]{0.26\textwidth}
\centering $\bar{\mathsf{u}}_{h}$\\
\includegraphics[width=2.6cm,height=2.8cm,scale=0.5]{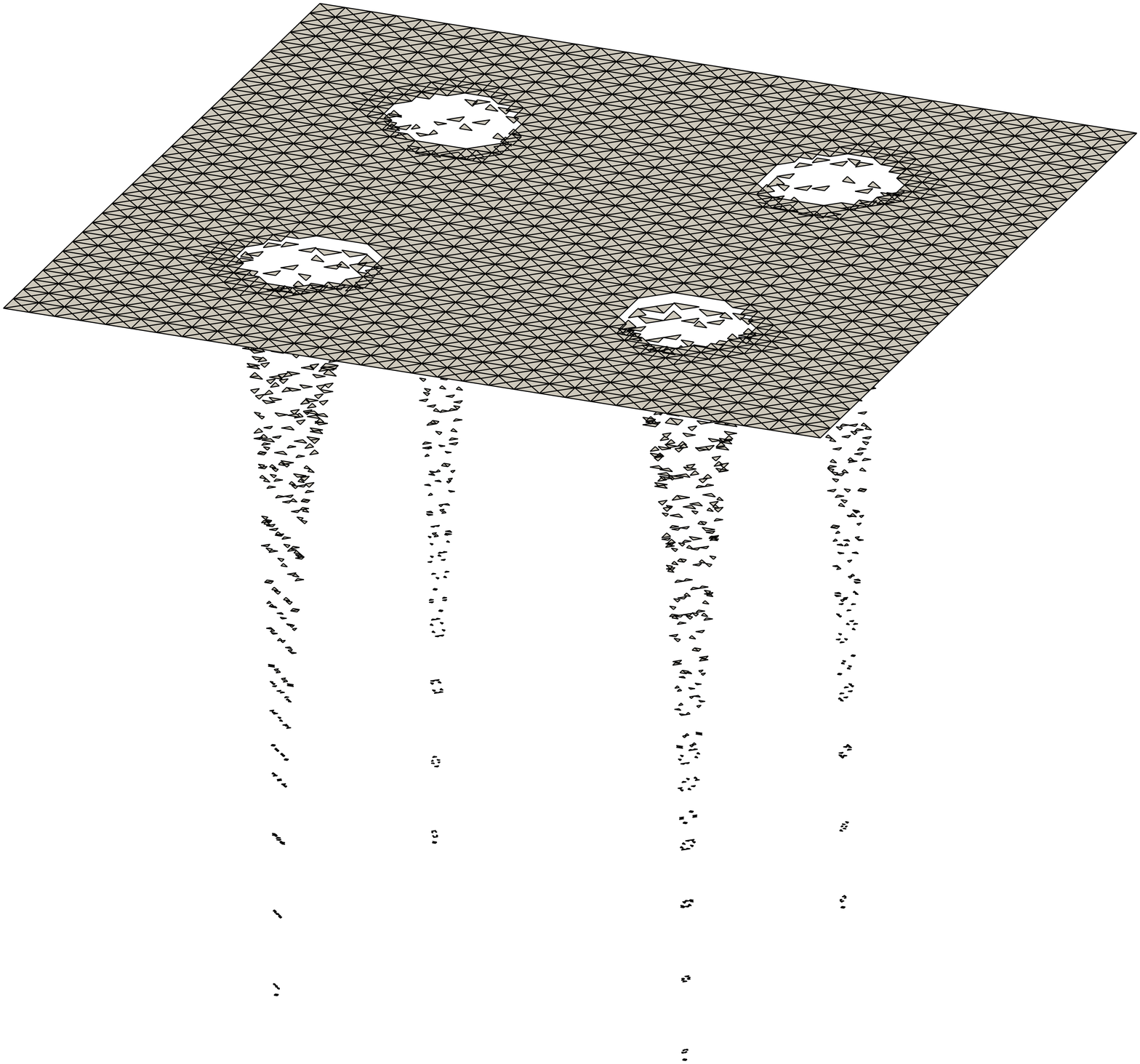}
\end{subfigure}
\begin{subfigure}[b]{0.2\textwidth}
\includegraphics[width=2.5cm,height=2.8cm,scale=0.5]{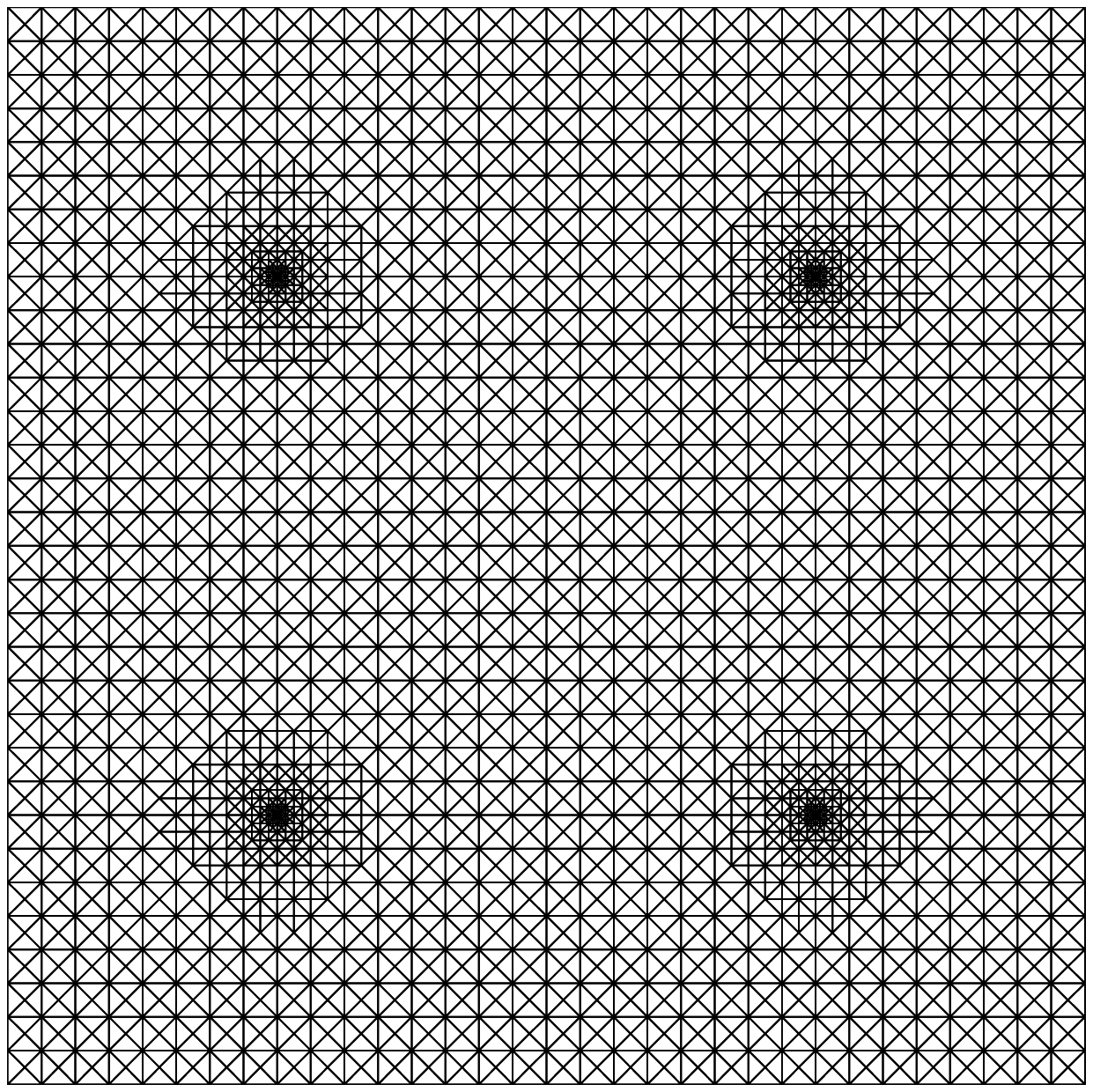}
\end{subfigure}
\\
\psfrag{FIGURAAAAAAAAAAAAAAAAAAAAAA4}{\huge $\|(e_{\ysf},e_{\psf},e_{\usf})\|_{\Omega}$ and $\E_{\mathsf{ocp}}$ for $\alpha=1.5$}
\psfrag{error para alpha=1.5}{\huge $\|(e_{\ysf},e_{\psf},e_{\usf})\|_{\Omega}$}
\psfrag{FIGURAAAAAAAAAAAAAAAAAAAAAA4-S}{\huge Error contributions for $\alpha=1.5$}
\psfrag{FIGURAAAAAAAAAAAAAAAAAAAAAA4-SS}{\huge Estimator contributions for $\alpha=1.5$}
\psfrag{eficiencias para alpha=1.5-----}{\huge Effectivity index for $\alpha=1.5$}
\psfrag{error y--------}{\huge $\|e_\ysf\|_{L^{\infty}(\Omega)}$}
\psfrag{error p--------}{\huge $\|\nabla e_\psf\|_{L^{2}(\rho,\Omega)}$}
\psfrag{error u--------}{\huge $\|e_\usf\|_{L^{2}(\Omega)}$}
\psfrag{estimador y}{\huge $\E_{\mathsf{y}}$}
\psfrag{estimador p}{\huge $\E_{\mathsf{p}}$}
\psfrag{estimador u}{\huge $\E_{\mathsf{u}}$}
\psfrag{estimador / error -----------}{\huge $\E_{\mathsf{ocp}}/\|(e_{\ysf},e_{\psf},e_{\usf})\|_{\Omega}$}
\psfrag{eficiencias y}{\huge $\E_{\mathsf{y}}/e_{\ysf}$}
\psfrag{eficiencias p}{\huge $\E_{\mathsf{p}}/e_{\psf}$}
\psfrag{eficiencias u}{\huge $\E_{\mathsf{u}}/e_{\usf}$}
\begin{subfigure}[b]{1\textwidth}
\includegraphics[width=3.1cm,height=3.4cm,scale=1]{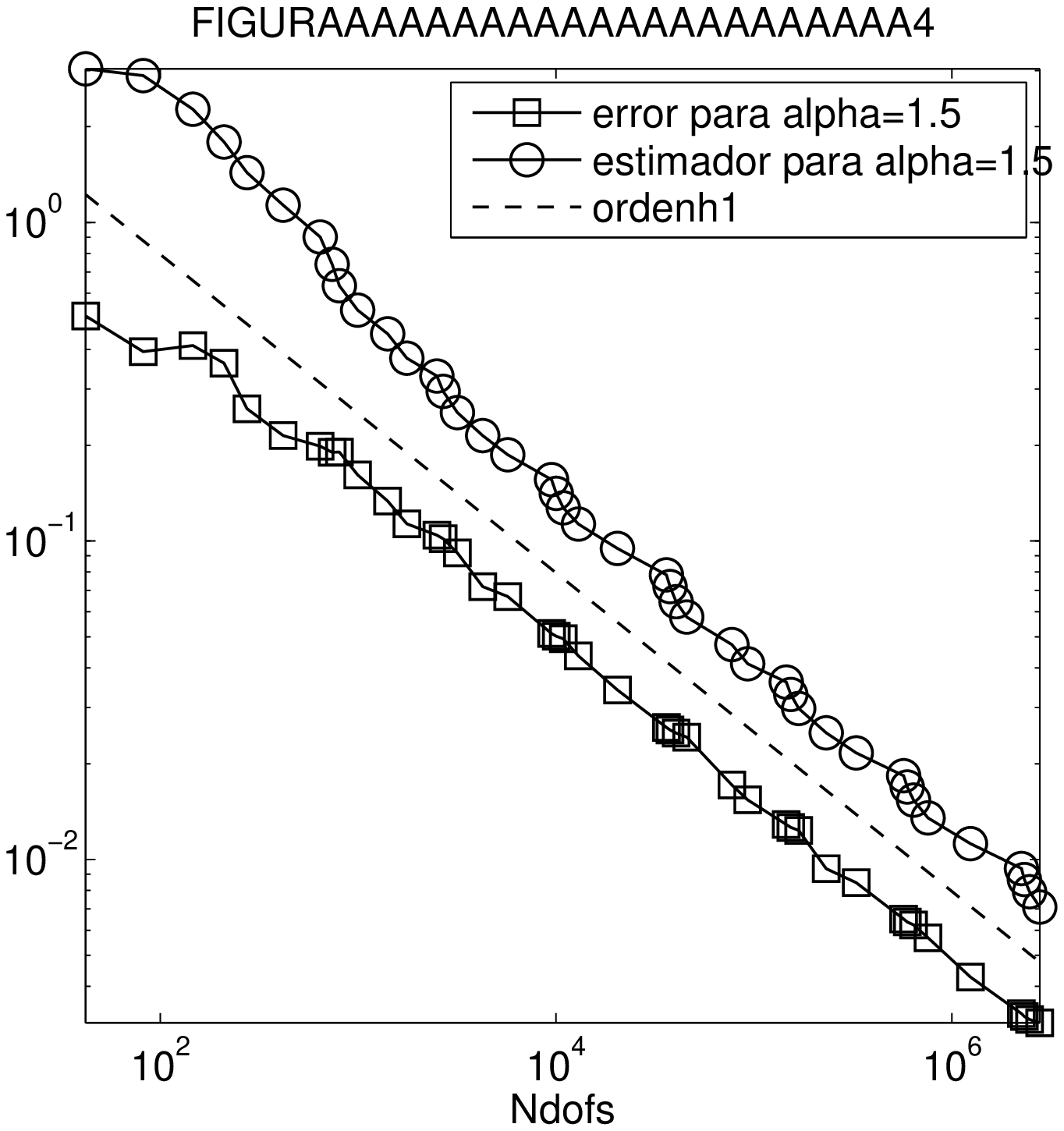}
\includegraphics[width=3.1cm,height=3.4cm,scale=1]{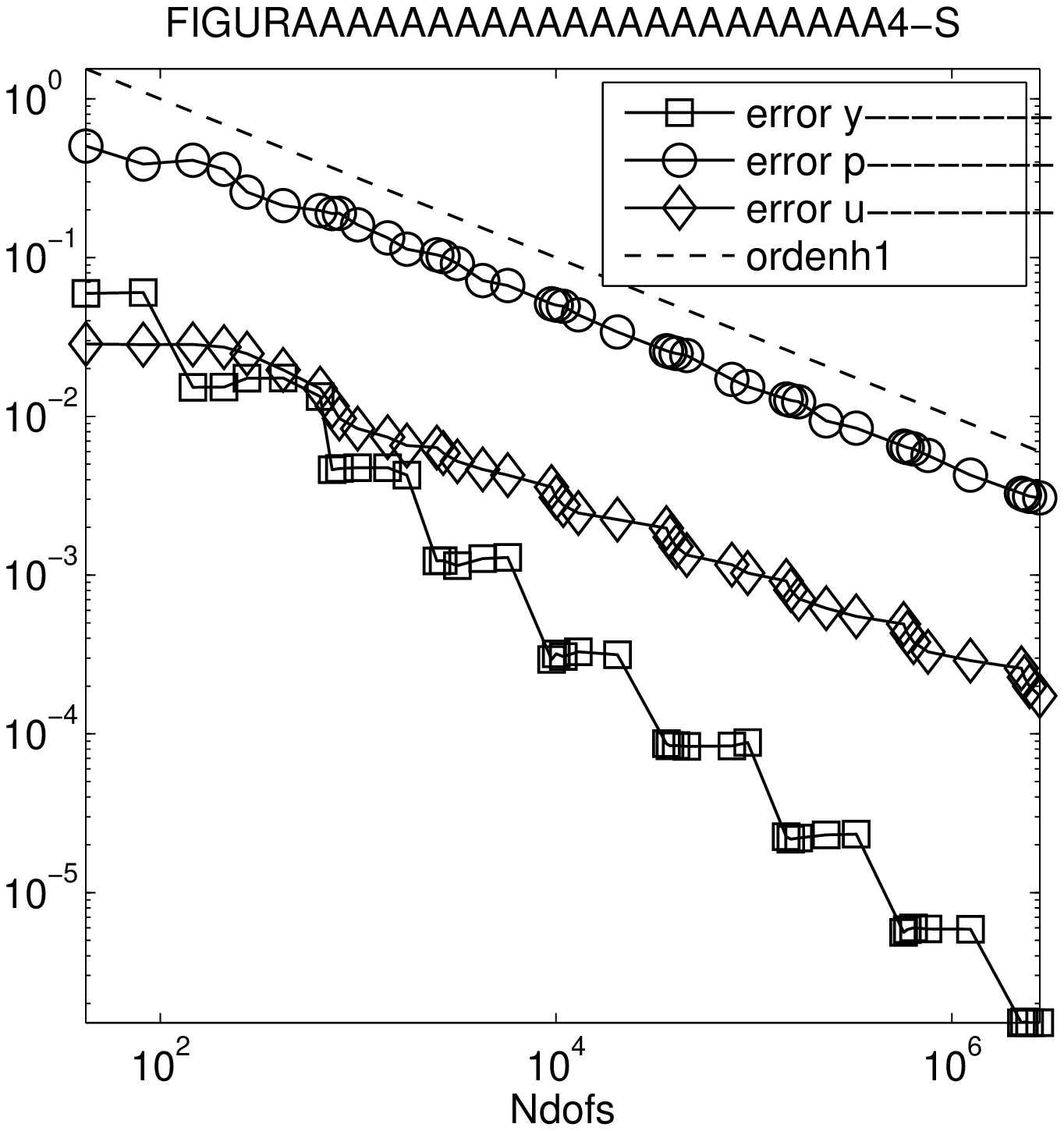}
\includegraphics[width=3.1cm,height=3.4cm,scale=1]{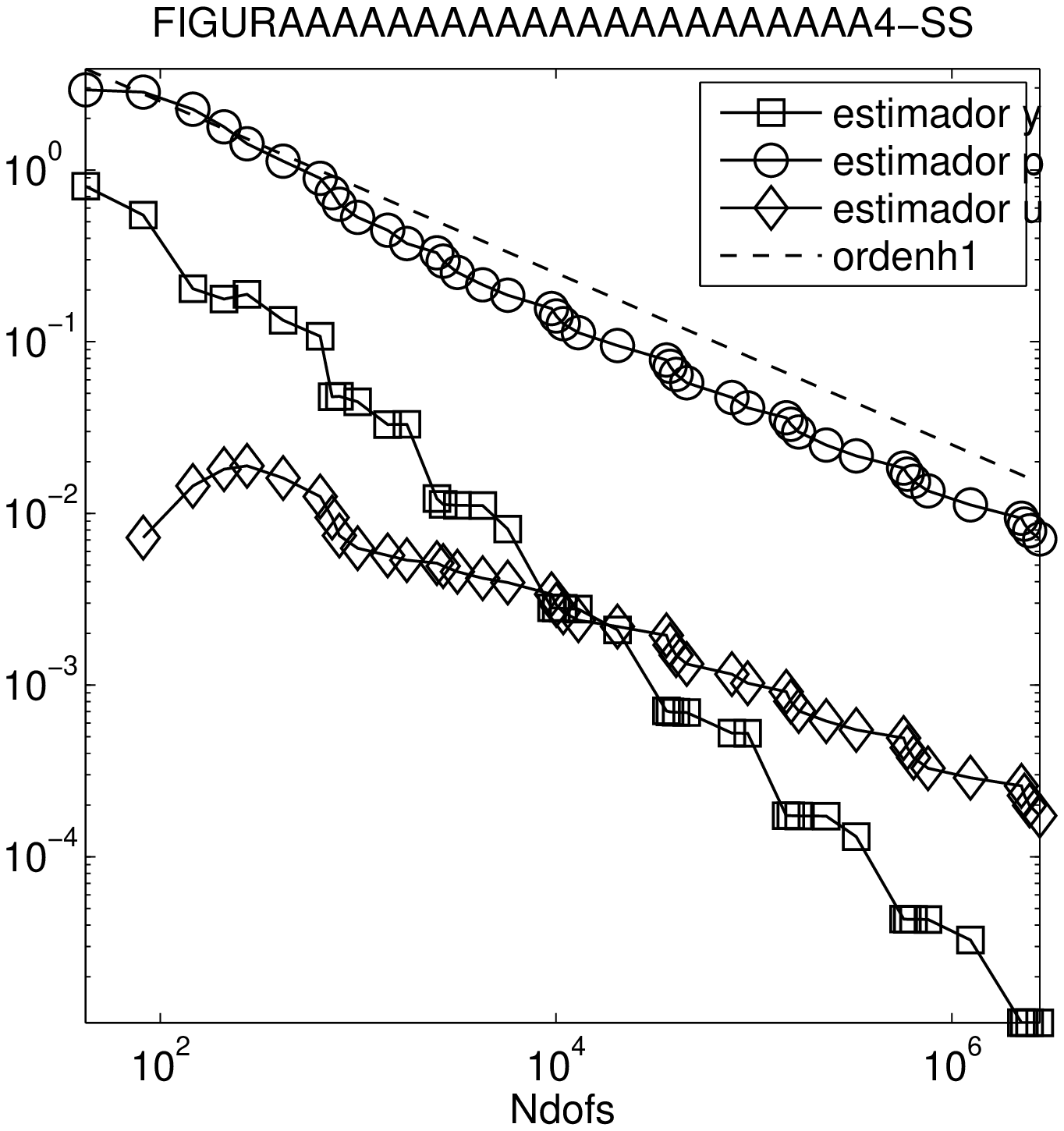}
\includegraphics[width=3.1cm,height=3.4cm,scale=1]{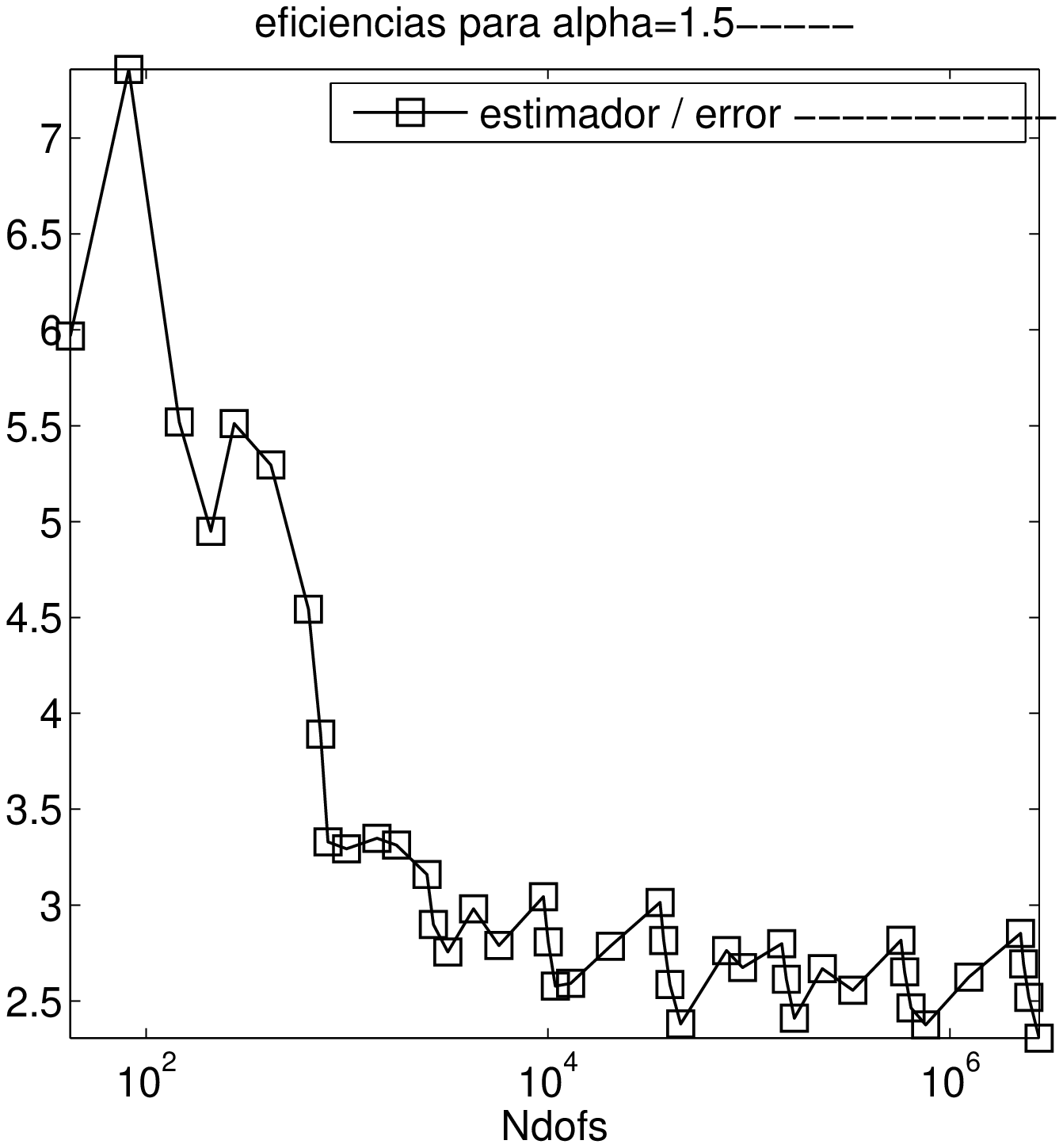}
\caption{Numerical results for Example 3.}
\end{subfigure}
\psfrag{FIGURAAAAAAAAAAAAAAAA5}{\huge $\E_{\ysf}$, $\E_{\psf}$ and $\E_{\usf}$ for $\alpha=1.5$}
\psfrag{error para y}{\huge $e_{\ysf}$}
\psfrag{error para p}{\huge $e_{\psf}$}
\psfrag{error para u}{\huge $e_{\usf}$}
\psfrag{estimador para alpha=1.5}{\huge $\E_{\mathsf{ocp}}$}
\psfrag{estimador para y}{\huge $\E_{\ysf}$}
\psfrag{estimador para p}{\huge $\E_{\psf}$}
\psfrag{estimador para u}{\huge $\E_{\usf}$}
\psfrag{maximum---}{\huge Maximum}
\psfrag{bulk---}{\huge Bulk}
\psfrag{average---}{\huge Average}
\psfrag{ordenh1}{\Large $\textrm{Ndof}^{-1/2}$}
\psfrag{Ndofs}{\huge Ndof}
\begin{subfigure}[b]{0.24\textwidth}
\centering $\bar{\mathsf{y}}_{h}$\\
\includegraphics[width=2.8cm,height=2.8cm,scale=1.5]{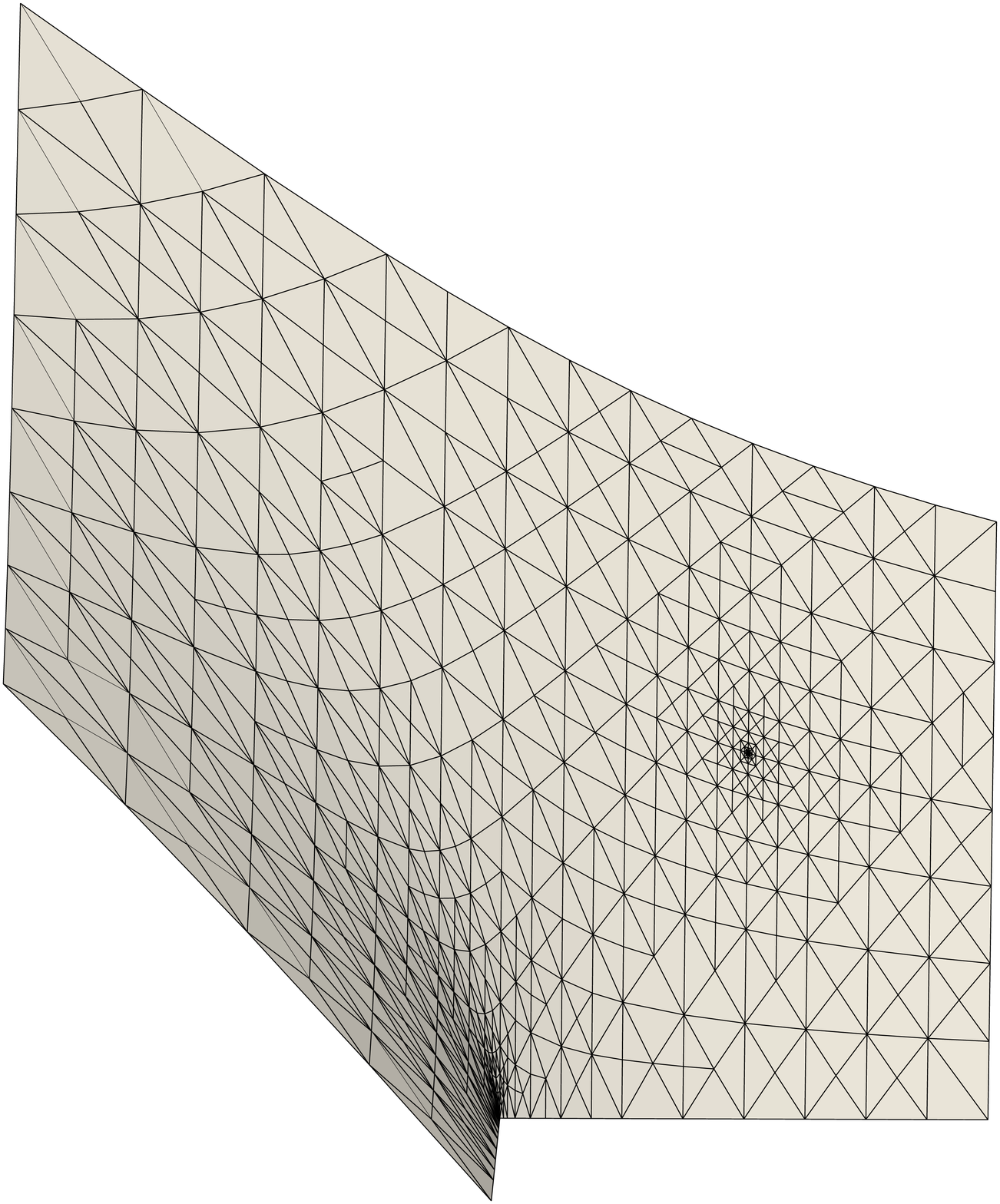}
\end{subfigure}
\begin{subfigure}[b]{0.26\textwidth}
\centering $\bar{\mathsf{p}}_{h}$\\
\includegraphics[width=2.8cm,height=2.8cm,scale=0.5]{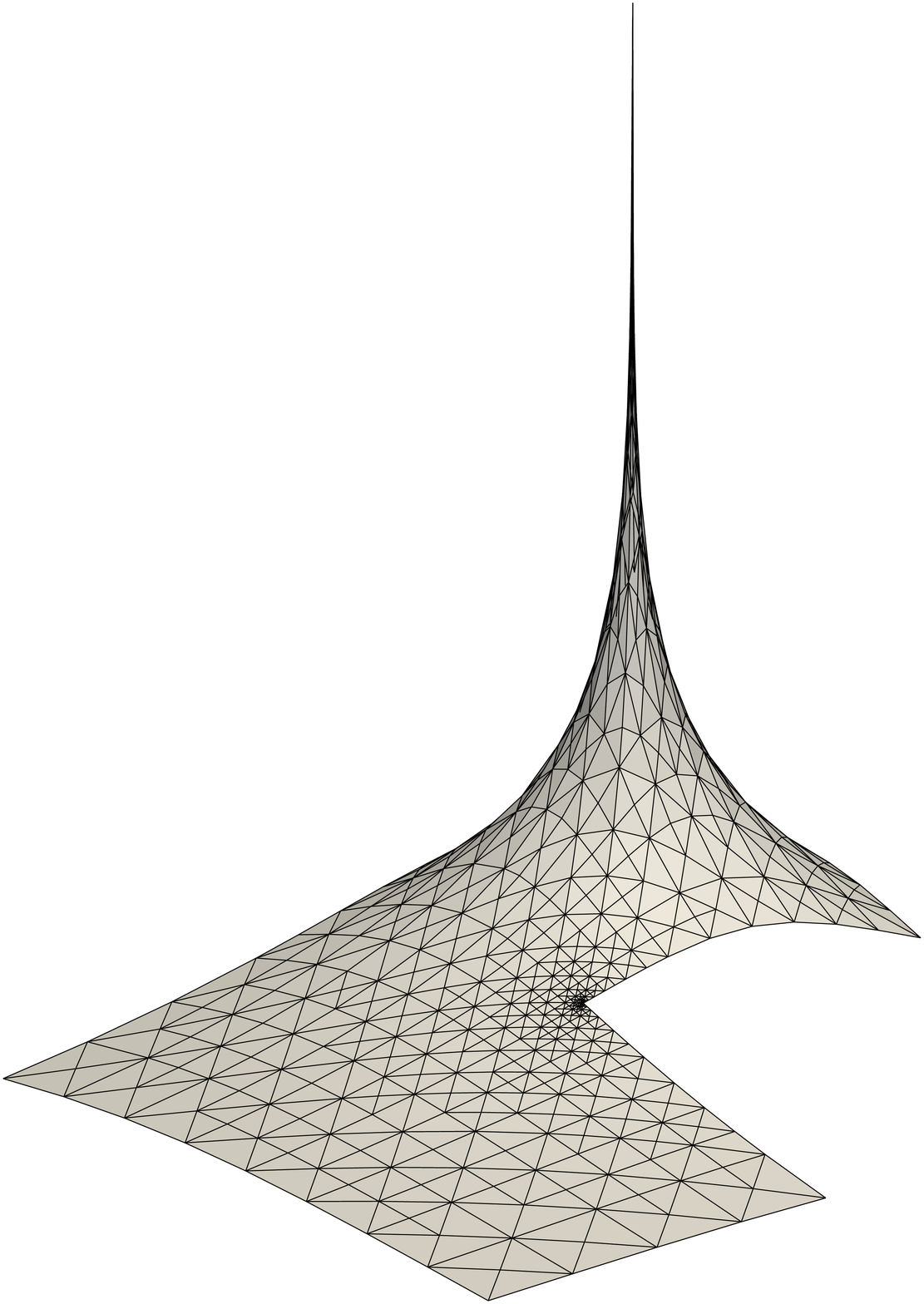}
\end{subfigure}
\begin{subfigure}[b]{0.26\textwidth}
\centering $\bar{\mathsf{u}}_{h}$\\
\includegraphics[width=2.6cm,height=2.8cm,scale=0.5]{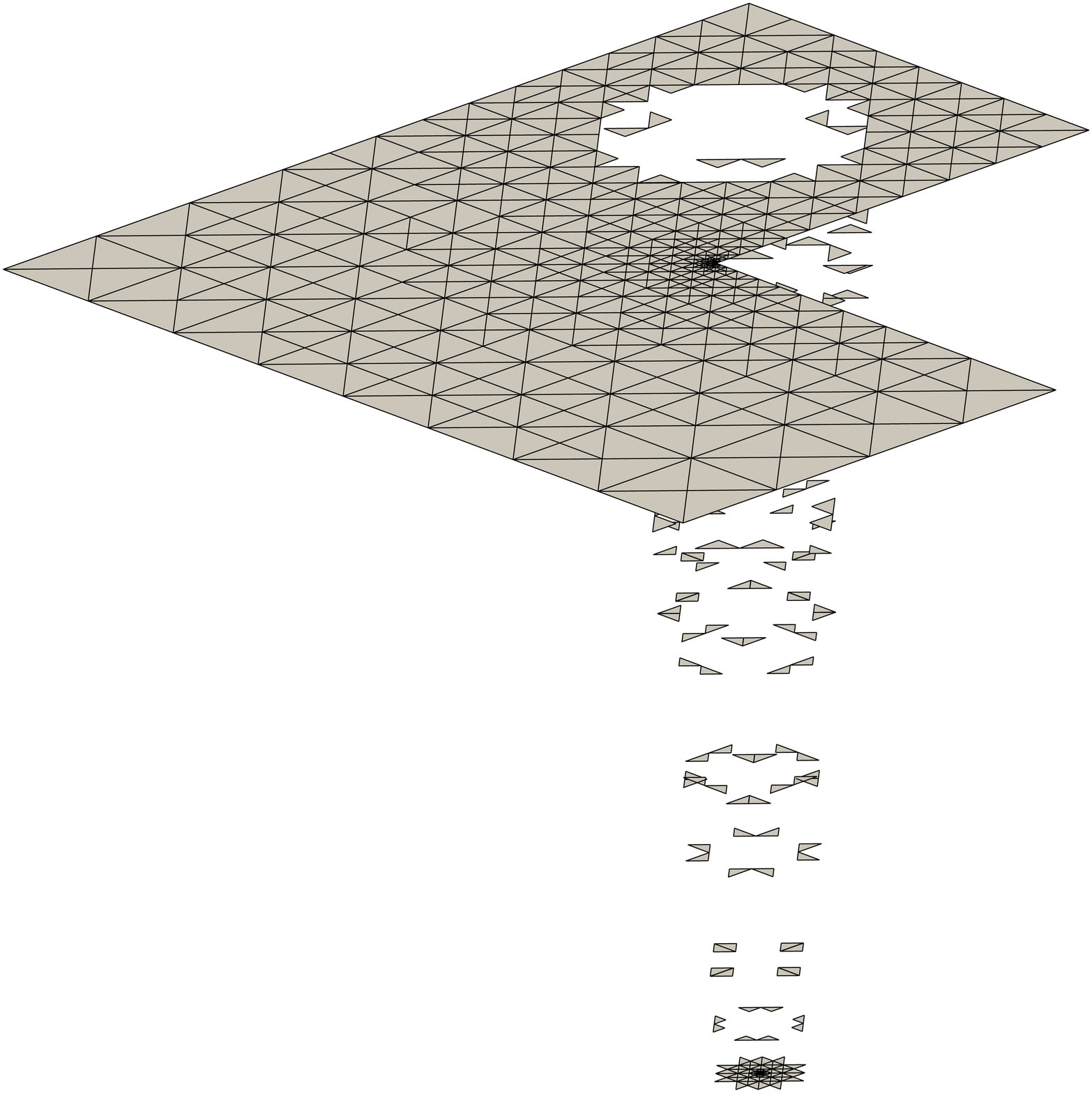}
\end{subfigure}
\begin{subfigure}[b]{0.2\textwidth}
\includegraphics[width=2.5cm,height=2.8cm,scale=0.5]{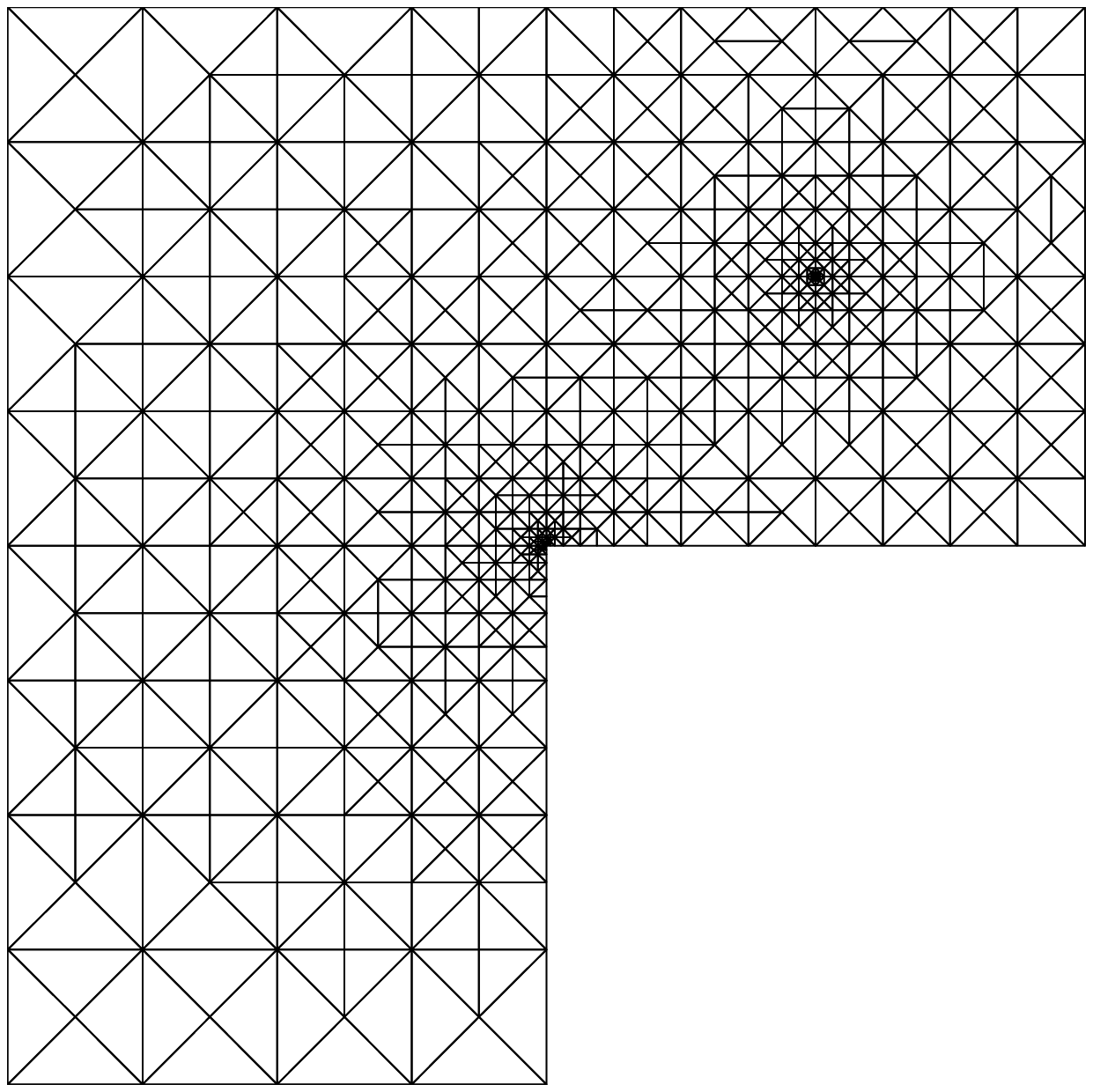}
\end{subfigure}
\\
\psfrag{FIGURAAAAAAAAAAAAAAAAAAAAAA4}{\huge $\|(e_{\ysf},e_{\psf},e_{\usf})\|_{\Omega}$ and $\E_{\mathsf{ocp}}$ for $\alpha=1.5$}
\psfrag{error para alpha=1.5}{\huge $\|(e_{\ysf},e_{\psf},e_{\usf})\|_{\Omega}$}
\psfrag{FIGURAAAAAAAAAAAAAAAAAAAAAA4-S}{\huge Error contributions for $\alpha=1.5$}
\psfrag{FIGURAAAAAAAAAAAAAAAAAAAAAA4-SS}{\huge Estimator contributions for $\alpha=1.5$}
\psfrag{eficiencias para alpha=1.5-----}{\huge Effectivity index for $\alpha=1.5$}
\psfrag{error y--------}{\huge $\|e_\ysf\|_{L^{\infty}(\Omega)}$}
\psfrag{error p--------}{\huge $\|\nabla e_\psf\|_{L^{2}(\rho,\Omega)}$}
\psfrag{error u--------}{\huge $\|e_\usf\|_{L^{2}(\Omega)}$}
\psfrag{estimador y}{\huge $\E_{\mathsf{y}}$}
\psfrag{estimador p}{\huge $\E_{\mathsf{p}}$}
\psfrag{estimador u}{\huge $\E_{\mathsf{u}}$}
\psfrag{estimador / error -----------}{\huge $\E_{\mathsf{ocp}}/\|(e_{\ysf},e_{\psf},e_{\usf})\|_{\Omega}$}
\psfrag{eficiencias y}{\huge $\E_{\mathsf{y}}/e_{\ysf}$}
\psfrag{eficiencias p}{\huge $\E_{\mathsf{p}}/e_{\psf}$}
\psfrag{eficiencias u}{\huge $\E_{\mathsf{u}}/e_{\usf}$}
\begin{subfigure}[b]{1\textwidth}
\includegraphics[width=3.1cm,height=3.4cm,scale=1]{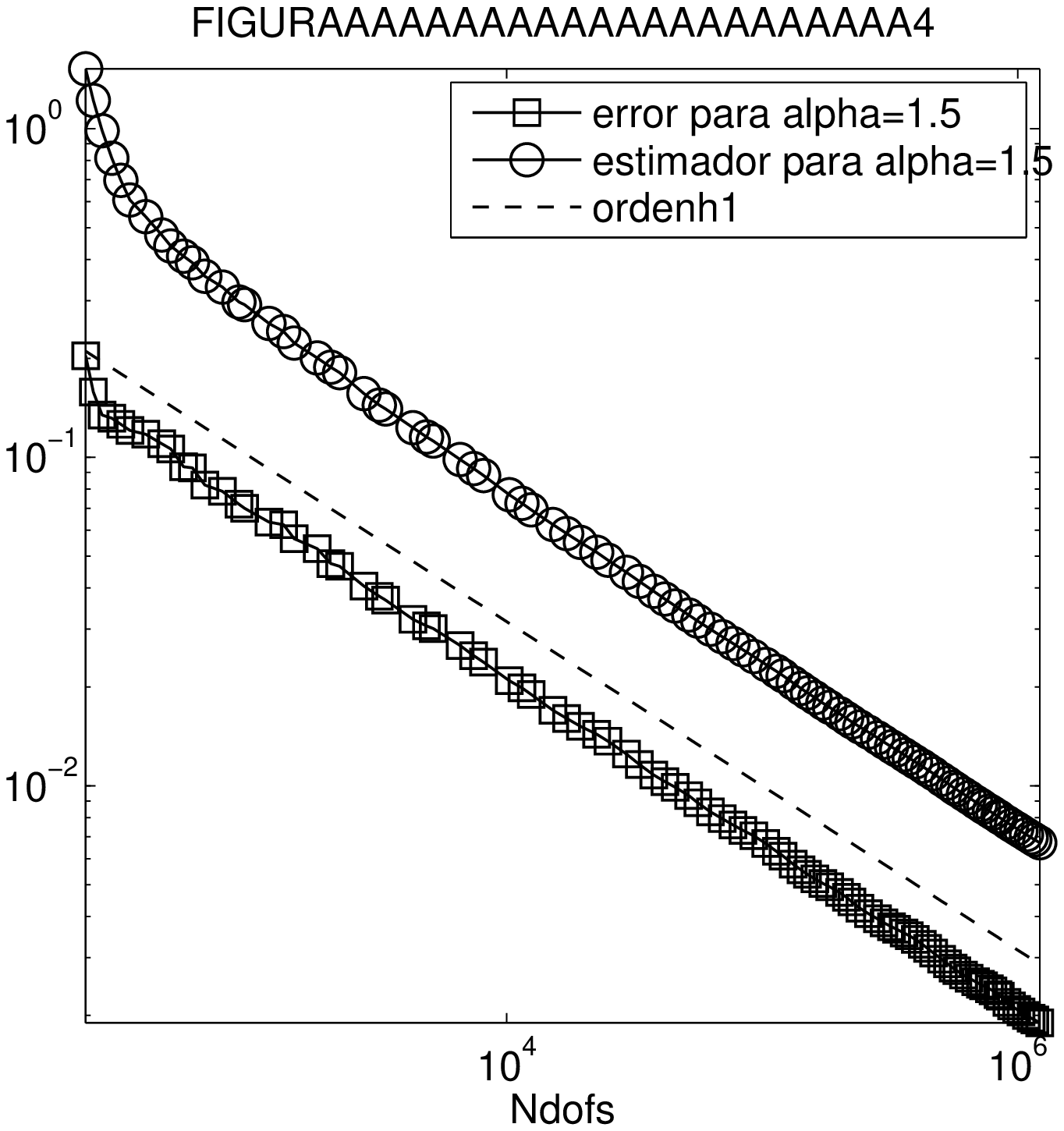}
\includegraphics[width=3.1cm,height=3.4cm,scale=1]{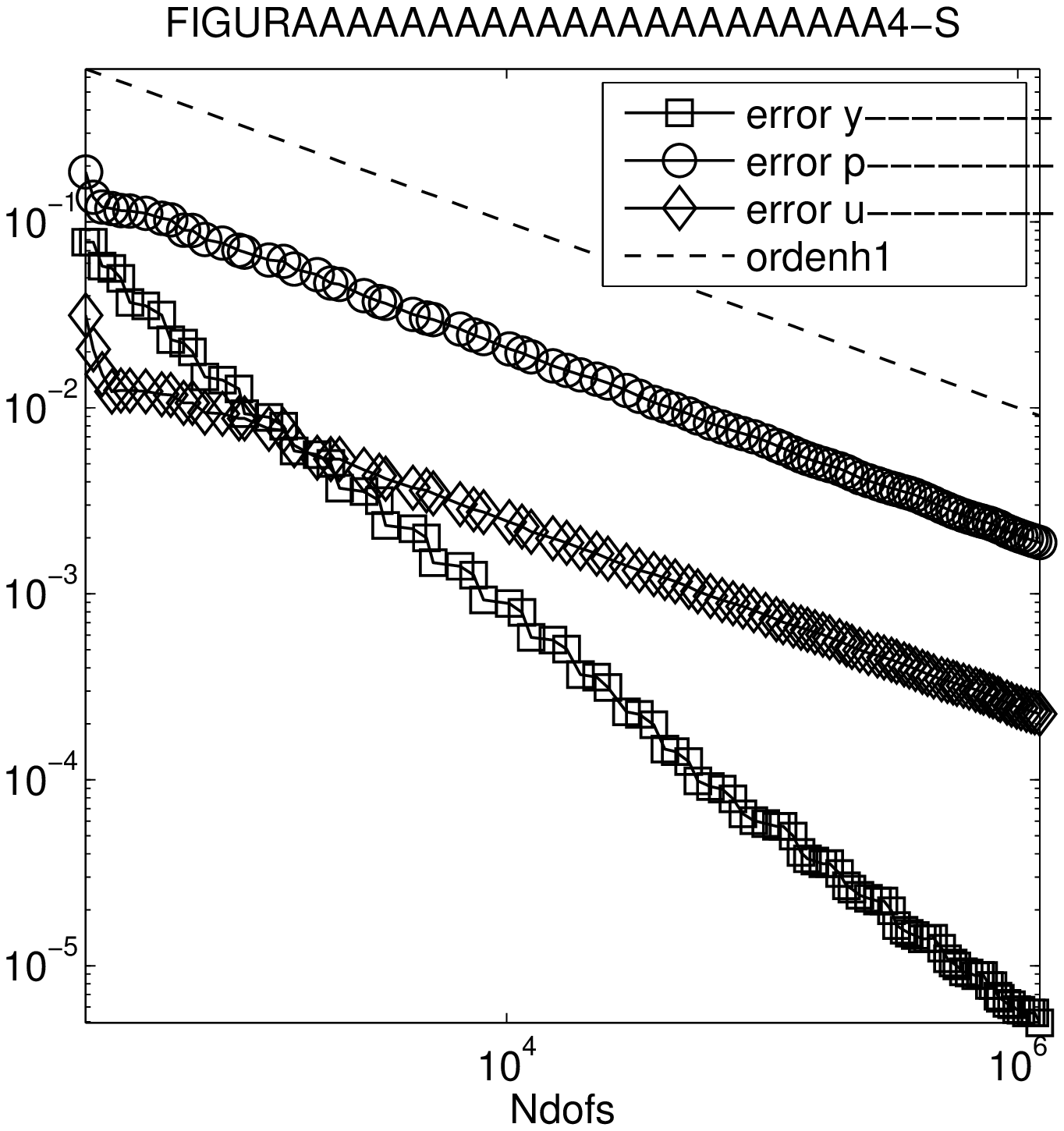}
\includegraphics[width=3.1cm,height=3.4cm,scale=1]{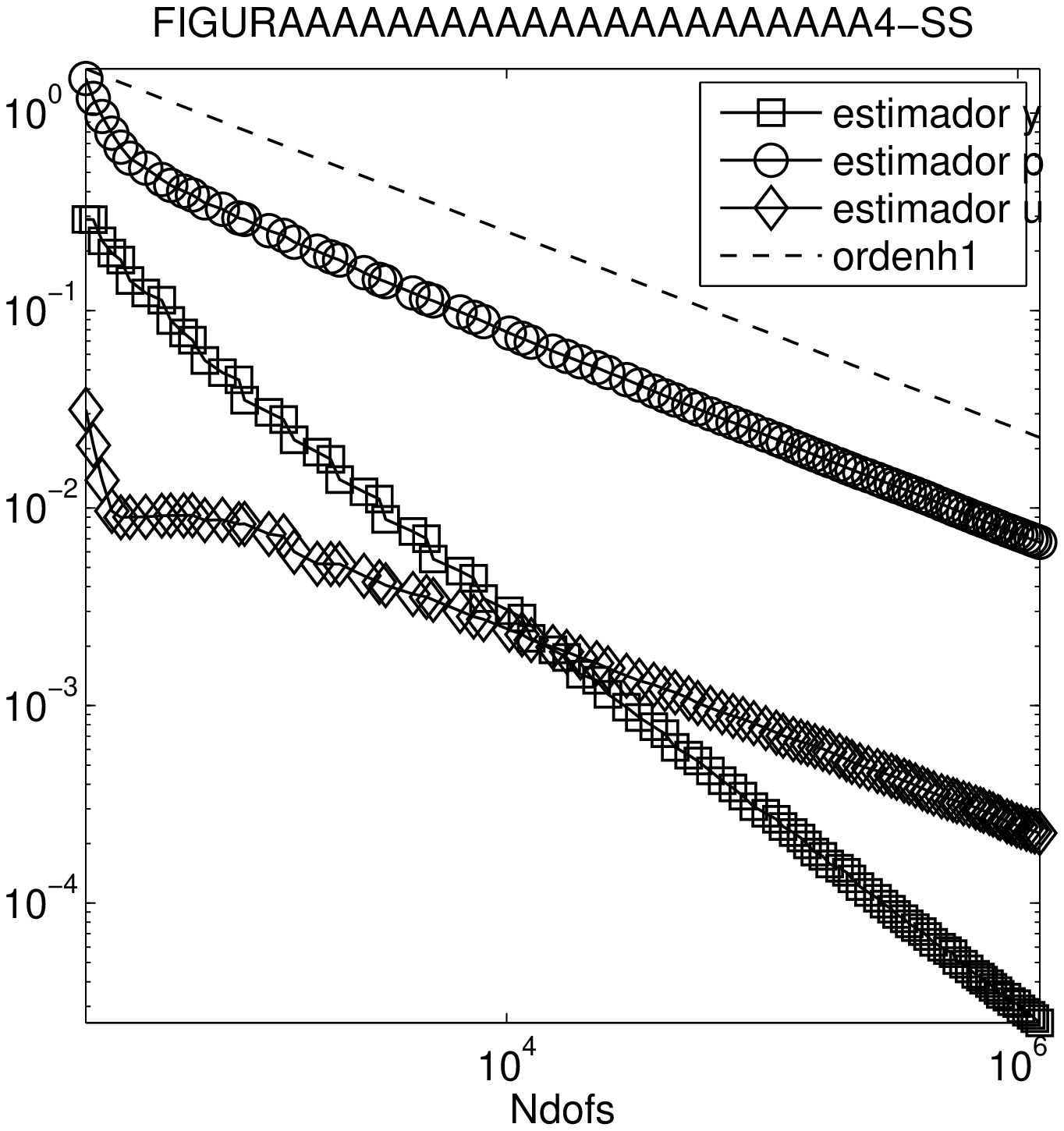}
\includegraphics[width=3.1cm,height=3.4cm,scale=1]{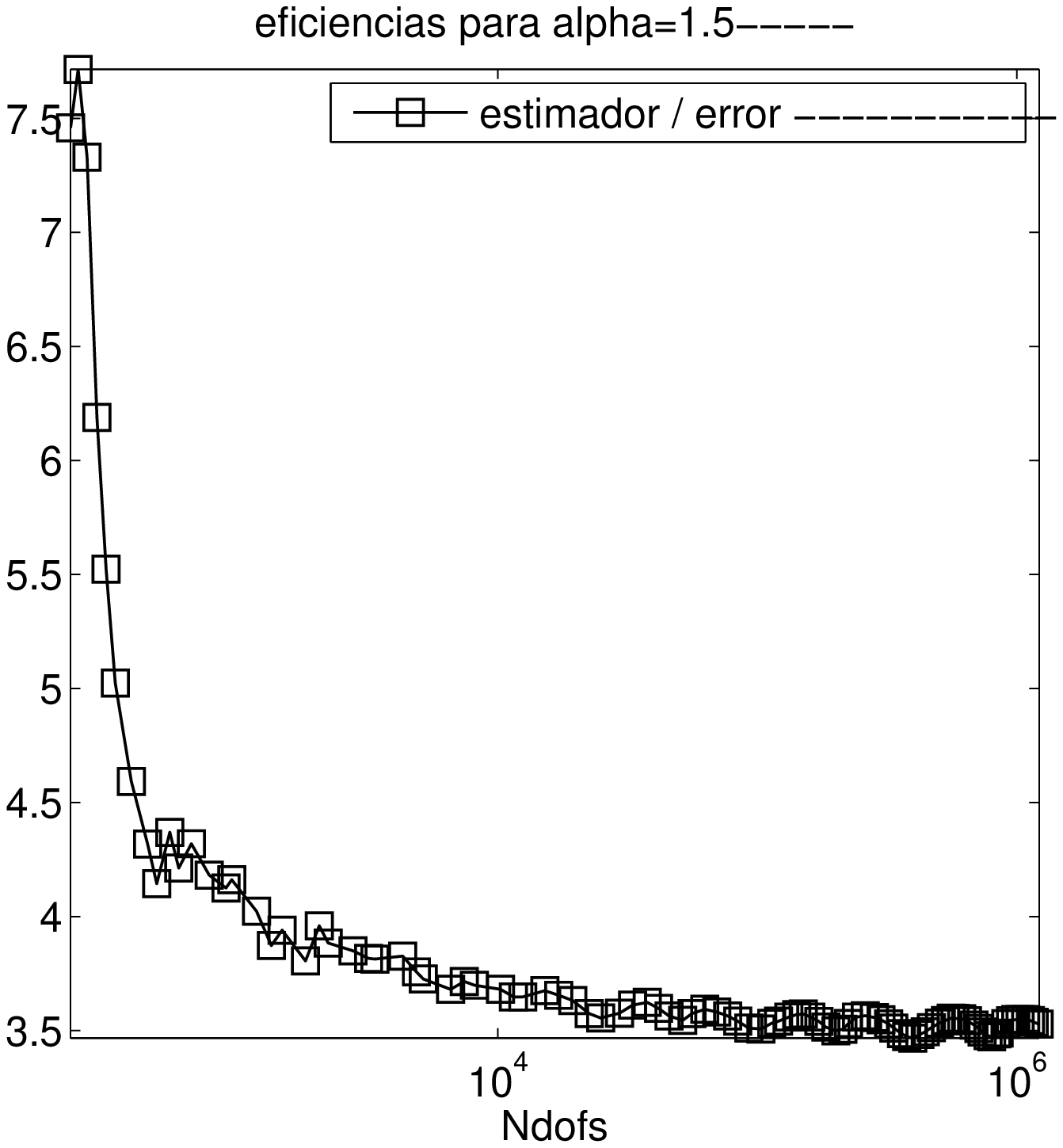}
\caption{Numerical results for Example 4.}
\end{subfigure}
\caption{\AAF{Examples 3 and 4: For $\alpha = 1.5$ and based on a maximum refinement strategy, we show the finite element solutions for the optimal state $\bar{\ysf}_{\T}$, adjoint state $\bar{\psf}_{\T}$, and optimal control $\bar{\usf}_{\T}$, on the mesh obtained after 20 steps of the adaptive loop. Also, convergence rates for the total error $\|(e_{\ysf},e_{\psf},e_{\usf})\|_{\Omega}$, error estimator $\E_{\mathsf{ocp}}$, their individual contributions and effectivity indices.}}
\label{Fig:Ex3-4}
\end{figure}
\subsection{Three dimensional examples}
\label{3d_examples}

\AAF{We show three dimensional examples, for the domain $\Omega=(0,1)^{3}$, and as before, with homogeneous and inhomogeneous Dirichlet boundary conditions and different number of source points.}

\noindent \textbf{Example \RR{5}:} We set $\asf=1$, $\bsf=10$,  $\fsf \equiv0$, and
\begin{gather*}
  \calZ=\{(0.25,0.25,0.25),(0.75,0.75,0.75)\},
\end{gather*}
and for all the desired tracking points $z\in\calZ$ we set $\ysf_{z}=1$. \RR{The exact solution is not known. The results are shown in Figure \ref{Fig:Ex5}.}
~\\
\noindent \textbf{Example \RR{6}:} We set $\asf=-1.5$ \RR{and} $\bsf=0.2$\RR{. The} optimal state is 
\[
  \bar{\ysf}(x_{1},x_{2},x_{3})=128x_1 x_2 x_3(1-x_1)(1-x_2)(1-x_3),
\]
while the optimal adjoint is given by \eqref{exact_adjoint}. The set of observation points is $\calZ=\{(0.5,0.5,0.5)\}$ and $\ysf_{(0.5,0.5,0.5)}=1$. \RR{The results are shown in Figure \ref{Fig:Ex6}.}
~\\
\noindent \textbf{Example \RR{7}:} \RR{We} set $\asf=-15$ and $\bsf=-5$. The optimal state is
\[
  \bar{\ysf}(x_{1},x_{2},x_{3})=\frac{8192}{27}x_1 x_2 x_3(1-x_1)(1-x_2)(1-x_3),
\]
whereas the optimal adjoint state is defined in \eqref{exact_adjoint}. The set of observation points is 
\begin{align*}
  \calZ = &\left\{(0.25,0.25,0.25),(0.25,0.25,0.75),(0.25,0.75,0.25),(0.25,0.75,0.75), \right. \\
  &\left. (0.75,0.25,0.25),(0.75,0.25,0.75),(0.75,0.75,0.25),(0.75,0.75,0.25) \right\}
\end{align*}
and we set $\ysf_z=1$ for all $z \in \calZ$. \RR{The results are shown in Figure \ref{Fig:Ex7}.}

\begin{figure}[!h]
\begin{center}
\psfrag{total estimator versus, for alpha=1.99}{\huge $\E_{\mathsf{ocp}}$ Adaptive vs Uniform for $\alpha=1.99$}
\psfrag{Ndof}{\huge Ndof}
\psfrag{total estimator (uniform)}{\LARGE $\E_{\mathsf{ocp}}$-Uniform}
\psfrag{total estimator (adaptive)}{\LARGE $\E_{\mathsf{ocp}}$-Adaptive}
\psfrag{1d3}{\LARGE Ndof$^{-1/3}$}
\psfrag{1d8}{\LARGE Ndof$^{-1/8}$}
\scalebox{0.23}{\includegraphics{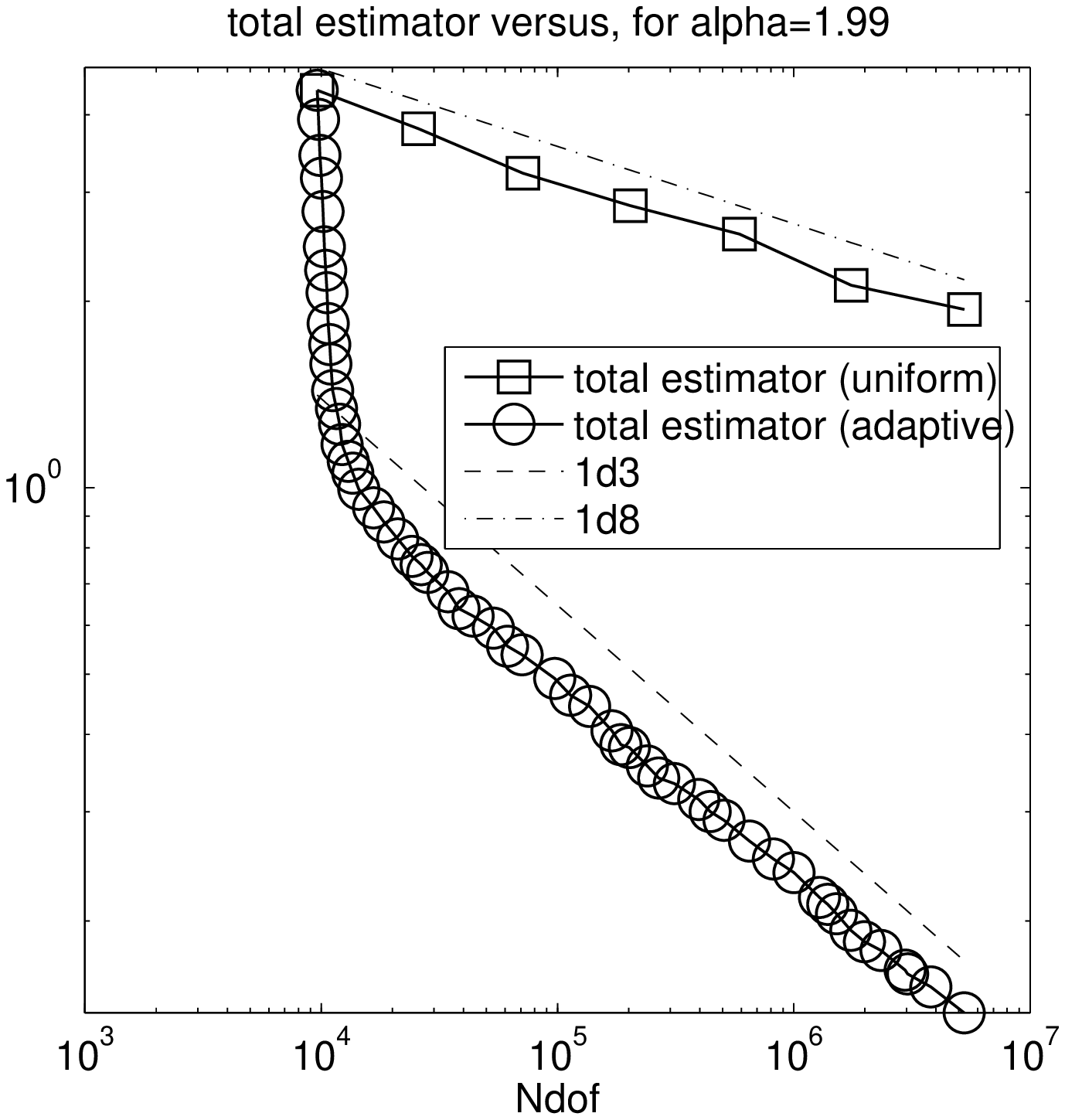}}
\psfrag{estimator contributions, for alpha=1.99}{\huge Estimator contributions, for $\alpha=1.99$}
\psfrag{Ndof}{\huge Ndof}
\psfrag{state estimator}{\LARGE $\E_{\mathsf{y}}$-Uniform}
\psfrag{adjoint estimator}{\LARGE $\E_{\mathsf{p}}$-Uniform}
\psfrag{control estimator}{\LARGE $\E_{\mathsf{u}}$-Uniform}
\psfrag{1d3}{\LARGE Ndof$^{-1/3}$}
\psfrag{5d12}{\LARGE Ndof$^{-5/12}$}
\scalebox{0.23}{\includegraphics{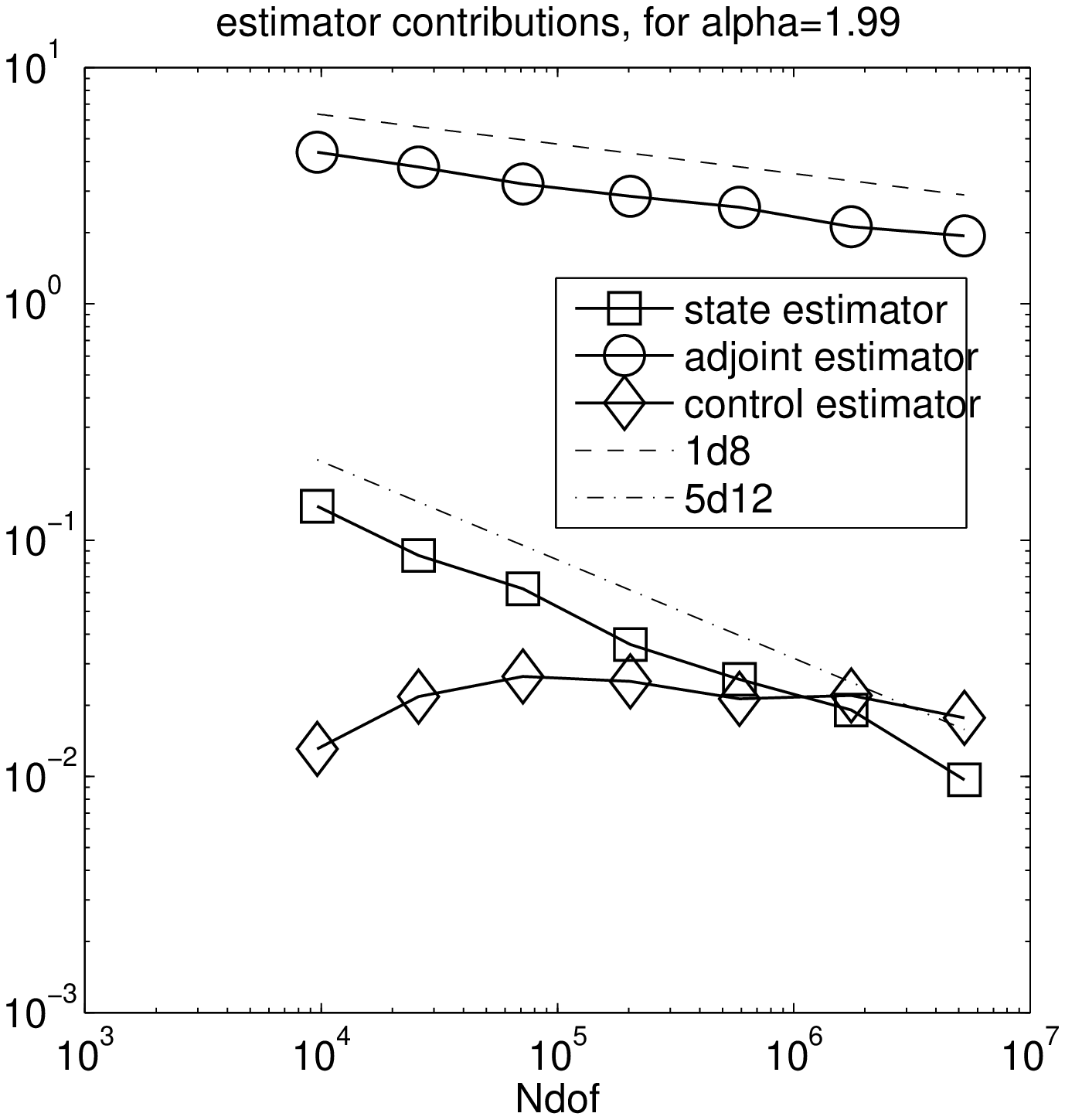}}
\psfrag{estimator contributions, for alpha=1.99}{\huge Estimator contributions, for $\alpha=1.99$}
\psfrag{Ndof}{\huge Ndof}
\psfrag{state estimator}{\LARGE $\E_{\mathsf{y}}$-Adaptive}
\psfrag{adjoint estimator}{\LARGE $\E_{\mathsf{p}}$-Adaptive}
\psfrag{control estimator}{\LARGE $\E_{\mathsf{u}}$-Adaptive}
\psfrag{1d3}{\LARGE Ndof$^{-1/3}$}
\psfrag{2d3}{\LARGE Ndof$^{-2/3}$}
\scalebox{0.23}{\includegraphics{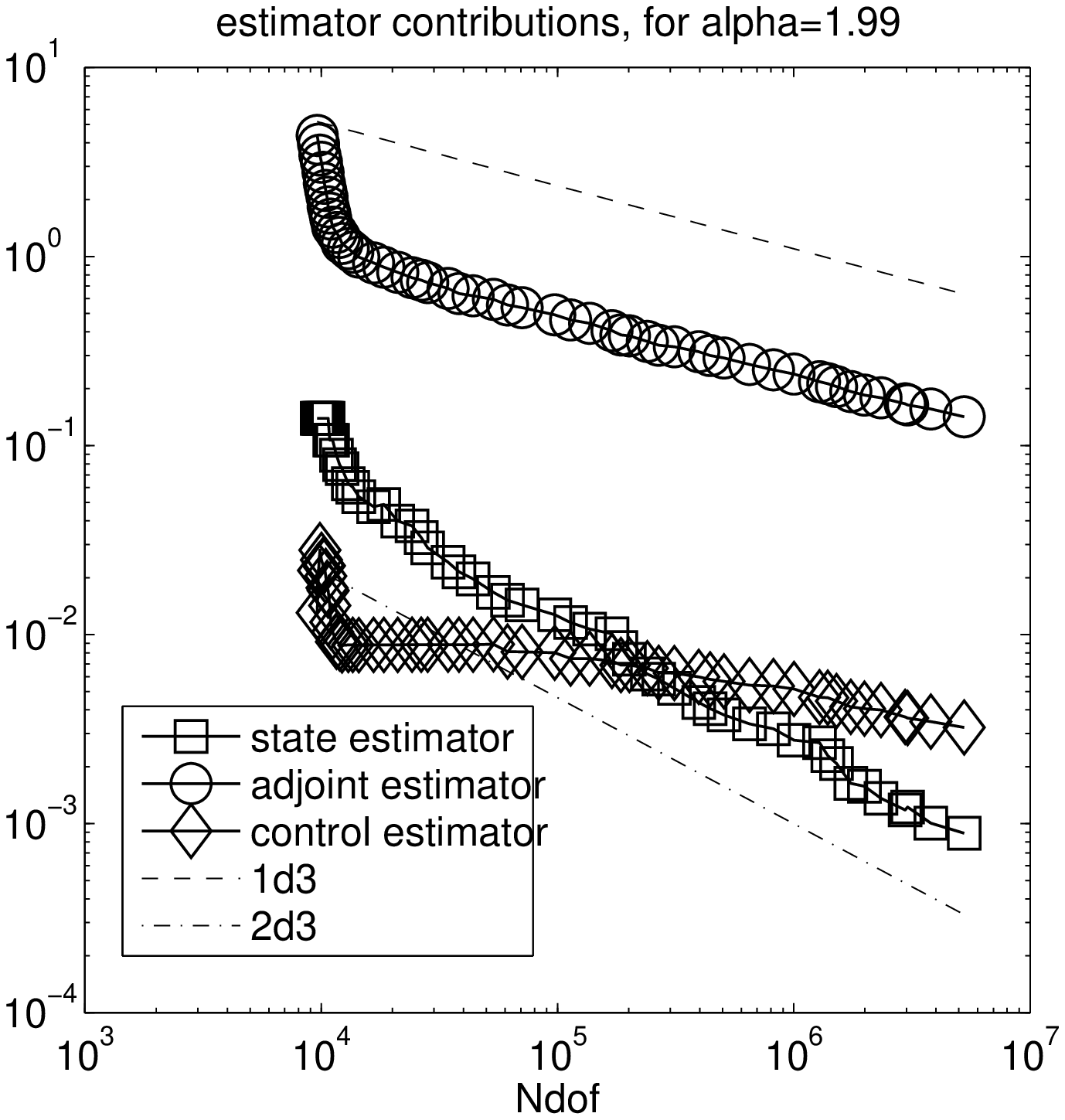}}
\scalebox{0.08}{\includegraphics{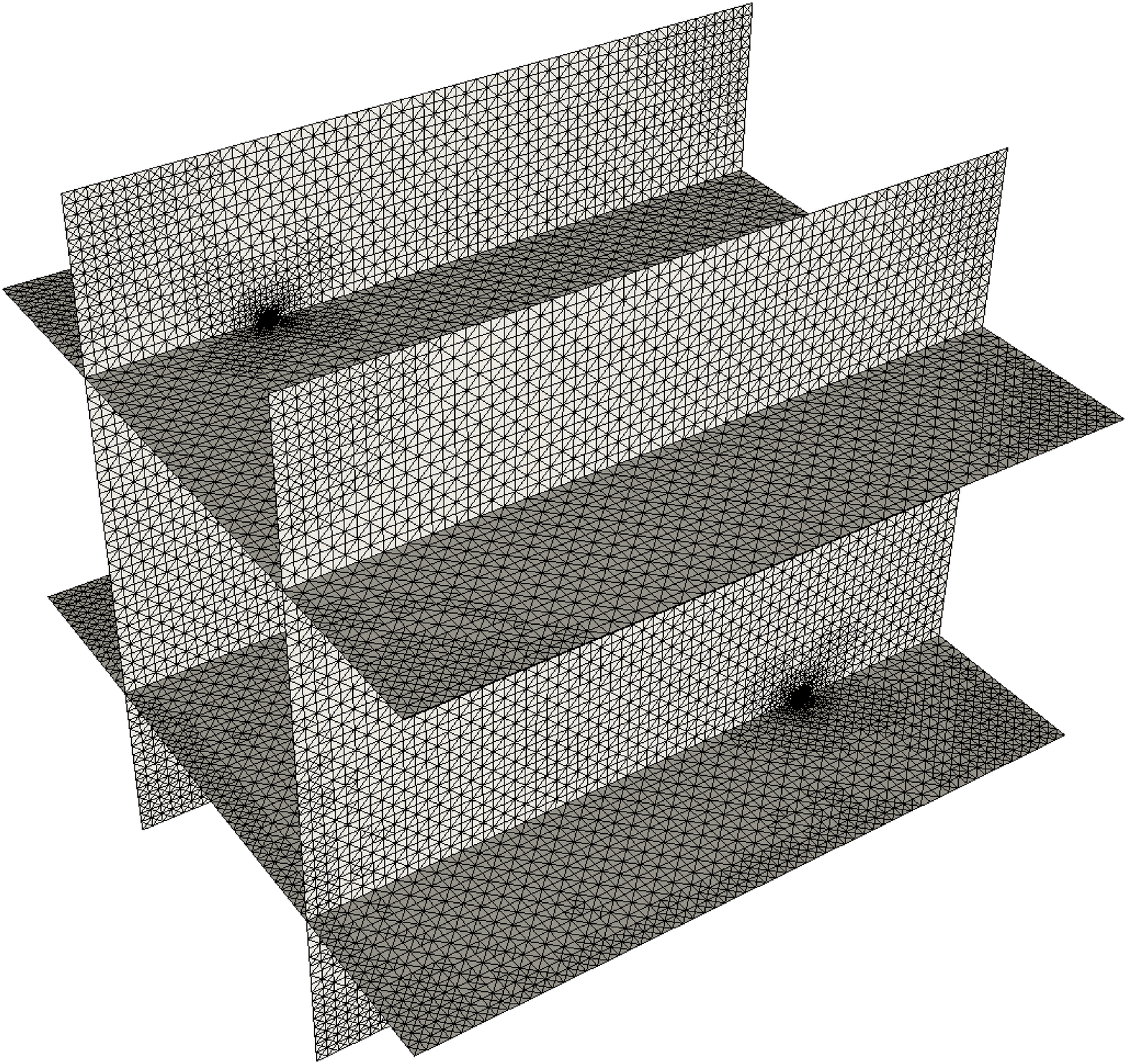}}
\end{center}
\caption{\RR{Example 5: For $\alpha = 1.99$ and both uniform refinement and adaptive refinement using a maximum strategy, we show the convergence rates for the error estimator $\E_{\mathsf{ocp}}$, its individual contributions and slices of the final adaptively refined mesh.}}
\label{Fig:Ex5}
\end{figure}
\begin{figure}[!h]
\begin{center}
\psfrag{total estimator and error, for alpha=1.99}{\huge $\|(e_{\ysf},e_{\psf},e_{\usf})\|_{\Omega}$ and $\E_{\mathsf{ocp}}$, for $\alpha=1.99$}
\psfrag{Ndof}{\huge Ndof}
\psfrag{total error ypu}{\LARGE $\|(e_{\ysf},e_{\psf},e_{\usf})\|_{\Omega}$}
\psfrag{total estimator ypu}{\LARGE $\E_{\mathsf{ocp}}$}
\psfrag{O(1)}{\LARGE Ndof$^{-1/3}$}
\scalebox{0.23}{\includegraphics{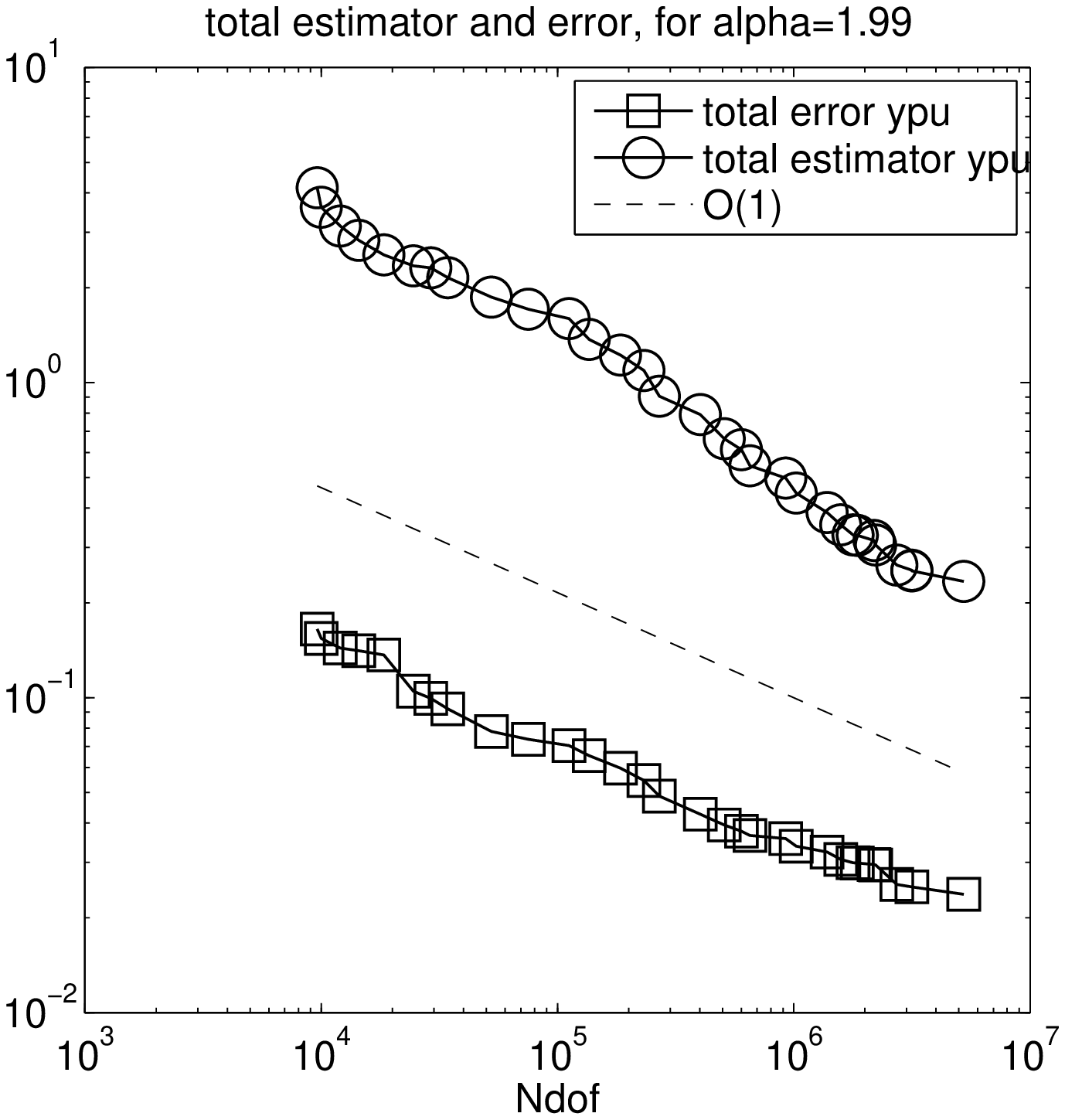}}
\psfrag{error contributions, for alpha=1.99}{\huge Error contributions, for $\alpha=1.99$}
\psfrag{Ndof}{\huge Ndof}
\psfrag{norm of state error}{\LARGE $\|e_\ysf\|_{L^{\infty}(\Omega)}$}
\psfrag{norm of adjoint error}{\LARGE $\|\nabla e_\psf\|_{L^{2}(\rho,\Omega)}$}
\psfrag{norm of control error}{\LARGE $\|e_\usf\|_{L^{2}(\Omega)}$}
\psfrag{O(1)}{\LARGE Ndof$^{-1/3}$}
\scalebox{0.23}{\includegraphics{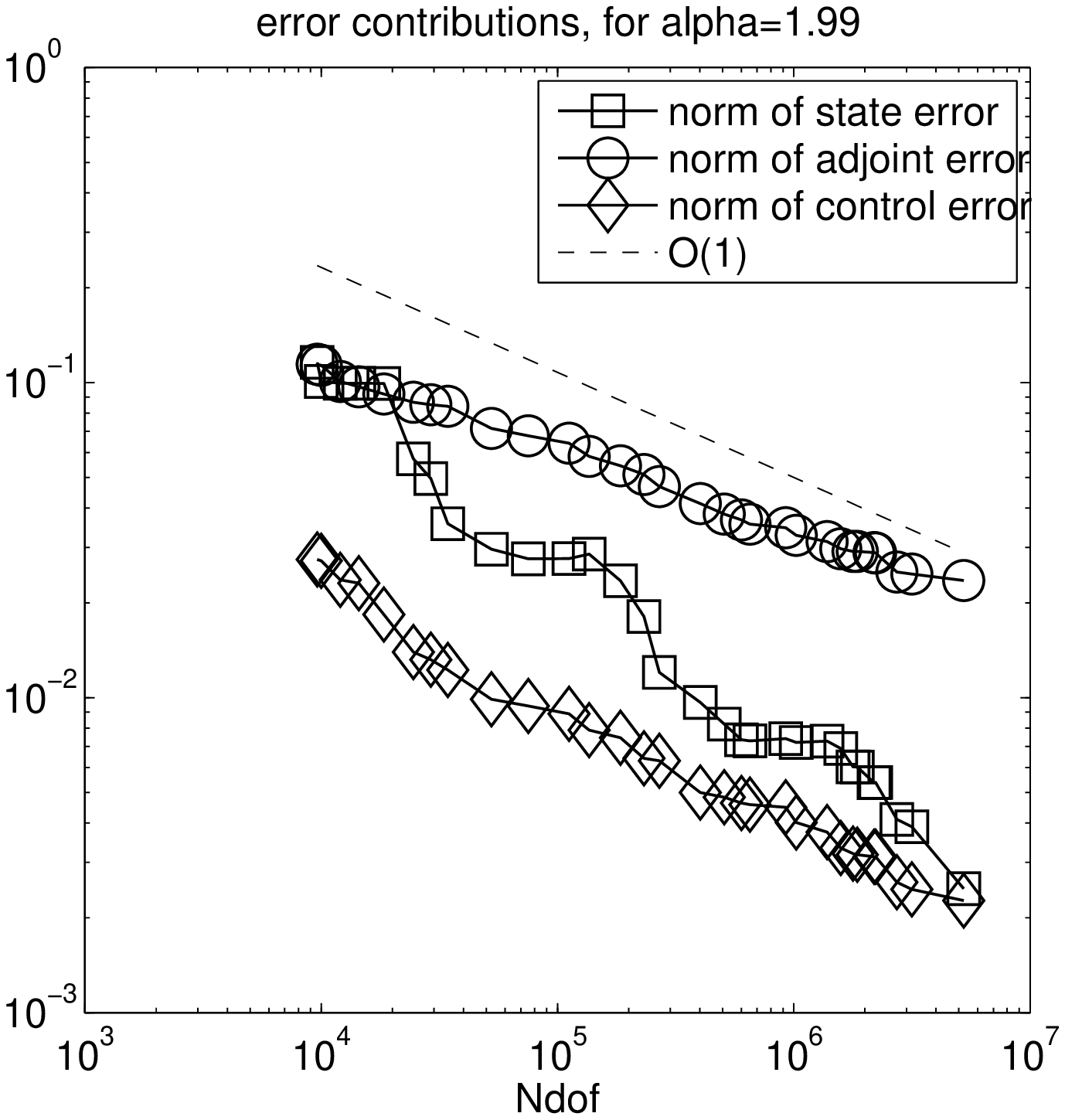}}
\psfrag{estimator contributions, for alpha=1.99}{\huge Estimator contributions, for $\alpha=1.99$}
\psfrag{Ndof}{\huge Ndof}
\psfrag{state estimator}{\LARGE $\E_{\mathsf{y}}$}
\psfrag{adjoint estimator}{\LARGE $\E_{\mathsf{p}}$}
\psfrag{control estimator}{\LARGE $\E_{\mathsf{u}}$}
\psfrag{O(1)}{\LARGE Ndof$^{-1/3}$}
\scalebox{0.23}{\includegraphics{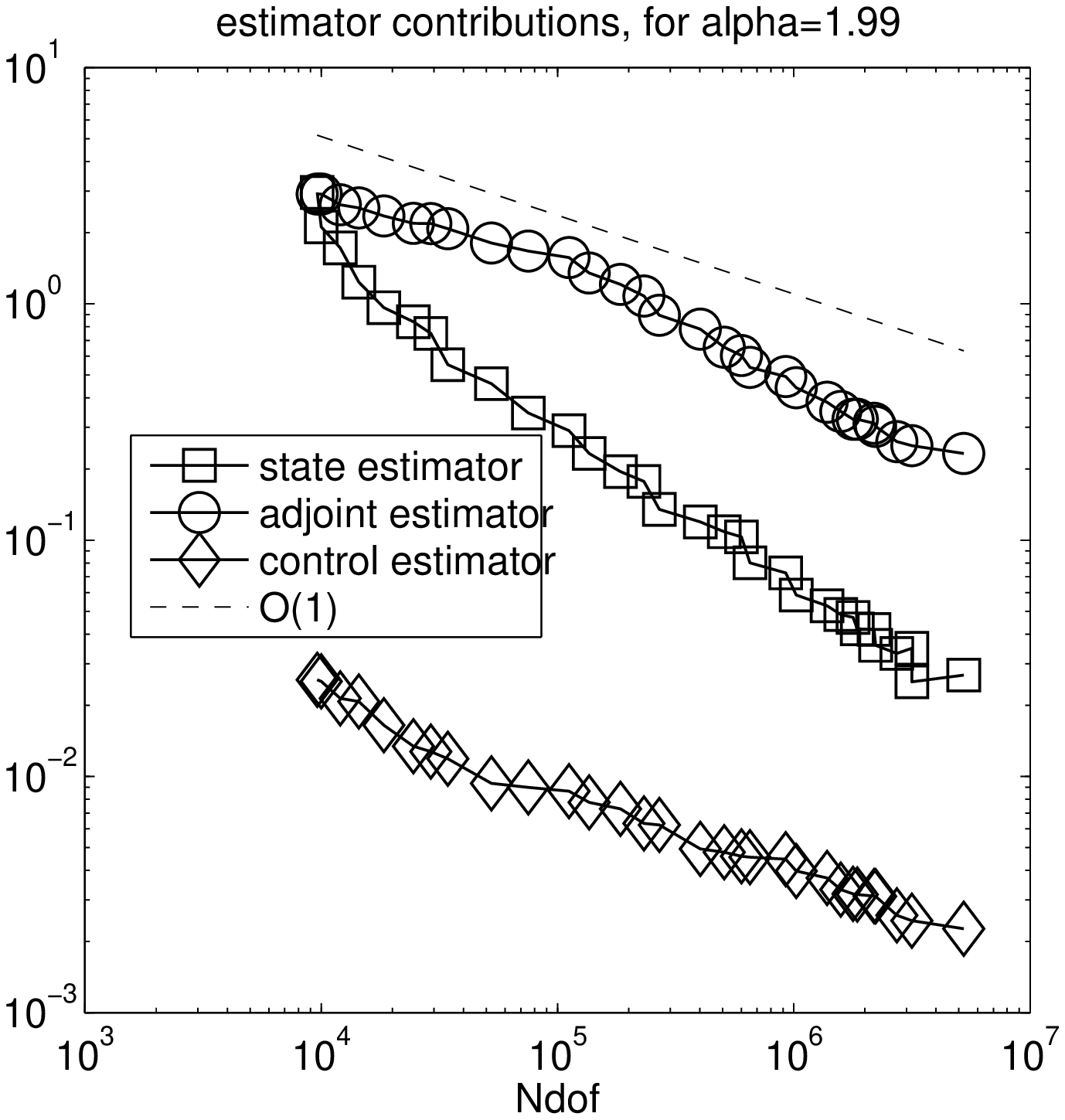}}
\psfrag{total effectivity, for alpha=1.99}{\huge Effectivity index, for $\alpha=1.99$}
\psfrag{Ndof}{\huge Ndof}
\psfrag{total effectivity index ede}{\LARGE $\E_{\mathsf{ocp}}/\|(e_{\ysf},e_{\psf},e_{\usf})\|_{\Omega}$}
\psfrag{O(1)}{\LARGE Ndof$^{-1/3}$}
\scalebox{0.23}{\includegraphics{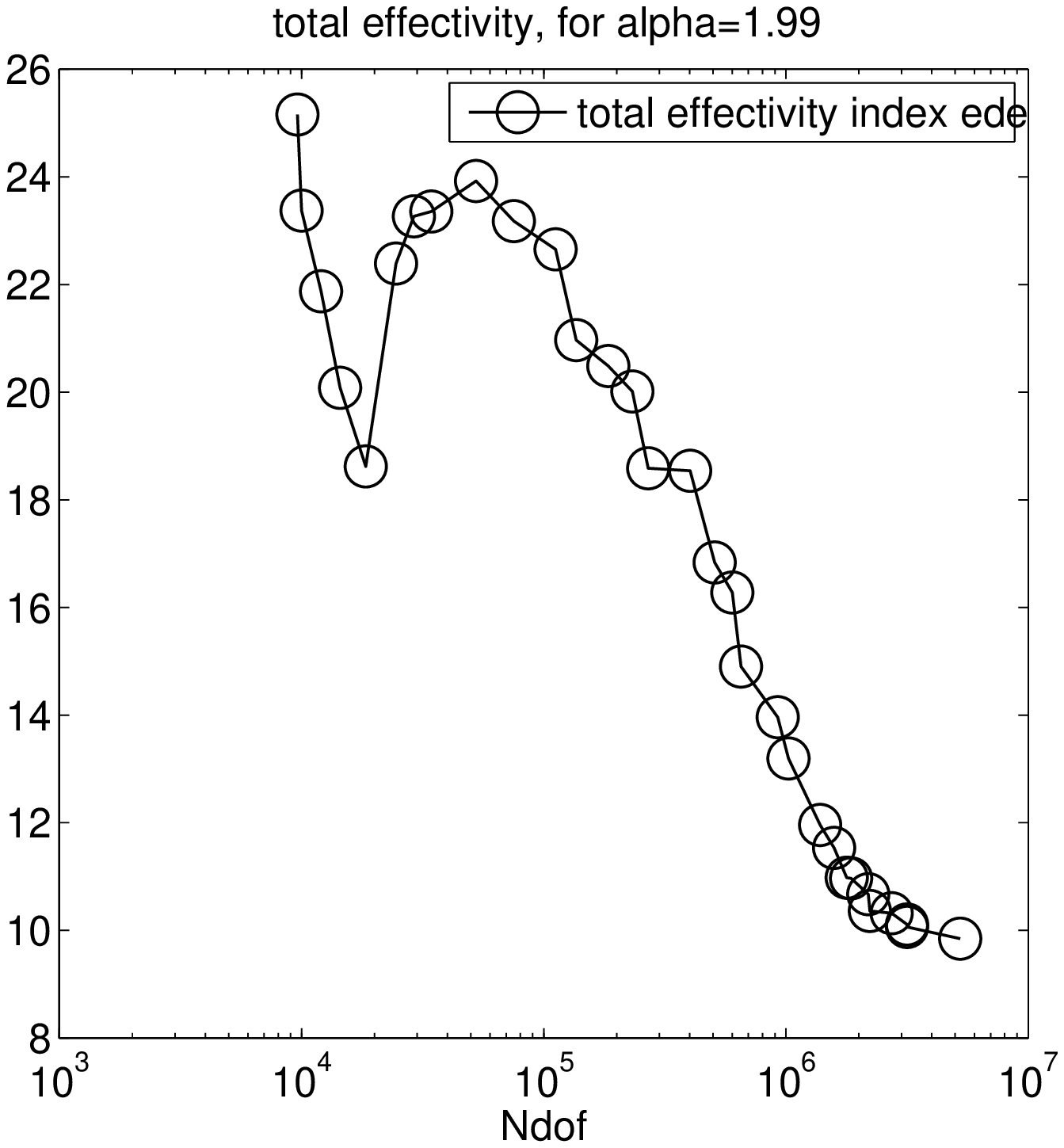}}
\scalebox{0.08}{\includegraphics{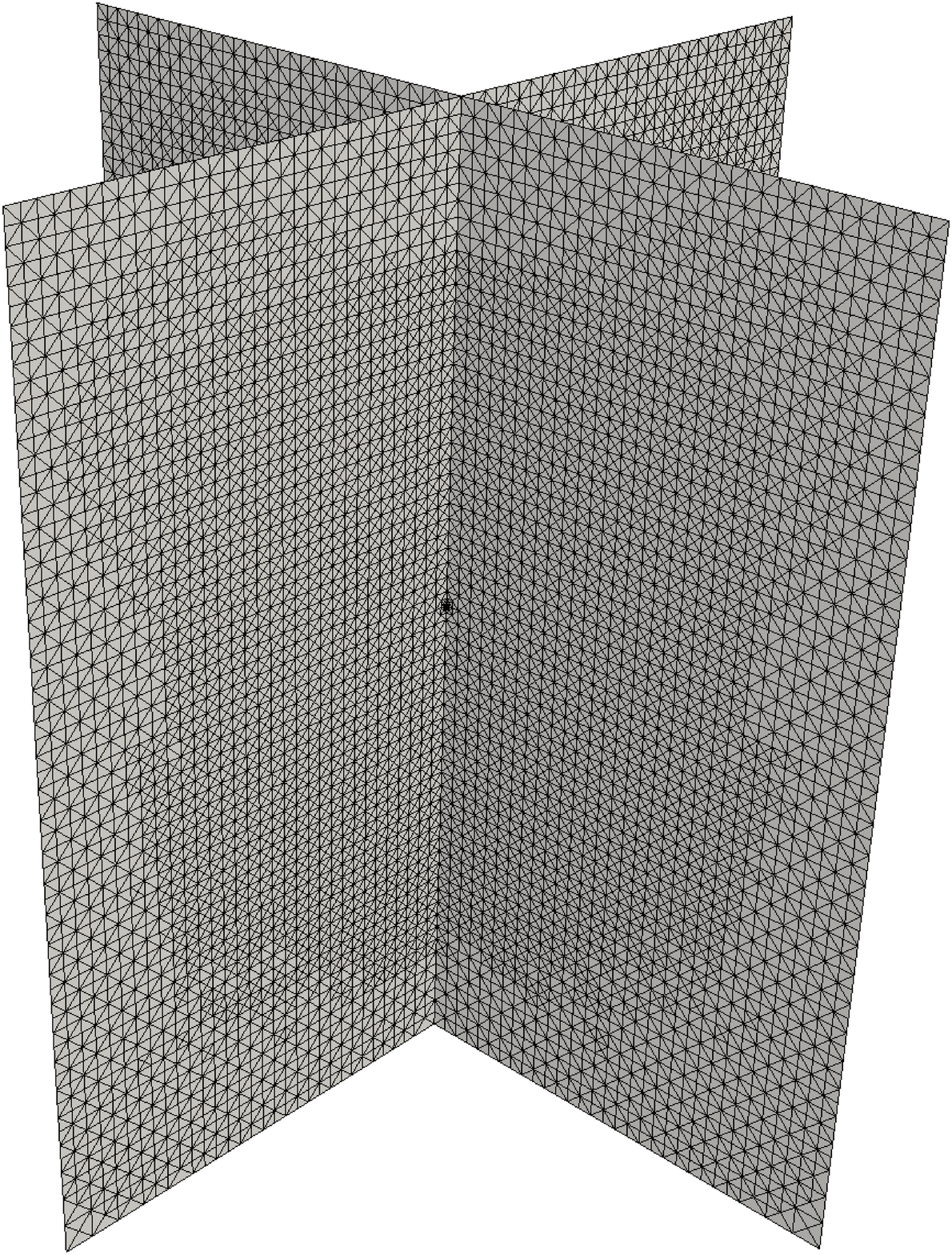}}
\end{center}
\caption{\RR{Example 6: For $\alpha = 1.99$ and adaptive refinement using a maximum strategy, we show the convergence rates for the total error $\|(e_{\ysf},e_{\psf},e_{\usf})\|_{\Omega}$, error estimator $\E_{\mathsf{ocp}}$, their individual contributions, effectivity indices and slices of the final adaptively refined mesh.}}
\label{Fig:Ex6}
\end{figure}
\begin{figure}[!h]
\begin{center}
\psfrag{total estimator and error, for alpha=1.99}{\huge $\|(e_{\ysf},e_{\psf},e_{\usf})\|_{\Omega}$ and $\E_{\mathsf{ocp}}$, for $\alpha=1.99$}
\psfrag{Ndof}{\huge Ndof}
\psfrag{total error ypu}{\LARGE $\|(e_{\ysf},e_{\psf},e_{\usf})\|_{\Omega}$}
\psfrag{total estimator ypu}{\LARGE $\E_{\mathsf{ocp}}$}
\psfrag{O(1)}{\LARGE Ndof$^{-1/3}$}
\scalebox{0.23}{\includegraphics{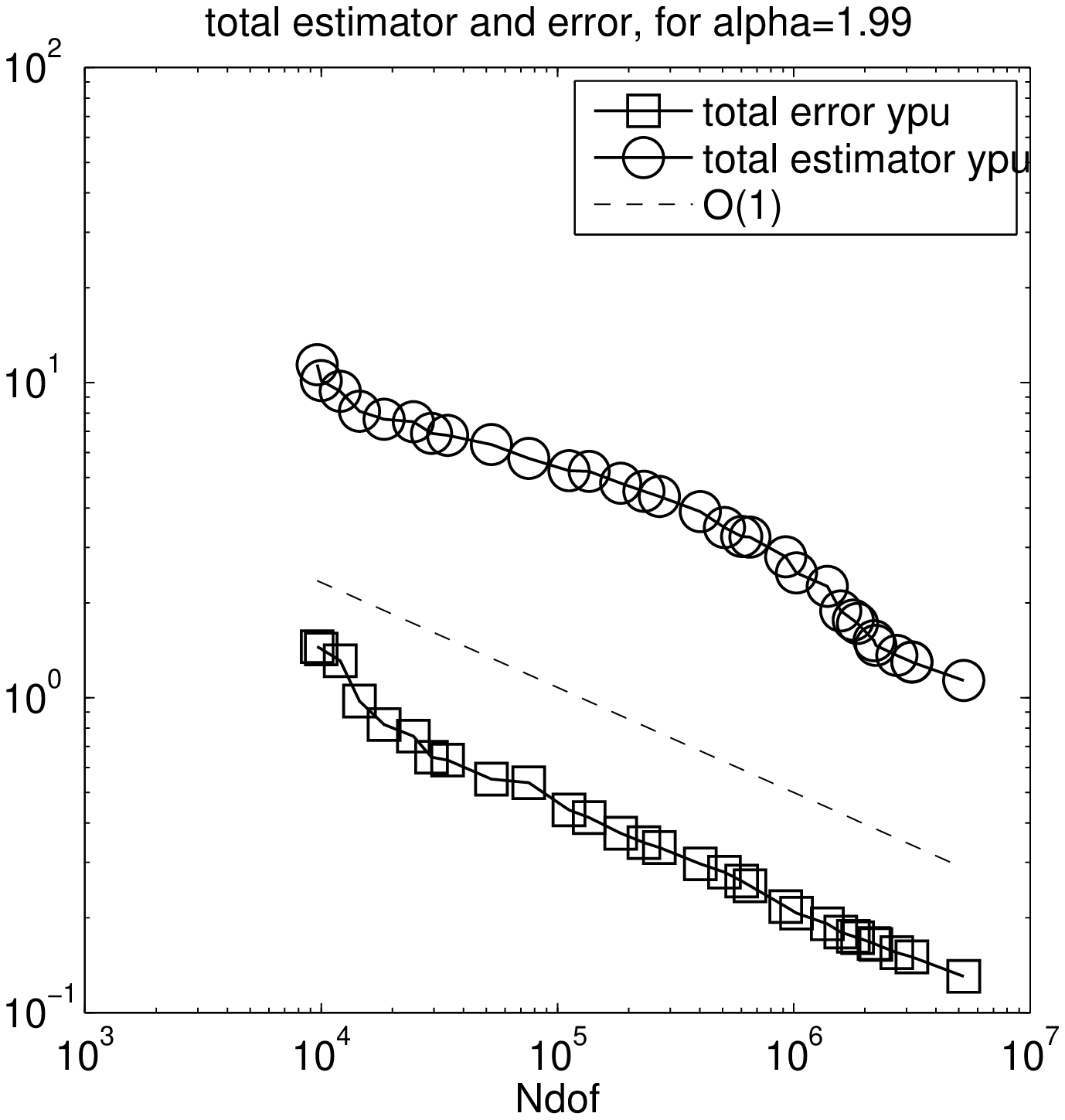}}
\psfrag{error contributions, for alpha=1.99}{\huge Error contributions, for $\alpha=1.99$}
\psfrag{Ndof}{\huge Ndof}
\psfrag{norm of state error}{\LARGE $\|e_\ysf\|_{L^{\infty}(\Omega)}$}
\psfrag{norm of adjoint error}{\LARGE $\|\nabla e_\psf\|_{L^{2}(\rho,\Omega)}$}
\psfrag{norm of control error}{\LARGE $\|e_\usf\|_{L^{2}(\Omega)}$}
\psfrag{O(1)}{\LARGE Ndof$^{-1/3}$}
\scalebox{0.23}{\includegraphics{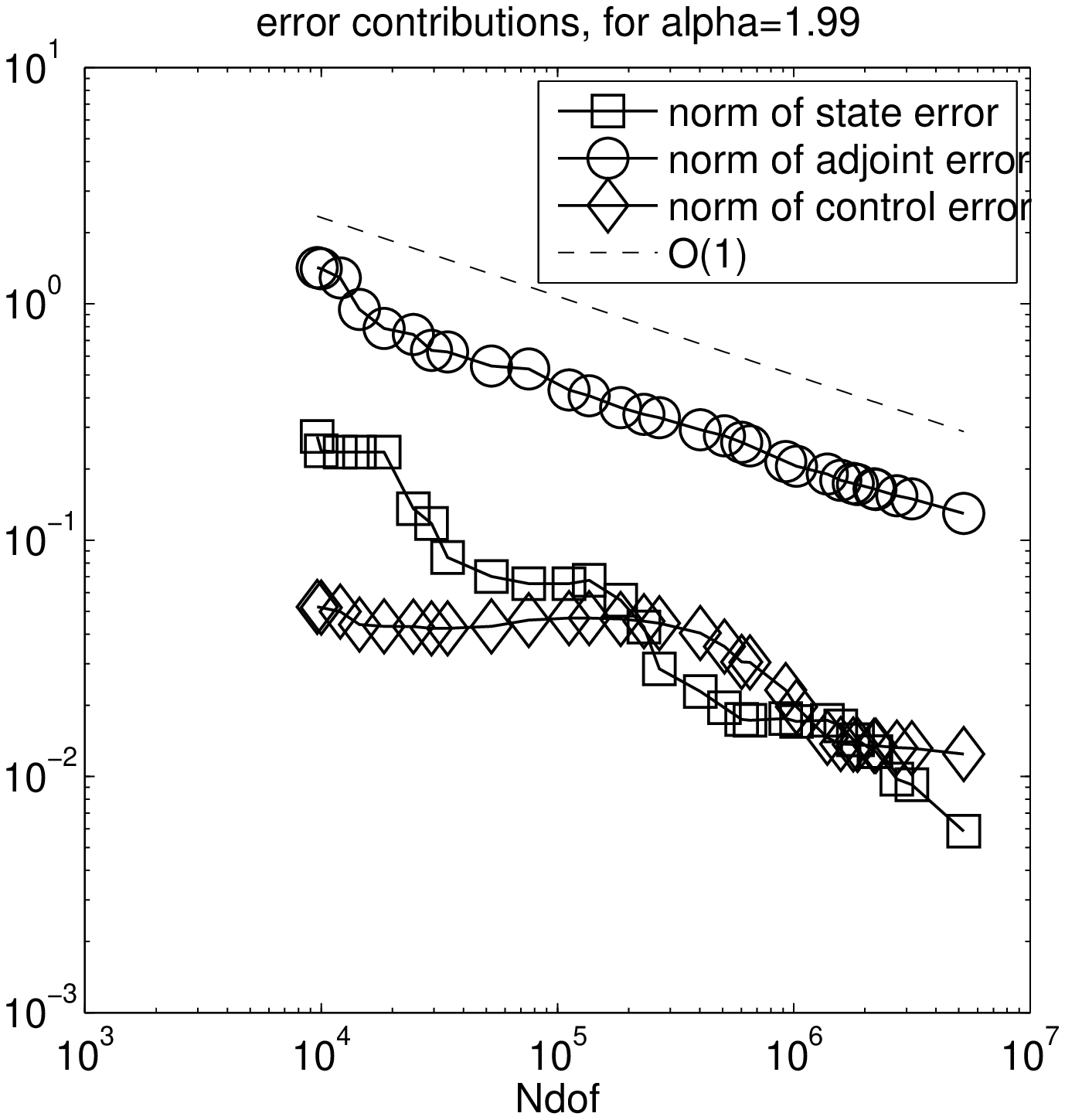}}
\psfrag{estimator contributions, for alpha=1.99}{\huge Estimator contributions, for $\alpha=1.99$}
\psfrag{Ndof}{\huge Ndof}
\psfrag{state estimator}{\LARGE $\E_{\mathsf{y}}$}
\psfrag{adjoint estimator}{\LARGE $\E_{\mathsf{p}}$}
\psfrag{control estimator}{\LARGE $\E_{\mathsf{u}}$}
\psfrag{O(1)}{\LARGE Ndof$^{-1/3}$}
\scalebox{0.23}{\includegraphics{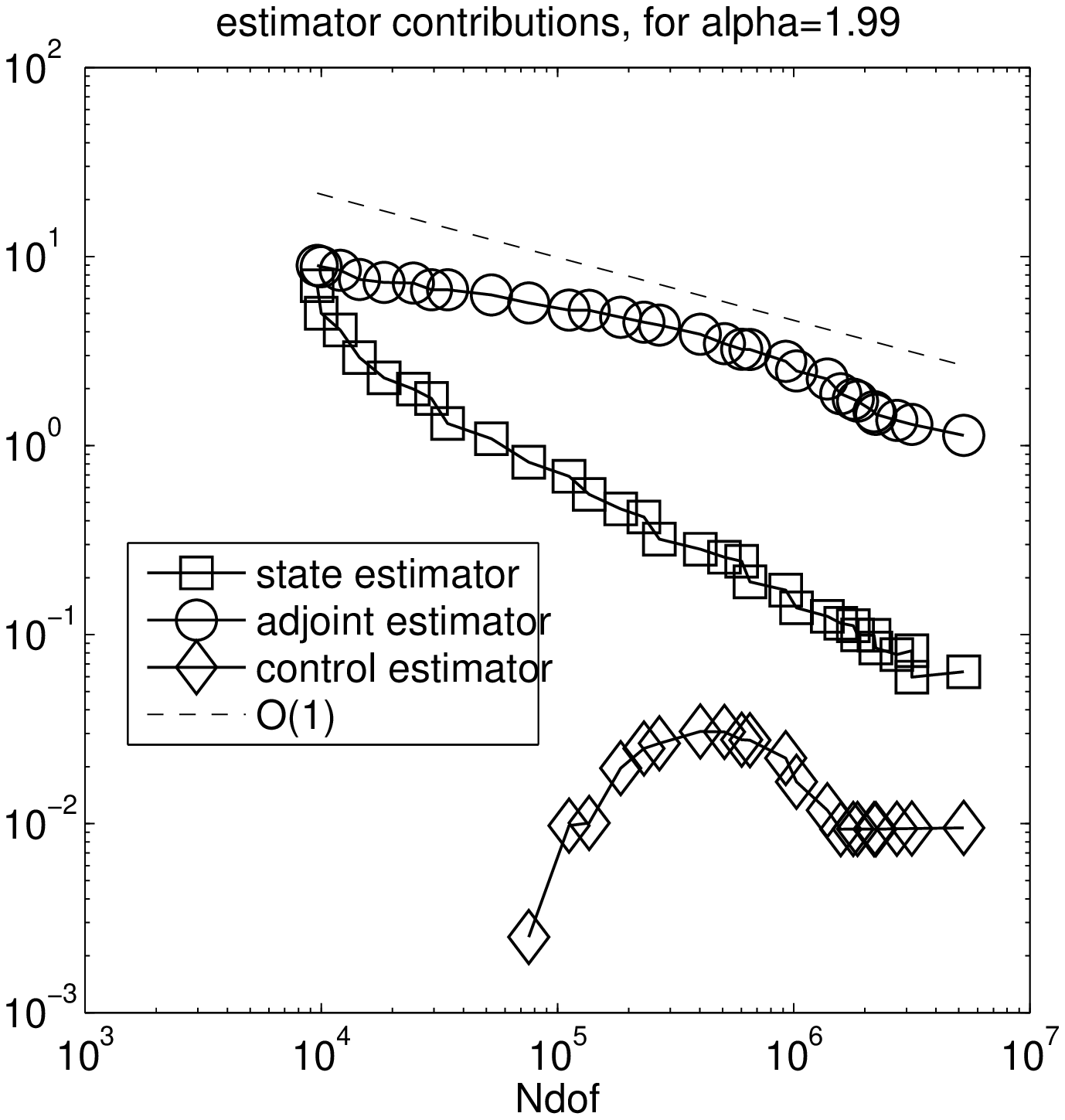}}
\psfrag{total effectivity, for alpha=1.99}{\huge Effectivity index, for $\alpha=1.99$}
\psfrag{Ndof}{\huge Ndof}
\psfrag{total effectivity index ede}{\LARGE $\E_{\mathsf{ocp}}/\|(e_{\ysf},e_{\psf},e_{\usf})\|_{\Omega}$}
\psfrag{O(1)}{\LARGE Ndof$^{-1/3}$}
\scalebox{0.23}{\includegraphics{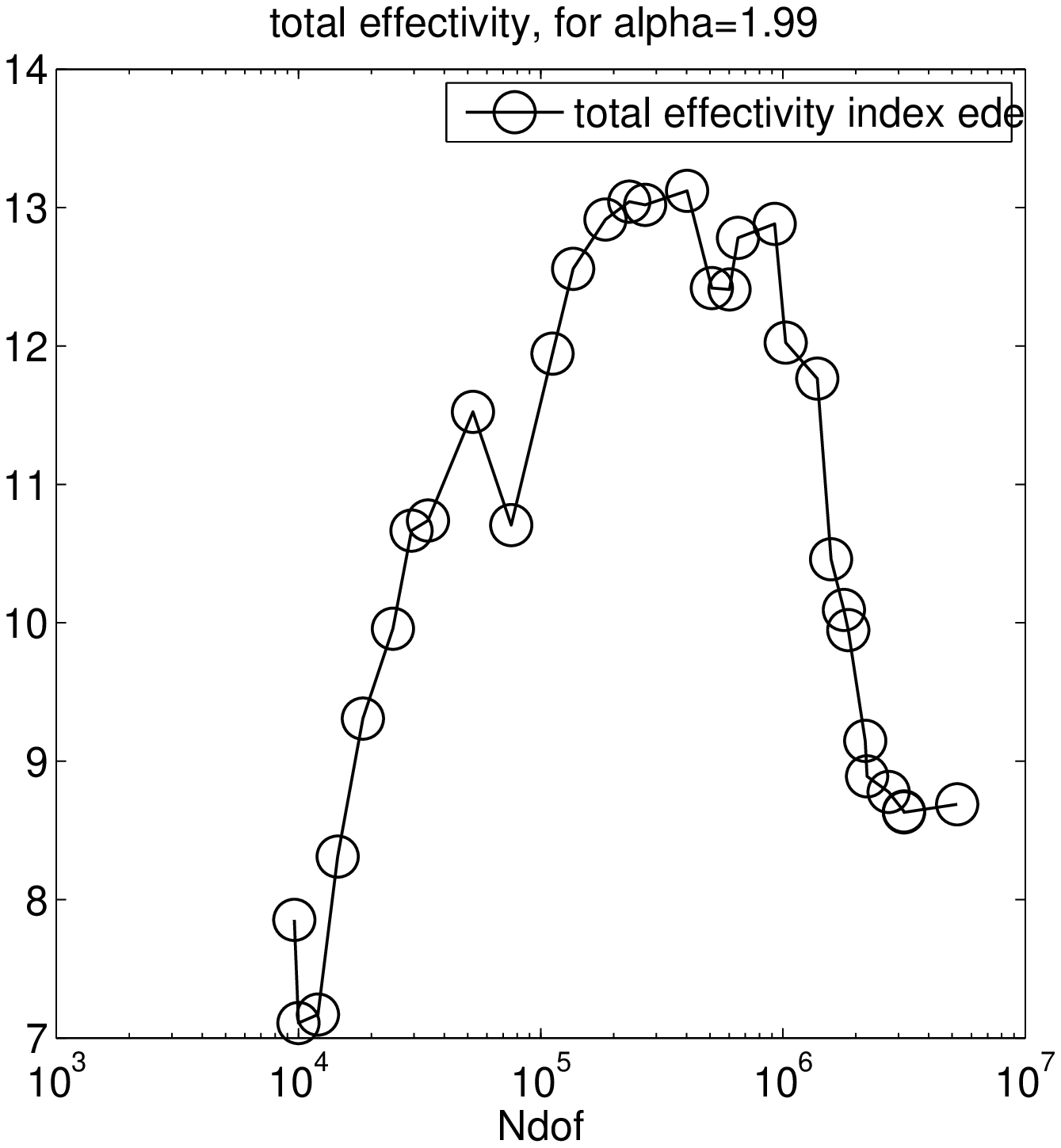}}
\scalebox{0.08}{\includegraphics{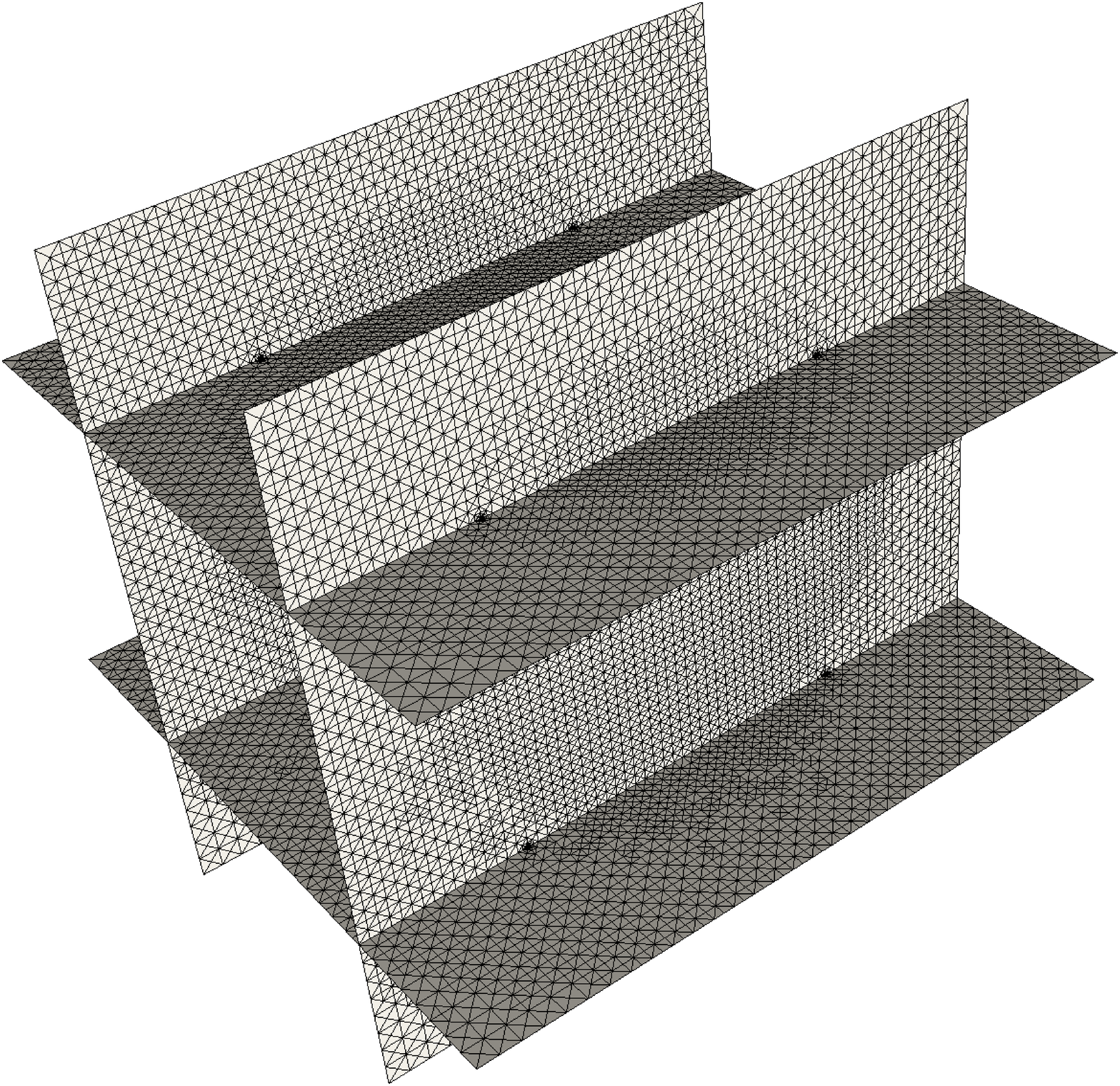}}
\end{center}
\caption{\RR{Example 7: For $\alpha = 1.99$ and adaptive refinement using a maximum strategy, we show the convergence rates for the total error $\|(e_{\ysf},e_{\psf},e_{\usf})\|_{\Omega}$, error estimator $\E_{\mathsf{ocp}}$, their individual contributions, effectivity indices and slices of the final adaptively refined mesh.}}
\label{Fig:Ex7}
\end{figure}
\subsection{Conclusions}
From the presented numerical examples several general conclusions can be drawn:
\begin{enumerate}[$\bullet$]
  \item Most of the refinement occurs near the observation points, which attests to the efficiency of the devised estimators.
  
  \item A larger value of $\alpha$ together with the maximum refinement strategy delivers the best results. This might be due to the fact that a careful examination of the derivation of \eqref{eq:efficiencyglobal} reveals that the constant in this inequality depends on $h_{\T}^{\frac\alpha2+1-\frac{n}2}$ so that, the larger $\alpha$ the smaller its value.
  
  \item The contribution $\E_{\psf}(\bar{\psf}_{\T},\bar{\ysf}_{\T},\T)$ to the error estimator is most of the time the dominating one. We believe that this shows the very singular nature of the problem that defines the adjoint variable. Nevertheless, our estimator is able to deliver optimal rates of convergence.
\end{enumerate}

\section*{Acknowledgement}
The authors would like to thank Harbir Antil (George Mason University) for insightful discussions.

\bibliographystyle{plain}
\bibliography{biblio}

\end{document}